\documentclass[leqno,11pt,a4paper]{amsart}

\usepackage{jc}

\begin{document}

\author{Julian Chaidez}
\address{Department of Mathematics\\University of Southern California\\Los Angeles, CA\\90007\\USA}
\email{julian.chaidez@usc.edu}

\title[Robustly Non-Convex Hypersurfaces In Contact Manifolds]{Robustly Non-Convex Hypersurfaces In Contact Manifolds}

\begin{abstract} We construct the first examples of hypersurfaces in any contact manifold of dimension 5 and larger that cannot be $C^2$-approximated by convex hypersurfaces, contrasting sharply with the foundational results of Giroux in dimension $3$ and Honda-Huang in the $C^0$ case. The main technical step is the first construction of a dynamical blender in the contact setting. 
\end{abstract}

\vspace*{-20pt}

\maketitle

\vspace*{-20pt}

\section{Introduction} \label{sec:introduction}

A hypersurface $\Sigma$ in a contact manifold $(Y,\xi)$ is \emph{convex} if there is a contact vector-field $V$ that is transverse to $\Sigma$. Convex surface theory was first introduced by Giroux \cite{g1991}, and has since proven to be a deep and powerful tool for the study of contact $3$-manifolds. Applications include classifications of contact structures \cite{y1997,h1999,h2000,hkm2003,gls2006,cm2020,t2000,mn2023} and Legendrians \cite{fm2023,hkm2002,emm2022,cemm2023}; the construction of $3$-manifolds with no tight contact structures \cite{e2001} and tight, non-fillable contact structures \cite{eh2002}; and finiteness results for tight contact structures \cite{cgh2003,cgh2009}. There have also been many fruitful interactions with Floer homology \cite{cghh2011,a2023,az2024,fh2021,bs2016,hkm2008}

\vspace{3pt}

Recently, the study of higher dimensional convex surface theory was initiated by Honda-Huang \cite{hh2018,hh2019} and Breen-Honda-Huang \cite{bhh2023}, who have systematically generalized many of the foundational results of Giroux to dimensions larger than three. One such result is the following.

\begin{thm*}[Giroux] \label{thm:giroux_theorem} Any closed surface $\Sigma$ in a contact manifold $(Y,\xi)$ in dimension $3$ can be $C^\infty$-approximated by a convex surface $\Sigma'$.
\end{thm*}

\noindent In \cite{hh2019} (and in Eliashberg-Pancholi \cite{ep2022}) the following partial generalization was proven.

\begin{thm*}[Honda-Huang] \label{thm:honda_huang} Any closed hypersurface $\Sigma$ in a contact manifold $(Y,\xi)$ can be $C^0$-approximated by a Weinstein convex hypersurface.
\end{thm*} 

\noindent Any convex hypersurface $\Sigma$ naturally divides into two ideal Liouville manifolds meeting along their boundary, and $\Sigma$ is called \emph{Weinstein} if these Liouville manifolds are Weinstein. 

\vspace{3pt}

Theorem \ref{thm:honda_huang} is both stronger than Theorem \ref{thm:giroux_theorem} due to the Weinstein condition, but also weaker since it only provides $C^0$-approximations. Indeed, Honda-Huang noted in \cite{hh2019} that the precise generalization of Giroux's theorem was left unresolved by their work.

\begin{question*} \cite[Rmk 1.2.4]{hh2019} or \cite[Problem 2.1]{AIMproblems}. \label{qu:smooth_approximation} Can any closed hypersurface $\Sigma$ in a contact manifold $(Y,\xi)$ be $C^\infty$-approximated by a convex hypersurface?
\end{question*}

\begin{remark*} A counter-example to Question \ref{qu:smooth_approximation} was previously proposed by Mori \cite{mori2011,mori2009reeb} but was later proven to be $C^\infty$-approximable by convex hypersurfaces by Breen \cite[Cor 1.8]{b2021}. A different candidate counter-example (with boundary) remains unverified \cite[Rmk 1.9]{b2021}. \end{remark*}

\vspace{3pt}

\noindent The goal of this paper is to resolve this longstanding question by proving the following theorem.

\begin{thm*}[Main Theorem] \label{thm:main} For any $n \ge 2$, there is a closed hypersurface $\Sigma$ in standard contact $\R^{2n+1}$ (and thus in any contact $(2n+1)$-manifold) that cannot be $C^2$-approximated by convex hypersurfaces. 
\end{thm*}

Our proof is a novel combination of disparate methods in contact topology and partially hyperbolic dynamics. Specifically, we construct $C^1$-robustly topologically mixing contactomorphisms as perturbations of time one maps of contact Anosov flows. In the process, we give the first construction of a blender \cite{bonattidiaz1995,bonatti2004dynamics} in contact dynamics. We then use suspension and embedding constructions from contact topology to produce hypersurfaces in standard contact $\R^{2n+1}$ with $C^2$-robustly transitive characteristic foliations. The characteristic foliation of a convex surface cannot be topologically transitive, so this yields Theorem \ref{thm:main}. Based on the approach of this paper, we propose several new conjectures in convex hypersurface theory (see Section \ref{subsec:questions}).

\newpage

\subsection{Characteristic Foliations} Let us briefly recall some dynamical aspects of the structure of hypersurfaces in contact manifolds, before discussing the key results in our proof.

\vspace{3pt}

Let $\Sigma$ be any hypersurface in a contact manifold $(Y,\xi)$. Recall that the \emph{characteristic foliation} $\Sigma_\xi$ of $\Sigma$ is the (generally singular) oriented, 1-dimensional foliation given by
\[\Sigma_\xi = (T\Sigma \cap \xi)^\omega \subset T\Sigma\]
Here $V^\omega$ denotes the symplectic perpendicular of a subspace $V \subset \xi$ with respect to the conformal symplectic structure on $\xi$. Relatedly, a \emph{characteristic vector-field} $Z$ of $\Sigma$ is a vector-field that spans $\Sigma_\xi$ everywhere and that generates the given orientation. 

\vspace{3pt}

Any convex surface can be divided into two regions that act as a source and a sink for the flow of any characteristic vector-field for the characteristic foliation.

\begin{definition*}[Dividing Set] Let $\Sigma$ be a convex surface in a contact manifold $(Y,\xi)$ with transverse contact vector-field $V$. The \emph{dividing set} $\Gamma$ with respect to $V$ is given by
\[\Gamma = H^{-1}(0) \cap \Sigma \quad \text{where }H = \alpha(V) \text{ for any contact form $\alpha$ for $\xi$}\]
The dividing set $\Gamma$ is the intersection of the \emph{negative region} $\Sigma_-$ and \emph{positive region} $\Sigma_+$ given by the inverse images
$H^{-1}(-\infty,0] \cap \Sigma$ and $H^{-1}[0,\infty) \cap \Sigma$ respectively. \end{definition*}

The dividing set $\Gamma$ is always a transversely cutout hypersurface in $\Sigma$ with a natural contact structure $T\Gamma \cap \xi$ \cite{hh2019}. The two regions $\Sigma_+$ and $\Sigma_-$ are ideal Liouville domains with Liouville forms given by the restriction of $\alpha$, and the characteristic foliation is given by the span of the corresponding Liouville vector-field. Thus $\Sigma$ is equipped with a folded symplectic structure \cite{gsw2000,b2024}. Moreover, this structure is independent of the choices of $V$ and $\alpha$ up to isotopy, and is thus canonical up to deformation.

\begin{figure}[h!]
    \centering
    \includegraphics[width=.7\textwidth]{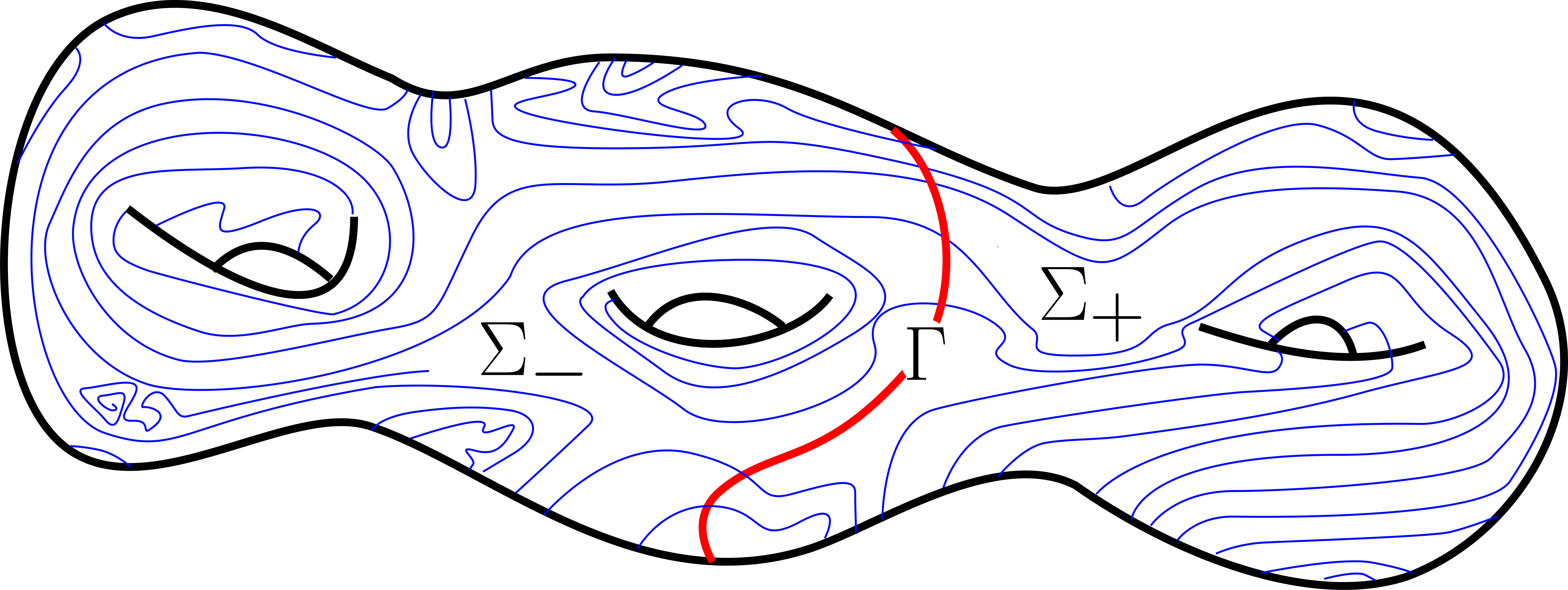}
    \caption{A convex surface with dividing set and characteristic foliation.}
    \label{fig:chacteristic_foliation}
\end{figure}

The characteristic foliation $\Sigma_\xi$ of a convex surface $\Sigma$ is always transverse to the dividing set $\Gamma$, pointing out of $\Sigma_+$ and into $\Sigma_-$, and this has some significant implications for the dynamics. For example, we recall the following definition.

\begin{definition*}[Transitive/Mixing] A smooth flow $\Phi:\R \times \Sigma \to \Sigma$ on a closed manifold $\Sigma$ is \emph{topologically transitive} if, for any non-empty open subsets $U,V \subset \Sigma$, there is a time $T$ such that
\begin{equation} \label{eq:transitivity}
\Phi_T(U) \cap V \neq \emptyset
\end{equation}
Additionally, $\Phi$ is \emph{topologically mixing} if there is a $S$ depending on $U$ and $V$ so that (\ref{eq:transitivity}) holds for all $T > S$. We adopt the analogous definitions for a diffeomorphism $\Phi:Y \to Y$ of a manifold $Y$.\end{definition*}

\begin{lemma*} \label{lem:convex_non_transitive} Let $\Sigma$ be a closed convex hypersurface in a contact manifold $(Y,\xi)$ and let $Z$ be a characteristic vector field for $\Sigma$. Then $Z$ is not topologically transitive (and in particular, not topologically mixing).
\end{lemma*}

\begin{proof} Choose a dividing set $\Gamma$ on $\Sigma$ and a pair of disjoint open subsets $A$ and $B$ of $\Gamma$. Let $\Phi$ be the flow of a characteristic vector-field $Z$ and consider the  subsets
\[
U = \Phi(\R \times A) \quad\text{and}\quad V = \Phi(\R \times B) 
\]
Any flowline of the characteristic foliation $\Sigma_\xi$ intersects $\Gamma$ at most once, so $U$ and $V$ are open and disjoint. Moreover, $\Phi_t(U) \cap V = U \cap V = \emptyset$ for all $t$. Thus $\Phi$ is not topologically transitive. \end{proof}

Our goal (in view of Lemma \ref{lem:convex_non_transitive}) is now to construct examples of hypersurfaces that are robustly topologically transitive in the following sense.

\begin{definition*} \label{def:robustly_transitive} A smooth flow $\Phi:\R \times \Sigma \to \Sigma$ is \emph{robustly transitive} if there is a $C^1$-open set
\[\mathcal{U} \subset \on{Flow}(\Sigma) \quad\text{with}\quad \Phi \in \mathcal{U}\]
such that any $\Psi \in \mathcal{U}$ is topologically transitive. Note that we adopt the analogous definition of a \emph{robustly  mixing} flow $\Phi$, as well for a \emph{robustly transitive (or mixing)} diffeomorphism $\Phi:Y \to Y$.
\end{definition*}

\subsection{Suspension} \label{subsec:intro_suspension} The first step towards this goal is to reduce the problem to a question about contactomorphisms. Fix a contactomorphism $\Phi$ of contact manifold $(Y,\xi)$ to itself. We may take the suspension (or mapping torus) to get an even dimensional space
\[\Sigma(\Phi) = [0,1]_s \times Y/\sim \qquad\text{with a hyperplane field }\eta = \on{span}(\partial_s) \oplus \xi\]
The hyperplane field $\eta$ is an example of an even contact structure, and so $\Sigma(\Phi)$ has the structure of a contact Hamiltonian manifold (also referred to as an even contact manifold \cite{bm2021}). Any such space has a natural characteristic foliation, which in this case is given by
\[
\Sigma(\Phi)_\xi = \on{span}(\partial_s)
\]
In particular, the flow of the characteristic foliation is simply the suspension flow of the contactomorphism $\Phi$. One may also take the contactization
\[
\R_s \times \Sigma(\Phi)
\]
such that the characteristic foliation on $0 \times \Sigma(\Phi)$ (as a hypersurface within the contactization) agrees with the intrinsic characteristic foliation $\Sigma(\Phi)_\xi$. 

\begin{figure}[h!]
    \centering
    \includegraphics[width=.4\textwidth]{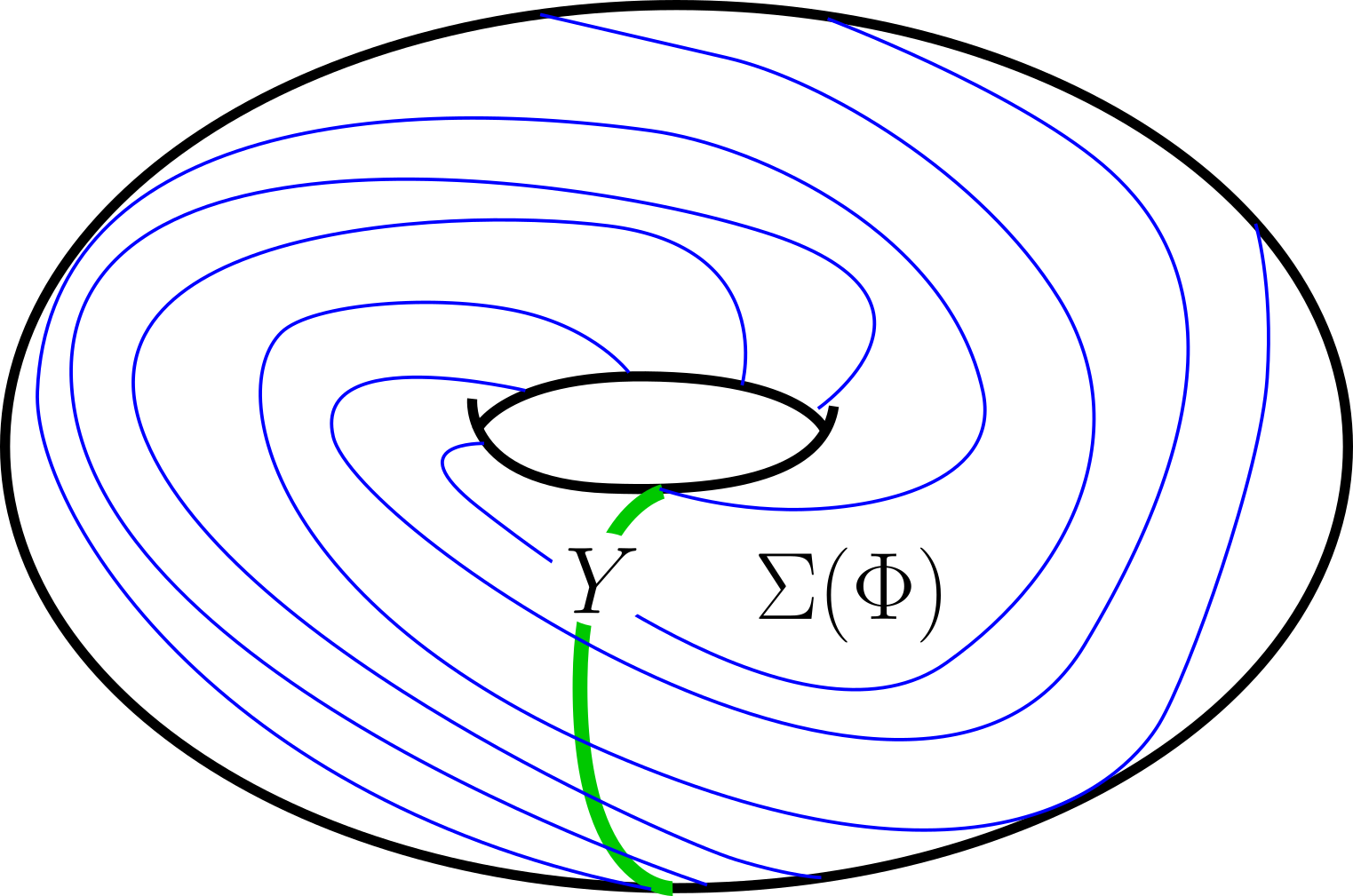}
    \caption{The suspension of a half rotation of the circle $Y = S^1$.}
    \label{fig:suspension}
\end{figure}

\vspace{-5pt}

As an example, we have depicted the suspension $\Sigma(\Phi)$ of a half rotation $\Phi:S^1 \to S^1$ of the circle in Figure \ref{fig:suspension}. This is an example of a contactomorphism that gives rise to a non-convex surface in the contactization of its suspension, since it has no dividing curves. It is a nice exercise for the reader to find a perturbation of this foliation that satisfies Giroux's convexity criterion \cite{g1991,b2021}.

\vspace{3pt}

In Section \ref{sec:contact_hamiltonian_manifolds}, we will discuss contact Hamiltonian manifolds and the suspension construction in detail. As an application of this construction, we will prove the following result.

\begin{thm*}[Proposition \ref{prop:main_suspension}] \label{thm:main_suspension} Let $\Phi:Y \to Y$ be a robustly transitive contactomorphism. Then $\Sigma(\Phi)$ has a $C^2$-open neighborhood $\mathcal{U}$ such that every $\Sigma' \in \mathcal{U}$ has topologically transitive characteristic foliation. \end{thm*}

\subsection{Robustly Mixing Contactomorphisms And Blenders} The next step is the construction of robustly mixing contactomorphisms. This is the subject of our most difficult theorem.

\begin{thm*}[Theorem \ref{thm:body_anosov}] \label{thm:main_anosov} Let $(Y,\xi)$ be a closed contact manifold admitting an Anosov Reeb flow $\Phi$ with $C^\infty$ stable and unstable foliations.  Then the $C^1$-open set of robustly mixing contactomorphisms
\[
\on{Cont}_{\on{RM}}(Y,\xi) \subset \on{Cont}(Y,\xi)
\]
is non-empty. More precisely, if $T$ is a (non-zero) multiple of the period of a closed Reeb orbit of $\Phi$, then $\Phi_T$ is in the $C^\infty$-closure of the set of robustly mixing contactomorphisms.
\end{thm*}

\noindent Note that the contactomorphisms constructed in Theorem \ref{thm:main_anosov} are also topologically transitive, since topological mixing is a stronger property than topological transitivity. Before discussing the proof, we briefly note that there are examples where Theorem \ref{thm:main_anosov} applies.

\begin{example*} \label{ex:hyperbolic_manifolds} Let $X$ be a closed $n$-manifold with a hyperbolic metric $g$. Then the geodesic flow
\[\Phi:\R \times SX \to SX\]
on the unit cosphere bundle $SX$ is Anosov with smooth stable and unstable foliations. These foliations are precisely the quotients of the foliations by the positive and negative unit conormal bundles of the horospheres in $\mathbb{H}^n$, respectively. Specific constructions of closed hyperbolic manifolds in any dimension $2$ or greater can be found in \cite{b1963,m1976,g1987}.  \end{example*}

Theorem \ref{thm:main_anosov} is version of a seminal theorem of Bonatti-Diaz \cite[Thm A]{bonattidiaz1995} in the contact setting, and our main task is to construct a certain dynamical structure introduced in \cite{bonattidiaz1995} called a blender. Roughly, a blender is a robust, horseshoe-type structure that forces certain invariant manifolds to be larger than expected. Since their introduction in \cite{bonattidiaz1995}, blenders have become an essential tool in the study of robust properties of smooth dynamical systems. We refer the reader to \cite{whatisblender,bonatti2004dynamics} for a survey on this topic and \cite{bdp2003,bd2008,l2024,np2012} for just a few of their applications. 

\vspace{3pt}

Our construction is the first instance of a blender in contact dynamics, and constitutes the first application of this fundamental tool to a longstanding problem in either symplectic or contact topology. The construction departs in several places from \cite{bonattidiaz1995} and is unrelated to the symplectic blender construction in the (conservative) Hamiltonian setting due to Nassiri-Pujals \cite{np2012}. Indeed, the work \cite{np2012} has been applied solely to the study problems from dynamics (e.g. related to Arnold diffusion) and produces robust transitivity in certain proper hyperbolic invariant sets fibering over closed symplectic manifolds, which is not useful for our purposes.

\vspace{3pt}

The contact blender construction proceeds by perturbing the time one map of an Anosov Reeb flow $\Phi$ along a period one closed orbit, to produce a new map $\Psi$ possessing a pair of hyperbolic fixed points $P$ and $Q$ of coindex one, with a heteroclinic cycle connecting them (Section \ref{sec:blender_construction}). We then carefully verify our variant of the blender axioms (Definition \ref{def:blender}) in Section \ref{sec:proof_of_blender_axioms}. 

\vspace{3pt}

Constructions in contact dynamics tend to encounter difficulties that are not present in the smooth setting. For example, it is not possible to cut off a contact vector field without introducing additional dynamics near the boundary of the cutoff, and this makes constructing dynamical plugs difficult. This issue is manifest in Honda-Huang \cite{hh2019} and Eliashberg-Pancholi \cite{ep2022}. Our construction of the contact blender is also subject to such difficulties not present in \cite{bonattidiaz1995}. Of particular note, there are additional error terms appearing in the perturbed map $\Psi$ that could ruin the blender properties in general. A key insight is that these error terms can be effectively controlled by choosing the contact Hamiltonians in our perturbation carefully with respect to a well-behaved Weinstein chart, so that the errors are tangent to a certain smoothly integrable Legendrian foliation that is uniformly contracted by $\Psi$. The existence of this foliation is guaranteed in the setting of Theorem \ref{thm:main_anosov} by the hypothesis that the Anosov Reeb flow $\Phi$ has $C^\infty$ stable and unstable foliations. This is discussed further in Remark \ref{rmk:error_terms}.

%\begin{remark} The hypothesis of smooth stable and unstable foliations essentially restricts us to Example \ref{ex:hyperbolic_manifolds}. It is interesting to ask if this hypothesis can be removed.\end{remark}

\begin{figure}[h!]
    \centering
    \includegraphics[width=.8\textwidth]{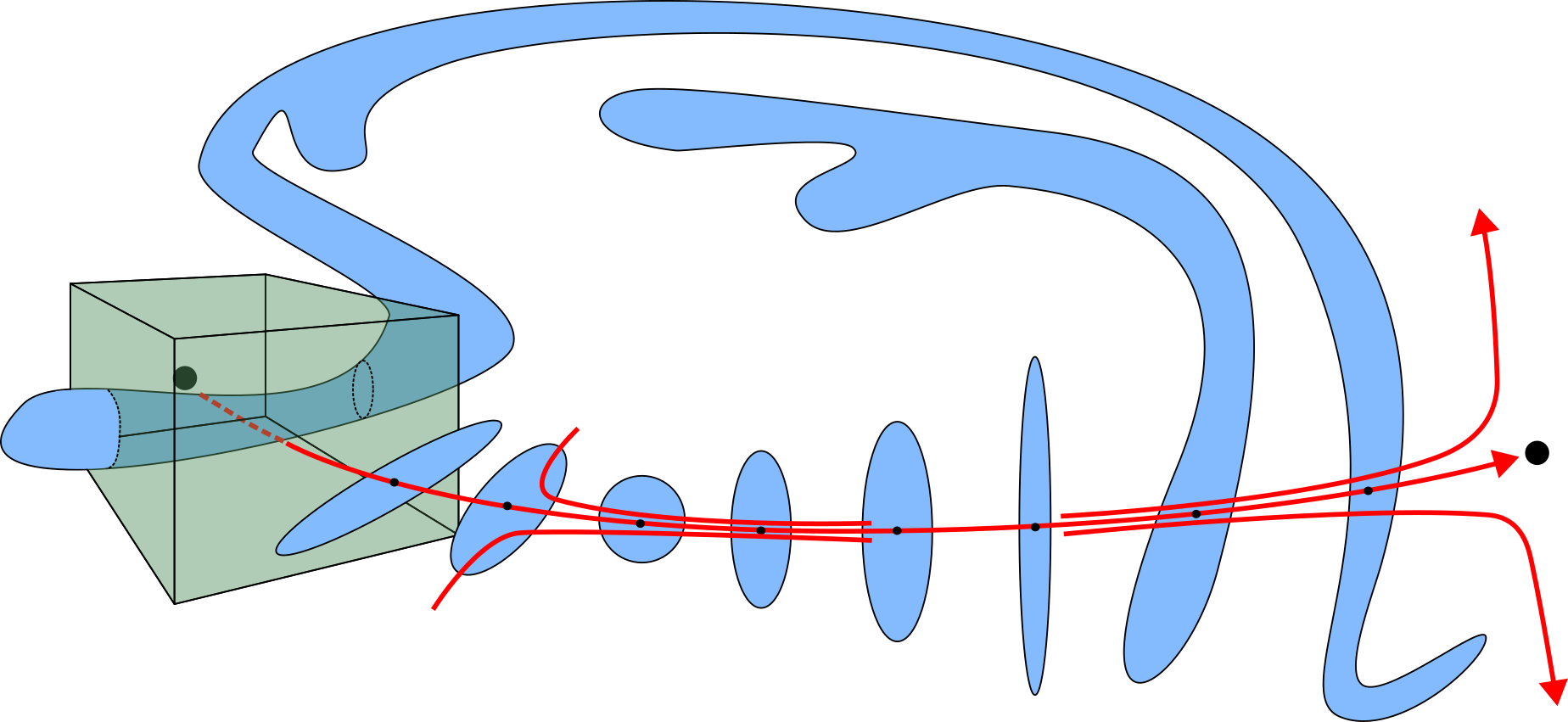}
    \caption{A cartoon of a blender, created by the interaction of stable and unstable manifolds of two fixed points of coindex one. }
    \label{fig:blender}
\end{figure}

\vspace{-5pt}

\subsection{Non-Convex Hypersurfaces} We are now ready to combine the results discussed thus far to prove the main theorem. Already, we can use Lemma \ref{lem:convex_non_transitive}, Theorems \ref{thm:main_suspension}-\ref{thm:main_anosov} and Example \ref{ex:hyperbolic_manifolds} to immediately acquire a specific case of our main theorem.

\begin{thm*} The cosphere bundle $SX$ of a closed hyperbolic manifold $X$ has a contactomorphism $\Phi:X \to X$ such that the suspension
\[\Sigma(\Phi) \subset \R \times \Sigma(\Phi)\]
cannot be $C^2$-approximated by a convex hypersurface. \end{thm*}

In order to enhance this result to acquire the more general Theorem \ref{thm:main}, we apply two difficult theorems. First, we have the following theorem of Sullivan.

\begin{thm*} \cite{s1979} \label{thm:intro_parallelizable} Every closed hyperbolic manifold $W$ has a finite cover $X$ that is stably parallelizable.
\end{thm*}

\noindent We also have the following existence theorem for Legendrian embeddings. This follows from the h-principle of Murphy \cite{m2012}, although this  case follows from the earlier h-principle of Gromov.

\begin{thm*} \cite{m2012} \label{thm:intro_legendrian} Any closed, stably parallelizable manifold $X$ has a Legendrian embedding $X \to \R^{2n+1}$.
\end{thm*}

\noindent Finally, we need the following lemma that will be proven in Section \ref{sec:contact_hamiltonian_manifolds}. Recall that a contactomorphism is \emph{positive} if it is the time $1$ map of a (possibly time-dependent) contact Hamiltonian that is positive as a smooth function.

\begin{lemma*}[Lemma \ref{lem:legendrian_embeddings}] \label{lem:intro_embedding}  Let $\Lambda \subset (Y,\xi)$ be a closed Legendrian and let $\Phi:S\Lambda \to S\Lambda$ be a positive contactomorphism of the cosphere bundle $S\Lambda$. Then there exists a contact embedding
\[(-\epsilon,\epsilon) \times \Sigma(\Phi) \to (Y,\xi) \qquad\text{for small $\epsilon$}\]
\end{lemma*}

\noindent With these preliminary results in hand, we can now proceed with the proof of the main result.

\begin{thm*}[Theorem \ref{thm:main}] There is a closed, embedded hypersurface $\Sigma \subset \R^{2n+1}$ for any $n \ge 2$ that cannot be $C^2$-approximated by convex hypersurfaces.
\end{thm*}

\begin{proof} Take a stably parallelizable closed hyperbolic manifold $X$ (via Theorem \ref{thm:intro_parallelizable} and Example \ref{ex:hyperbolic_manifolds}) with a Legendrian embedding $X \to \R^{2n+1}$ (via Theorem \ref{thm:intro_legendrian}). By Theorem \ref{thm:main_anosov}, there is a robustly mixing contactomorphism
\[\Phi:SX \to SX\]
that is $C^\infty$-close to the time $T$ Reeb flow, where $T$ is the period of a closed Reeb orbit. Since the time $T$ Reeb flow is a positive contactomorphism, $\Phi$ is also positive.
By Lemma \ref{lem:intro_embedding}, there is a contact embedding $U \to \R^{2n+1}$ of a neighborhood $U \subset \R \times \Sigma(\Phi)$ of the suspension $\Sigma(\Phi)$ in its contactization. By Theorem \ref{thm:main_suspension}, $\Sigma(\Phi)$ cannot be $C^2$-approximated by a convex surface. 
\end{proof}

\subsection{Conjectures} \label{subsec:questions} This paper establishes a new connection between contact topology and smooth (hyperbolic) dynamics. We conclude by this introduction proposing several new conjectures in convex hypersurface theory that build on this connection. We require the following definition.

\begin{definition*}[Robust Non-Convexity] A hypersurface $\Sigma$ in a contact manifold is \emph{robustly non-convex} if there is a $C^2$-neighborhood $\mathcal{U}$ in the space of embedded hypersurfaces such that
\[\Sigma' \text{ is not convex for any }\Sigma' \in \mathcal{U}\]
Similarly, a contactomorphism $\Phi:Y \to Y$ is called \emph{robustly non-convex} if the suspension $\Sigma(\Phi)$ is. \end{definition*} 

\begin{conjecture*}[Cycles] \label{conj:robust_hetero} A closed hypersurface $\Sigma \subset (Y,\xi)$ is robustly non-convex if and only if $\Sigma_\xi$ has a robust heterodimensional cycle between a positive hyperbolic orbit and a negative hyperbolic orbit.
\end{conjecture*}

\noindent Recall that a \emph{robust heterodimensional cycle} of a vector field $V$ is a cycle of heteroclinics that is robust under $C^1$-perturbations. Also recall that any periodic orbit of the characteristic foliation of a hypersurface can be classified as \emph{positive} or \emph{negative} via a divergence criterion (cf. Breen \cite{b2021}). 

\vspace{3pt}

To motivate this, recall that the existence of robust heterodimensional cycles is closely related to the theory of blenders \cite{bonatti2008robust,li2024persistence}. One may check that the contact blender constructed in this paper contains a robust cycle as in Conjecture \ref{conj:robust_hetero}. Moreover, the existence of such a cycle in the characteristic foliation implies the existence of a retrograde heteroclinic from a negative to a positive orbit, which is an obstruction to convexity (cf. Honda-Huang \cite{hh2019} and Breen \cite{b2021}).

\vspace{3pt}

Next, we propose the following improvement of Theorem \ref{thm:main}. This is a strong counterpart to the approximation theorem of Honda-Huang (Theorem \ref{thm:honda_huang}). 

\begin{conjecture*}[Approximation] \label{conj:robust_density} Every closed oriented hypersurface $\Sigma$ in a contact manifold $(Y,\xi)$ of dimensions $5$ or higher is $C^\infty$-isotopic by a $C^0$-small isotopy to a robustly non-convex hypersurface.
\end{conjecture*}

\noindent Conjecture \ref{conj:robust_density} would follow from the backwards direction of Conjecture \ref{conj:robust_hetero} along with a local construction of robust heterodimensional cycles by $C^0$-small perturbations (cf. \cite{li2024persistence} for a smooth analogue). This approach is the subject of ongoing joint investigation with M. Huang \cite{ch2025}. 

\vspace{3pt}

Finally, it is natural to ask if Theorem \ref{thm:main} is sharp. We conjecture that this is indeed the case.

\begin{conjecture*}[$C^1$-Genericity] \label{conj:C1_density} Convex hypersurfaces are $C^1$-generic in any contact manifold.
\end{conjecture*}

\noindent Resolving Conjecture \ref{conj:C1_density} will require entirely new methods from \cite{hh2019,ep2022} since hypersurfaces with gradient-like characteristic foliation are not $C^1$-dense. To further motivate Conjecture \ref{conj:C1_density}, recall that a flow is \emph{Smale} if it satisfies Axiom A and the strong transversality property. Such flows are structurally stable and cannot have heteroclinic cycles. Smale diffeomorphisms are $C^0$-dense by Sullivan-Shub \cite{{shubhomology}} and the same is believed for flows \cite{zeeman1972c0}. A more precise version of Conjecture \ref{conj:C1_density} is that hypersurfaces with Smale characteristic foliation are convex and $C^1$-dense. This conjecture is the subject of ongoing joint investigation with Eliashberg-Pancholi.

\subsection*{Outline} In Section \ref{sec:contact_hamiltonian_manifolds}, we discuss contact Hamiltonian manifolds and the suspension construction. In Section \ref{sec:partial_hyperbolicity}, we review the necessary background from the theory of partially hyperbolic diffeomorphisms and blenders. In Section 
\ref{sec:robustly_mixing_contactomorphisms}, we prove Theorem \ref{thm:main_anosov} assuming the existence of the contact blender. In Sections \ref{sec:blender_construction} and \ref{sec:proof_of_blender_axioms} we conclude with the construction of the contact blender.

\subsection*{Acknowledgements} Question \ref{qu:smooth_approximation} was discussed at the workshop Higher-Dimensional Contact Topology at the American Institute of Mathematics in April 2024. We thank the organizers Roger Casals, Yakov Eliashberg, Ko Honda, and Gordana Mati\'{c} and the AIM staff for a fruitul week. We also thank the members of the conformal symplectic structures room (M\'{e}lanie Bertelson,  Fabio Gironella, Pac\^{o}me Van Overschelde, Kevin Sackel and Lisa Traynor) for our discussion of this problem. Finally, we thank Joseph Breen and Austin Christian for input on earlier drafts, and Kai Cieliebak for pointing out an innacurate claim about the suspension of mixing diffeomorphisms being a mixing, leading to several unjustified claims in an earlier version.

\section{Contact Hamiltonian Manifolds} \label{sec:contact_hamiltonian_manifolds} 

In this brief section, we discuss the theory of contact Hamiltonian manifolds, which are also called even contact manifolds in the terminology of Bertelson-Meigniez \cite{bm2021}. 

\vspace{3pt}

This theory has satisfying parallels with the theory of (stable) Hamiltonian manifolds \cite{ce2015,w2016,c2023}, which motivates our preferred nomenclature. We freely use these two terms as synonyms. 

\subsection{Fundamentals} We start with the basic facts, which mirror the stable Hamiltonian case.

\begin{definition} \label{def:contact_hamiltonian_manifold} A \emph{contact Hamiltonian manifold} $(\Sigma,\eta)$ is a $2n$-manifold $\Sigma$ equipped with a coorientable, maximally non-integrable, plane distribution of codimension one
\[\eta \subset T\Sigma\]
Equivalently, $\eta$ is the kernel of a \emph{contact Hamiltonian form} $\nu$ with $\nu \wedge d\nu^{n-1}$ is nowhere vanishing. 
\end{definition}

\begin{example}[Product] \label{ex:product} Let $(Y,\xi)$ be a contact manifold. Then the manifolds
\[\R \times Y \qquad\text{and}\qquad \R/\Z \times Y \]
are contact Hamiltonian manifolds with distribution $\eta = \partial_t \oplus \xi$, where $t$ is the $\R$-coordinate. \end{example}

\noindent Every contact Hamiltonian manifold has a natural line distribution (or equivalently, foliation).

\begin{definition} \label{def:characteristic_foliation} The \emph{characteristic foliation} $\Sigma_\eta$ of a contact Hamiltonian manifold $(\Sigma,\eta)$ is given by
\[
\Sigma_\eta = \on{ker}(d\nu|_\eta) \subset T\Sigma \qquad\text{for any contact Hamiltonian form $\nu$ for $\eta$}
\]
A \emph{characteristic vector-field} $Z$ is a nowhere vanishing section of $\Sigma_\eta$ and a \emph{framing} form $\theta$ is a $1$-form whose restriction to $\Sigma_\eta$ is nowhere vanishing. A framing form determines a characteristic vector-field by the following equations.
\begin{equation} \label{eq:characteristic_vector-field} \nu(Z) = 0 \qquad (\iota_Zd\nu)|_\eta = 0 \quad\text{and}\quad \theta(Z) = 1\end{equation}
\end{definition}

\begin{lemma} \label{lem:characteristic_flow} Any characteristic vector-field $Z$ on a contact Hamiltonian manifold $(\Sigma,\eta)$ preserves $\eta$.
\end{lemma}

\begin{proof} Let $\Phi$ be the flow of $Z$. Note that $\iota_Z d\nu = f\nu$ for some smooth $f$ since $(\iota_Zd\nu)|_\eta = 0$. Thus
\[
\mathcal{L}_Z\nu = d(\iota_Z\nu) + \iota_Zd\nu = f \nu \qquad\text{and therefore}\qquad \Phi_t^*\nu = g_t \cdot \nu\qedhere
\]\end{proof}

\begin{definition}[Hamiltonian Vector-fields] \label{def:hamiltonian_vector-fields} The \emph{Reeb vector-field} $R$ of a contact Hamiltonian manifold $(\Sigma,\eta)$ with contact Hamiltonian form $\nu$ and framing $\theta$ is the unique vector-field satisfying
\[
\theta(R) = 0 \qquad \nu(R) = 1 \quad\text{and}\quad (\iota_R d\nu)|_\xi = 0 \quad \text{where}\quad \xi = \on{ker}(\theta)
\]
The \emph{Hamiltonian vector-field} $V_H$ of a function $H:\Sigma \to \R$ is the unique vector-field satisfying
\[
\theta(V_H) = 0 \qquad \nu(V_H) = H  \quad\text{and}\quad (\iota_{V_H}d\nu - dH(R) \cdot \nu + dH)|_\xi = 0
\]
Note that the last condition is equivalent to $\mathcal{L}_{V_H}\nu = dH(R) \cdot \nu + g \cdot \theta$ given that $H = \nu(V_H)$. \end{definition}

\subsection{Contact Hamiltonian Hypersurfaces} A natural source of contact Hamiltonian hypersurfaces are (special) hypersurfaces in contact manifolds.

\begin{definition} \label{ex:hypersurfaces} A hypersurface $\Sigma$ in a contact manifold $(Y,\xi)$ is called \emph{contact Hamiltonian} if
\[\xi \text{ is transverse to }T\Sigma\]
A \emph{framing vector-field} $U$ is a vector-field in a neighborhood of $\Sigma$ such that
\[\text{$U$ is transverse to }\Sigma \qquad\text{and}\qquad \text{$U$ is tangent to $\xi$} \]\end{definition}

\begin{lemma} Let $\Sigma$ be a  contact Hamiltonian hypersurface in $(Y,\xi)$ with framing vector-field $U$. Then
\[(\Sigma,\xi \cap T\Sigma) \quad \text{is contact Hamiltonian with framing form}\quad  \theta = \iota_Ud\alpha\]
\end{lemma}

\begin{proof} Fix a contact form $\alpha$ on $Y$. Then $d\alpha$ has a $1$-dimensional kernel on $\on{ker}(\alpha) \cap T\Sigma$ by standard symplectic linear algebra. It follows that $\nu = \alpha|_{T\Sigma}$ is a contact Hamiltonian form. Similarly, since $d\alpha$ is non-degenerate on $\xi$, we must have
\[d\alpha(U,Z) \neq 0 \text{ for any non-vanishing }Z \in \Sigma_\eta \qedhere\]
\end{proof}

\noindent Every contact Hamiltonian manifold arises as a hypersurface in its own contactization.

\begin{definition} The \emph{contactization} $(C\Sigma,\alpha)$ of a closed contact Hamiltonian manifold $(\Sigma,\eta)$ with contact Hamiltonian form $\nu$ and framing form $\theta$ is given by
\[
C\Sigma = (-\epsilon,\epsilon)_s \times \Sigma \qquad\text{with contact form }\alpha = s\theta + \nu
\]
\end{definition}

\begin{lemma}The contactization $(C\Sigma,\alpha)$ is a contact manifold for $\epsilon$ small, and $\Sigma$ naturally embeds as a contact Hamiltonian hypersurface
\[\Sigma = 0 \times \Sigma \subset C\Sigma \qquad\text{with}\qquad \nu = \alpha|_\Sigma\]\end{lemma}

\begin{proof} Note that $\nu \wedge d\nu^{n-1}$ is nowhere vanishing and the characteristic vector-field $Z$ of $\theta$ is a nowhere vanishing vector-field that satisfies
\[\iota_Z(\nu \wedge d\nu^{n-1}) = 0\]
A $1$-form $\mu$ on $\Sigma$ thus satisfies $\mu \wedge \nu \wedge d\nu^{n-1} \neq 0$ if and only if $\mu(Z) \neq 0$ everywhere. Therefore
\[\mu \wedge \nu \wedge d\nu^{n-1} \qquad\text{and}\qquad \nu \wedge ds \wedge \theta  \wedge d\nu^{n-1}\]
are volume forms on $\Sigma$ and $C\Sigma$ respectively. The second volume form above agrees with $\alpha \wedge d\alpha^n$ along $0 \times \Sigma$, so there is a neighborhood of $0 \times \Sigma$ where $\alpha \wedge d\alpha^n$ is a volume form.
\end{proof}

\noindent There is a natural way to deform a contact Hamiltonian manifold as a graph in its own contactization (c.f. \cite{ct2023} for a stable Hamiltonian analogue).

\begin{definition}[Deformation] \label{def:deformation} The \emph{deformation} $(\Sigma,\eta_H)$ of the contact Hamiltonian manifold $(\Sigma,\eta)$ by the Hamiltonian $H:\Sigma \to (-\epsilon,\epsilon)$ is given by
\[\eta_H = \on{ker}(\nu_H) \qquad\text{where}\qquad \nu_H = H \cdot \theta + \nu\]
This is precisely the pullback of the induced contact Hamiltonian structure on the graph
\[
\on{Gr} H \subset C\Sigma \qquad\text{given by}\qquad \on{Gr} H = \big\{(H(x),x) \; : \; x \in  \Sigma\big\} 
\]\end{definition}

Finally, we note that the contactization provides a local model for the neighborhood of any contact Hamiltonian hypersurface. Specifically, we have the following (strict) standard neighborhood lemma. 

\begin{lemma}[Collar Neighborhood] \label{lem:collar_neighborhood} Let $\Sigma$ be a contact Hamiltonian hypersurface in a contact manifold $(Y,\xi)$. Fix a contact form $\alpha$ on $Y$ and a framing vector-field $U$ of $\Sigma$ such that
\[
\iota_Ud\alpha \text{ is closed}
\]
Then the flow by $U$ yields a strict contact embedding
\[
\iota:(-\epsilon,\epsilon) \times \Sigma \to Y \qquad\text{with}\qquad \iota^*\alpha = s \cdot \theta + \nu \quad \text{where} \quad\text{$\nu = \alpha|_\Sigma$ and $\theta = \iota_Ud\alpha|_\Sigma$}
\]
\end{lemma}

\begin{proof} First note that we have the following calculation.
\[\mathcal{L}_U(\iota_Ud\alpha) = d(d\alpha(U,U)) + d(\iota_Ud\alpha) = 0\]
Now let $\iota:(-\epsilon,\epsilon)_s \times \Sigma \to Y$ be the tubular neighborhood coordinates of $\Sigma$ induced by $U$. Then the previous calculation and the fact that $\iota_U\theta = 0$ shows that the $1$-form
\[
\theta = \iota^*(\iota_Ud\alpha) = \iota_{\partial_s}d(\iota^*\alpha) \qquad\text{satisfies}\qquad \mathcal{L}_{\partial_s}\theta = 0 \text{ and }\theta(\partial_s) = 0
\]
Thus $\theta$ is the pullback of a differential form on $\Sigma$ to $(-\epsilon,\epsilon) \times \Sigma$. Moreover, we see that
\[
\mathcal{L}_U\alpha = d(\iota_U\alpha) + \iota_U d\alpha = \iota_Ud\alpha \qquad\text{and thus}\qquad \mathcal{L}_{\partial_s}\iota^*\alpha = \theta
\]
It follows that $\iota^*\alpha$ and $s \theta + \nu$ satisfy the same ODE and have the same restriction to $0 \times \Sigma$. Therefore they are equal on the given tubular neighborhood. \end{proof}

\subsection{Suspensions} The key examples of contact Hamiltonian manifolds for the purposes of this paper are suspensions of contactomorphisms (or synonymously, mapping tori). This is analogous to the mapping torus construction of stable Hamiltonian manifolds \cite[\S 2.1]{ce2015}.

\vspace{3pt}

Fix a contact manifold $(Y,\xi)$ with a contactomorphism $\Phi$ of $Y$. Recall that the suspension $\Sigma(\Phi)$ of $\Phi$ is the quotient of $\R \times Y$ by the map
\[\bar{\Phi}:\R \times Y \to \R \times Y \qquad\text{given by}\qquad \bar{\Phi}(t,y) = (t-1,\Phi(y))\]
Since $\xi$ is preserved by $\Phi$, the contact Hamiltonian structure $\on{span}(\partial_t) \oplus \xi$ on the product $\R \times Y$ (see Example \ref{ex:product}) is $\bar{\Phi}$-invariant. It descends to a contact Hamiltonian structure $\eta$ on the suspension. 

\begin{definition}[Contact Suspension] The \emph{contact suspension} of a contactomorphism $\Phi:(Y,\xi) \to (Y,\xi)$ is the contact Hamiltonian manifold given by
\[(\Sigma(\Phi),\eta)\]
The characteristic foliation and a natural framing form are given by the coordinate vector-field and covector-field in the $t$-direction.
\[\Sigma(\Phi)_\eta = \on{span}(\partial_t) \qquad\text{and}\qquad \theta = dt\]
\end{definition}

\begin{remark} In this case, the contact structure on the contactization extends to all of $\R \times \Sigma(\Phi)$, and we will refer to this latter space as the contactization.
\end{remark}

The most important result of this section is the following lemma, which relates graph-like perturbations of the suspension hypersurface to perturbations of the underlying contactomorphism.  

\begin{lemma}[Hamiltonian Perturbation] \label{lem:hamiltonian_perturbation} Let $H:\Sigma(\Phi) \to \R$ be a smooth function on the suspension of $\Phi:(Y,\xi) \to (Y,\xi)$ such that $\theta = dt$ frames $\eta_H$. Then there exists a contactomorphism
\[
\Phi^H:(Y,\xi) \to (Y,\xi) \qquad\text{with an isomorphism}\qquad \Psi^H:(\Sigma(\Phi^H),\eta) \to (\Sigma(\Phi),\eta_H)
\]
such that $\on{dist}_{C^1}(\Phi^H,\Phi) \le C \cdot \|H\|_{C^2}$ for any Riemannian metric $g$ and a constant $C = C(g)$. 
\end{lemma}

\begin{proof} Let $Z_H$ denote the characteristic vector-field of $\eta_H$ with respect to the framing form $dt$. The characteristic flow $\Psi^H$ of $Z_H$ satisfies $\Psi^H_t(0 \times Y) = t \times Y$ since $dt(Z_H) = 1$, and $\Psi^H$ preserves $\eta_H$ by Lemma \ref{lem:characteristic_flow}. Finally, note that $\nu_H = H dt + \nu$ restricts to $\alpha$ on $0 \times Y$ and thus $\eta_H \cap (0 \times Y)$ is $\xi$. By restriction to $\R \times Y$ where $Y$ is identified with $0 \times Y$ in $\Sigma(\Phi)$, we get a map
\[
\Psi^H:\R \times Y \to \Sigma(\Phi) \qquad\text{with}\qquad (\Psi^H)^*dt = dt \text{ and } (\Psi^H)^*\eta_H = \on{span}(\partial_s) \oplus \xi
\]
We now define $\Phi^H$ to be the time 1 map $\Psi^H_1$ of the flow. Then the map $\Psi^H$ satisfies
\[
\Psi^H(s,x) = \Psi^H(s-1,\Phi^H(x))
\]
In particular, $\Psi^H$ descends to a map $\Phi^H:\Sigma(\Phi^H) \to \Sigma(\Phi)$ with $(\Phi^H)^*\eta_H = \eta$. Finally, note that by Definition \ref{def:characteristic_foliation}, $Z_H$ is defined by the formulas
\[
\theta(Z_H) = 1 \qquad \iota_{Z_H}(H \cdot \theta + \nu) = 0 \qquad \iota_{Z_H}(dH \wedge \theta + d\nu)  = 0
\]
It follows that there is a smooth linear bundle map
\[
T:\Lambda^0(\Sigma(\Phi)) \oplus \Lambda^1(\Sigma(\Phi)) \oplus \Lambda^2(\Sigma(\Phi)) \to T\Sigma(\Phi) \qquad\text{such that}\qquad Z_H = T(H,dH,\nu,d\nu)
\]
In particular, for any choice of metric on $\Sigma(\Phi)$, there is a constant $C > 0$ and an estimate
\[
\|Z_H - Z\|_{C^1} \le C \cdot \|H\|_{C^2}
\]
The same estimate holds for the flow and the time-1 maps. \end{proof}

\begin{example}[Mapping Torus Of Identity] \label{ex:deforming_mapping_torus} Let $(Y,\xi)$ be a contact manifold with contact form $\alpha$ and consider the suspension of the identity
\[\Sigma(\on{Id}_Y) = (\R/\Z)_t \times Y \qquad\text{with contact Hamiltonian form }\nu = \alpha\]
Fix a Hamiltonian $H:\R/\Z \times Y \to \R$ and let $V_H:\R/\Z \times Y \to TY$ be the contact vector-field of $H$. Consider the deformation
\[
(\R/\Z \times Y,\nu_{-H}) \qquad\text{with}\qquad \nu_{-H} = -Hdt + \alpha
\]It is simple to check that the characteristic vector field for the framing form $dt$ is given by
\[
Z_H = \partial_t + V_H
\]
It follows that the contactomorphism $\Phi^H$ constructed in Lemma \ref{lem:hamiltonian_perturbation} is precisely the time 1 map of the contactomorphism generated by $-H$ and map $\Psi$ defines an isomorphism
\[
(\Sigma(\Phi^H),\eta) \simeq (\R/\Z \times Y, \eta_{-H})
\]\end{example}

We easily derive the following analogue of Lemma \ref{lem:hamiltonian_perturbation} for perturbations of hypersurfaces.

\begin{lemma}[Surface Perturbation] \label{lem:surface_perturbation} Let $\Phi$ be a contactomorphism. Then there is a $C^2$-neighborhood
\[
\mathcal{U} \subset \on{Emb}(C\Sigma(\Phi)) \qquad\text{of the sub-manifold}\qquad \iota:\Sigma(\Phi) \to C\Sigma(\Phi)
\]
in the space of embedded smooth sub-manifolds such that any element $\Sigma \in \mathcal{U}$ has an isomorphism
\[\Psi_\Sigma:(\Sigma,\eta \cap \Sigma) \to (\Sigma(\Phi_\Sigma),\eta) \qquad\text{where}\qquad \on{dist}_{C^1}(\Phi_\Sigma,\Phi) \le C \cdot \on{dist}_{C^2}(\Sigma,\Sigma(\Phi))
\]
\end{lemma}

\begin{proof} Any surface $\Sigma$ in the contactization $(-\epsilon,\epsilon) \times \Sigma(\Phi)$ that is $C^2$-close to $0 \times \Sigma(\Phi)$ is the graph of a function $H$ on $\Sigma(\Phi)$ with $C^2$-norm controlled by the $C^2$-distance of $\Sigma$ to $0 \times \Sigma$. Thus this lemma is immediate from Lemma \ref{lem:hamiltonian_perturbation} and Definition \ref{def:deformation}.
\end{proof}

We are now ready to prove the suspension result Theorem \ref{thm:main_suspension} that was used in the introduction.

\begin{prop}[Theorem \ref{thm:main_suspension}] \label{prop:main_suspension} Let $\Phi:Y \to Y$ be a $C^1$-robustly transitive contactomorphism. Then the suspension $\Sigma(\Phi)$ has a $C^2$-neighborhood
\[\mathcal{U} \subset \on{Emb}(C\Sigma(\Phi))\] consisting of hypersurfaces with topologically transitive characteristic foliation. \end{prop}

\begin{proof} By Lemma \ref{lem:surface_perturbation}, there is a neighborhood $\mathcal{U}$ of $\Sigma(\Phi)$ in the $C^2$-topology such that every $\Sigma \in \mathcal{U}$ is the suspension of a contactomorphism $\Phi_\Sigma$ that is $C^1$-close to $\Phi$. Since $\Phi$ is robustly transitive (Definition \ref{def:robustly_transitive}), we can assume after shrinking $\mathcal{U}$ that $\Phi_\Sigma$ is topologically transitive for any $\Sigma$ in $\mathcal{U}$. Finally, note that the suspension of a transitive diffeomorphism is a transitive flow. Indeed, a smooth flow (or diffeomorphism) is topologically transitive if and only if there is a dense orbit, and if a diffeomorphism $\Psi$ has a dense orbit, then the suspension of that orbit is also a dense orbit on $\Sigma(\Psi)$. Thus every $\Sigma \in \mathcal{U}$ has a topologically transitive characteristic foliation.\end{proof}

\subsection{Constructions Of Contact Hamiltonian Hypersurfaces} We conclude this section with constructions of contact Hamiltonian hypersurfaces. The following lemma is our main tool.

\begin{lemma}[Disk Neighborhood] \label{lem:disk_neighborhood} Let $\Gamma$ be a closed contact manifold and let $\Phi:\Gamma \to \Gamma$ be a positive contactomorphism. Fix a contact form $\beta$ and let
\[H:\R/\Z \times \Gamma \to (0,\infty)\]
be the corresponding generating Hamiltonian such that $\Phi = \Phi^H_1$. Then there is a contact embedding
\begin{equation} \label{eq:disk_neighborhood} (-\epsilon,\epsilon)_s \times \Sigma(\Phi) \to (D^2 \times \Gamma,-a \cdot r^2 d\theta + \beta) \qquad\text{ for any }a > \frac{1}{2\pi} \cdot \on{max} H\end{equation}
\end{lemma}

\begin{proof} We use the following smooth map in radial coordinates.
\[
\iota:(-2\pi a,0)_s \times \R/\Z_t \times \Gamma \to (D^2 - 0) \times \Gamma \quad\text{given by}\quad \iota(s,t,x) = ((-s/2\pi a)^{1/2},2\pi t,x)
\]
This map satisfies $\iota^*(-ar^2 d\theta + \beta) = sdt + \beta$. Fix $a$ satisfying $2\pi a > \on{max} H$. By Example \ref{ex:deforming_mapping_torus}, the graph of $-H$ defines an embedding
\[
(\Sigma(\Phi),\eta) \simeq (\R/\Z \times \Gamma,-Hdt + \beta) \to ((-2\pi a,0)_s \times \R/\Z_t \times \Gamma,sdt + \beta)
\]
This embedding extends to a contactomorphism (\ref{eq:disk_neighborhood}) by the flow of $\partial_s$ (see Lemma. \ref{lem:collar_neighborhood}). \end{proof}

The next lemma shows that small regions of the jet bundle contains arbitrarily large tubular neighborhoods of the cosphere bundle.  

\begin{lemma} \label{lem:cosphere_neighborhood} Let $X$ be a closed smooth manifold with cosphere bundle $SX$ and jet bundle $JX$. Fix a neighborhood $U$ of $X \subset JX$ and a contact form $\beta$ on $SX$. Then for any $a > 0$, there is a contact embedding
\[
(D^2 \times SX, \xi_{a,\beta}) \to U \qquad\text{where}\qquad \xi_{a,\beta} = \on{ker}(-a \cdot r^2d\theta + \beta)
\]
\end{lemma}

\begin{proof} Recall that the jet bundle $JX$ is given by $\R_t \times T^*X$ with the standard contact form $\alpha_{\on{std}} = dt + \lambda_{\on{std}}$. We break the proof into two steps.

\vspace{3pt}

\noindent {\bf Step 1.} First assume that $U = JX$ (so that we may ignore the neighborhood). There is a standard embedding of the symplectization of $\R \times SX$ into $T^*X$ via the Liouville flow of $T^*X$. By using $t$-translation, we can extend this to an embedding
\[
\kappa:\R_\rho \times \R_t \times SX \to JX \qquad\text{with}\qquad \kappa^*\alpha_{\on{std}} = dt + e^\rho \beta
\]
By applying a further change coordinates by taking $s = -e^{-\rho}$, we get a map
\[
\jmath:(-\infty,0]_s \times \R_t \times SX \to JX \qquad\text{with}\qquad \jmath^*(dt + e^\rho \beta) = dt - s^{-1} \beta  = -s^{-1} \cdot (-s \cdot dt + \beta)
\]
Now we construct an embedding to $(-\infty,0]_s \times \R_t \times SX$. Take a disk $D \subset (-\infty,0]_s \times \R_t$ of radius $(2a)^{1/2}$ centered at a point $(s_0,t_0)$. We consider the Liouville form
\[\lambda = \frac{1}{2}((t - t_0)ds - (s - s_0)dt)) \qquad\text{satisfying}\qquad d\lambda = -ds \wedge dt\]
Next, let $\tau:\R^2 \to \R$ be a primitive such that $\lambda = -sdt +  d\tau$ and consider the diffeomorphism
\[
\Psi:\R_s \times \R_t \times SX \to \R_s \times \R_t \times SX \quad\text{given by}\quad \Psi(s,t,x) = (s,t,\Phi^R(\tau(s,t),x)
\]
Here $\Phi^R:\R_\tau \times SX \to SX$ denotes the Reeb flow of $\beta$. We compute that
\[
\Psi^*(-sdt + \beta) = -sdt + \beta(R) \cdot d\tau + (\Phi^R)^*\beta = -sdt + d\tau + \beta = \lambda + \beta
\]
Finally, we compose $\Psi$ with the map $\Phi = \phi \times \on{Id}_{SX}$ where
\[\phi:D^2 \to \R_s \times \R_t \qquad\text{given by}\qquad (r,\theta) \mapsto (a r\cos(\theta) + s_0, a r\sin(\theta) + t_0)\]
The composition $\Psi \circ \Phi$ now restricts to a contact embedding
\[
D^2 \times SX \to (-\infty,0]_s  \times SX \qquad\text{with}\qquad (\Psi \circ \Phi)^*(-s \cdot dt + \beta) = -\frac{a}{2} \cdot r^2 d\theta + \beta
\]
The composition $\jmath \circ \Psi \circ \Phi:D^2 \times SX \to JX$ is the desired embedding in the lemma.

\vspace{3pt}

\noindent {\bf Step 2.} Now consider the general case where $U \subset JX$ is a proper open set. There is a natural flow of contactomorphisms $\Phi':\R_r \times JX \to JX$ given by $\Phi_r^Z(t,z) = (e^rt, \Phi^Z_r(z))$ where $\Phi^Z$ is the Liouville flow on $T^*X$. This flow is generated by a vector-field $U = t\partial_t + Z$ satisfying
\[
\mathcal{L}_U\alpha_{\on{std}} = \alpha_{\on{std}}
\]
In particular, any compact set in $JX$ can be pushed into $U$ by $\Phi'_r$ for $r$ sufficiently negative. We may then compose the embedding from Step 1 with $\Phi'_r$ to acquire the desired embedding. \end{proof}

By using the Weinstein neighborhood theorem for Legendrians \cite{g2008intro} to convert a Legendrian into an embedding of the cosphere bundle of the Legendrian, we acquire the following corollary.

\begin{lemma}[Lemma \ref{lem:intro_embedding}] \label{lem:legendrian_embeddings} Let $\Lambda \subset (Y,\xi)$ be a closed Legendrian sub-manifold and let $\Phi:S\Lambda \to S\Lambda$ be a positive contactomorphism of the cosphere bundle $S\Lambda$. Then there is a contact embedding
\[
(-\epsilon,\epsilon) \times \Sigma(\Phi) \to (Y,\xi) 
\]
\end{lemma}

\begin{proof} By the Weinstein neighborhood theorem, we have a contact embedding
\[
(U,\xi_{\on{std}}) \to (Y,\xi) \qquad\text{for a neighborhood }U \subset JX \text{ of }X
\]
By Lemma \ref{lem:disk_neighborhood} and Lemma \ref{lem:cosphere_neighborhood}, for any contact form $\beta$ on $S\Lambda$, there are constants $\epsilon,a > 0$ and a contact embedding of the form
\[
((-\epsilon,\epsilon) \times \Sigma(\Phi),\eta) \to (D^2 \times SX,\xi_{a,\beta}) \to (U,\xi_{\on{std}})
 \qedhere\]
\end{proof}

\section{Partial Hyperbolicity} \label{sec:partial_hyperbolicity} In this section, we review the theory of partially hyperbolic maps and blenders.

\begin{remark} This section contains extensive background aimed at non-experts in dynamics. We recommend that the reader look to  Crovisier-Potrie \cite{crovisier2015introduction}, Hertz-Hertz-Ures \cite{hertz2011partially}, Hirsch-Pugh-Shub \cite{hps1977} or Bonatti-Diaz-Viana \cite{bonatti2004dynamics} for a more comprehensive treatment. We also recommend the excellent book of Fisher-Hasselblatt \cite{fh2019hyperbolic} for an accessible treatment of hyperbolic dynamics.
\end{remark}

\subsection{Fundamentals} \label{subsec:fundamentals} Fix a compact smooth manifold $Y$ and a diffeomorphism
\[
\Phi:Y \to Y
\]
\begin{definition}[Expansion/Contraction] A sub-bundle $E \subset TY$ is \emph{uniformly expanding} with respect to $\Phi$ and a Riemannian metric $g$ on $Y$ if there are constants $C > 0$ and $\lambda > 1$ such that
\[
|T\Phi^n(u)| \ge C \cdot \lambda^n \cdot |u| \qquad\text{for all }n \ge 0
\]
Similarly, $E$ is \emph{uniformly contracting} for $\Phi$ and $g$ if it is uniformly expanding for $\Phi^{-1}$ and $g$. \end{definition}

\begin{definition}[Domination] A continuous splitting of $TY$ into continuous sub-bundles
\[TY = E_1 \oplus \dots \oplus E_k\] is \emph{dominated} with respect to $\Phi$ and a metric $g$ if there are constants $C > 0$ and $\lambda > 1$ such that
\[
|T\Phi^n(u)| \ge C \cdot \lambda^n \cdot |T\Phi^n(v)| \qquad\text{for all }n \ge 0 \text{ and any unit vectors }u \in E_{i+1} \text{ and }v \in E_i
\]
The constant $\lambda$ is the \emph{constant of dilation} and the splitting may also be called \emph{$\lambda$-dominated}.\end{definition}

Dominated splittings obey several fundamental properties (cf. \cite[Appendix B]{bonatti2004dynamics} or \cite{crovisier2015introduction}) that we now discuss briefly. First, dominated splittings are persistent  with respect to $C^1$-perturbation.

\begin{thm}[$C^1$-Persistence] \label{thm:persistence_of_splitting} Let $\Phi:Y \to Y$ be a diffeomorphism with a $\lambda$-dominated splitting
\[TY = E_1 \oplus \dots \oplus E_k\]
Then for $\mu < \lambda$, there is a $C^1$-neighborhood $\mathcal{U}$ of $\Phi$ such that every $\Psi \in \mathcal{U}$ has a $\mu$-dominated splitting
\[TY = E_1(\Psi) \oplus \dots \oplus E_k(\Psi) \qquad\text{such that}\qquad \on{dim} E_i(\Psi) = \on{dim} E_i(\Phi)\]
 \end{thm}

\noindent Next, dominated splittings with an expanding (or contracting) factor possess a unique, expanding (or contracting) invariant foliation (cf. \cite{hps1977} and \cite[Thm B.7]{bonatti2004dynamics}).

\begin{thm}[Foliations] \label{thm:invariant_foliations} \cite{hps1977} Let $\Phi:Y \to Y$ be a diffeomorphism with a dominated splitting
\[
TY = D \oplus E \quad\text{with $E$ uniformly expanding (or contracting)}
\]
Then there is a unique H\"{o}lder foliation $F$ with smooth leaves such that
\[\text{$F$ is tangent to $E$} \quad\qquad \Phi_*F = F \qquad\text{and}\qquad F(P) \simeq \R^{\on{rk}(E)} \]
where $F(P)$ denote the leaf through the point $P$ in $Y$.\end{thm}

\noindent Finally, the stable manifold theorem (c.f. Hirsch-Pugh-Shub \cite[Thm 4.1]{hps1977}) asserts the existence of stable and unstable manifolds for hyperbolic invariant sets.

\begin{thm}[Invariant Manifolds] Let $\Gamma$ be an invariant sub-manifold of a diffeomorphism $\Phi:Y \to Y$ with normal hyperbolic splitting $E^u \oplus T\Gamma \oplus E^s$. Then there are unique invariant sub-manifolds
\[W^s(\Gamma,\Phi) \qquad \text{and} \qquad W^u(\Gamma,\Phi) \qquad\text{containing }\Gamma\]
that are locally invariant near $\Gamma$, $C^1$-robust and that satisfy
\[TW^s(\Phi,\Gamma) = E^s \oplus T\Gamma \qquad\text{and}\qquad TW^u(\Phi,\Gamma) = E^u \oplus T\Gamma \qquad\text{along }\Gamma\]
\end{thm}

\begin{notation}[Local Invariant Manifolds] Given a normally hyperbolic invariant set $\Gamma$ and an open neighborhood $U$ of $\Gamma$ we will use the notation
\[W^s_{\on{loc}}(\Gamma,\Phi;U)  \qquad\text{and}\qquad W^u_{\on{loc}}(\Gamma,\Phi;U)\]
to denote the respective components of $W^s(\Gamma,\Phi) \cap U$ and $W^u(\Gamma,\Phi) \cap U$ that contain $\Gamma$. We adopt analogous notation for the leaves of the foliations $F^s(\Phi)$ and $F^u(\Phi)$ when defined. \end{notation}

This paper will be entirely oriented towards the following class of diffeomorphisms. 

\begin{definition}[Partially Hyperbolic] \label{def:partially_hyperbolic} A diffeomorphism $\Phi:Y \to Y$ is \emph{partially hyperbolic} if there is a splitting the tangent bundle into $\Phi$-invariant, continuous sub-bundles 
\[
TY = E^s(\Phi) \oplus E^c(\Phi) \oplus E^u(\Phi)
\]
such that $E^s(\Phi)$ is uniformly contracting, $E^u(\Phi)$ is uniformly expanding and the splitting is dominated with respect to $\Phi$ and some (or equivalently any) Riemannian metric $g$.
\end{definition}

\begin{example} Let $\Psi:\R \times Y \to Y$ be an Anosov flow generated by a vector-field $V$. Then for any $T$, the time $T$ map
\[\Psi_T:Y \to Y\]
is partially hyperbolic. The stable and unstable bundles of $\Psi_T$ are those of $\Psi$, while the central bundle is the span of $V$.  \end{example}

\noindent Partially hyperbolic contactomorphisms have some special compatibility properties with the underlying contact structure. For instance, we have the following lemma (for use later).

\begin{lemma} \label{lem:phcontactomorphism_properties} Fix a contact $(2n+1)$-manifold $(Y,\xi)$ and a partially hyperbolic contactomorphism
\[\Phi:(Y,\xi) \to (Y,\xi)\]
such that $\xi = E^s(\Phi) \oplus E^u(\Phi)$ and $E^c(\Phi)$ is transverse to $\xi$. Then the stable and unstable foliations $F^s(\Phi)$ and $F^u(\Phi)$ have Legendrian leaves.\end{lemma}

\begin{proof} 

The leaves of $F^s(\Phi)$ and $F^u(\Phi)$ are tangent to $\xi$, and the dimensions of the leaves of $F^s(\Phi)$ and $F^u(\Phi)$ add to the rank of $\xi$. It follows that the dimension of the leaves $F^s(\Phi)$ and $F^u(\Phi)$ must be $n$, so that the leaves are Legendrian.
\end{proof}

\subsection{Holonomy} \label{subsec:holonomy} The holonomy of the stable and unstable foliations are a key tool in the analysis of partially hyperbolic maps, which we will need in Section \ref{sec:blender_construction}. We next briefly review this concept.

\vspace{3pt}

Let $F$ be a transversely continuous foliation with smooth leaves on a manifold $Y$. Let $\Lambda$ be a leaf and let $P$ be a point in $\Lambda$. Recall that a transversal to $\Lambda$ at $P$ is a sub-manifold $S \subset Y$ of dimension $\on{codim} F$ tranverse to the leaves of $F$ and intersecting $\Lambda$ at $P$. 

\begin{definition}[Holonomy] \label{def:holonomy} \cite[Ch 2]{cc1985foliations} Let $\Lambda$ be a leaf of $F$ equipped with transversals $S$ and $T$ at points $P$ and $Q$ in $\Lambda$. Fix a continuous path
\[
\Gamma:[0,1] \to \Lambda \qquad\text{with}\qquad \Gamma(0) = P \text{ and }\Gamma(1) = Q
\]
The \emph{holonomy} is the unique correspondence assigning to $(\Gamma,S,T)$ the germ of a smooth map
\[
\on{Hol}_{F,\Gamma}:S \cap \on{Nbhd}(P) \to T \cap \on{Nbhd}(Q) \qquad\text{ with }\qquad \on{Hol}_{F,\Gamma}(P) = Q
\]
that satisfies the following properties.
\begin{itemize}
    \item The correspondence $\on{Hol}$ respects path composition.
    \vspace{3pt}
    \item If $\Gamma,S,T$ are in a foliation chart, then $\on{Hol}_\Gamma(s)$ is the point in $T$ in the same plaque as $s$. 
\end{itemize}
Recall that a plaque refers to a local leaf of the foliation in a foliation chart. The holonomy only depends on $\Gamma$ up to homotopy relative to $P$ and $Q$ in $\Lambda$ \cite[Prop 2.3.2]{cc1985foliations}. Moreover, given a foliation $G$ such that $TF$ and $TG$ are $C^0$-close, there is  a continuation point
\[Q_G \in \on{Nbhd}(Q) \cap T\]
on the leaf of $G$ through $P$, converging to $Q$ as $G$ converges to $F$. There is also a unique homotopy class of path $P$ to $Q_G$ corresponding to $\Gamma$. Thus there is a holonomy map
\[
\on{Hol}_{G,\Gamma}:S \cap \on{Nbhd}(P) \to T \cap \on{Nbhd}(Q)
\]
The holonomy $\on{Hol}_{F,\Gamma}$ varies continuously with respect to the foliation $F$. \end{definition}

By  Theorem \ref{thm:invariant_foliations}, the leaves of the stable foliation $F^s$ and unstable foliation $F^u$ of a partially hyperbolic diffeomorphism $\Phi$ are contractible and so the holonomy is independent of the path $\Gamma$. In this case, we denote the corresponding holonomies
\[\on{Hol}^s_\Phi = \on{Hol}_{F^s,\Gamma} \qquad\text{and}\qquad \on{Hol}^u_\Phi = \on{Hol}_{F^u,\Gamma}\]
The main property of holonomy that we will need is the following regularity result, which follows from (for instance) the uniform H\"{o}lder regularity results of Pugh-Shub-Wilkenson \cite[Thm A]{psw1997}.

\begin{lemma}[H\"{o}lder Holonomy] \label{lem:holder_holonomy} Let $\Phi:Y \to Y$ be a partially hyperbolic diffeomorphism. Fix points $P$ and $Q$ in an unstable leaf with transversals $S$ and $T$. Then there is
\[
\text{a $C^1$-neighborhood $\mathcal{U}$ of $\Phi$} \qquad \text{a H\"{o}lder constant $\kappa \in (0,1)$} \qquad\text{and}\qquad \text{a neighborhood $\on{Nbhd}(P)$}
\]
so that  the holonomy $\on{Hol}_\Psi^u:S \cap \on{Nbhd}(P) \to T \cap \on{Nbhd}(Q)$ of any diffeomorphism $\Psi$ in $\mathcal{U}$ satisfies
\begin{equation} \on{dist}\big(\on{Hol}^u_\Psi(x),\on{Hol}^u_\Psi(y)\big) < \on{dist}(x,y)^\kappa \end{equation}
\end{lemma}

\subsection{Cone-Fields} There is an alternative formulation of dominated splittings in terms of cone-fields. Cone-fields will also be used extensively in the definition and construction of blenders. 

\begin{definition}[Cone Fields] A continuous \emph{cone-field} $K$ on a manifold $X$ is a bundle of the form
\[
K = \big\{v \; : \; Q(v) \le 0\big\} \qquad\text{for a continuous, non-degenerate quadratic form }Q:TX \to \R
\]
A diffeomorphism $\Phi:X \to Y$ induces a natural pushforward cone-field
\[
\Phi_*K \qquad\text{with fiber} \qquad \Phi_*K_x = T\Phi_x(K_{\Phi^{-1}(x)})
\]
The \emph{interior} $\on{int} K$ of a cone-field $K$ is the union of the (fiberwise) interior of $K$ and the zero section.  \end{definition}

\begin{example}[Metric Cones] \label{ex:metric_cones} The cone-field $K_\epsilon E$ of width $\epsilon$ around a sub-bundle $E \subset TX$ of the tangent bundle of a Riemannian manifold $(X,g)$ is defined by
\[
K_\epsilon E := \big\{v \in TX \; : \; |v - \pi_E(v)|_g \le \epsilon \cdot |\pi_E(v)|_g\big\} \qquad\text{where $\pi_E$ is orthogonal projection to $E$}
\]
A cone-field $K$ has \emph{width less than $\epsilon$ with respect to $g$} if $K \subset K_\epsilon E$ for some linear sub-bundle $E \subset TX$. \end{example}

\begin{definition}[Contraction/Dilation] Fix a subset $A \subset X$ and a cone-field $K$ over $X$. A diffeomorphism $\Phi$ \emph{contracts $K$ over $A$} if
\[\Phi_*(K|_A) \subset \on{int} K\]
Given a Riemannian metric $g$, we say that $\Phi$ \emph{dilates} $K$ over $A$ with constant of dilation $\lambda > 0$ if
\[
\lambda \cdot |v|_g \le |\Phi_*v|_g \qquad\text{for every }v \in K|_A
\]\end{definition}

\begin{thm}[Invariant Cone-Fields] \label{thm:invariant_cone_field} (c.f. \cite[\S 2.2]{crovisier2015introduction}) Let $\Phi:Y \to Y$ be a diffeomorphism with a dominated splitting $TY = D \oplus E$. Then for any Riemannian metric $g$ and any $\epsilon > 0$, there is an $N > 0$ such that
\[\Phi^n \text{ contracts }K_\epsilon E \text{ for all }n \ge N\]
Moreover, if $E$ is uniformly expanding, then for any $\lambda > 1$, we may choose $N$ so that
\[
\Phi^n \text{ dilates }K_\epsilon E \text{ with dilation factor $\lambda$ for all }n \ge N
\]\end{thm}

\subsection{Blenders} \label{subsec:blenders} A blender is a type of robust hyperbolic set within a dynamical system, introduced by Bonatti-Diaz \cite{bonattidiaz1995}. Blenders play a central role in the construction of robustly mixing maps in the partially hyperbolic setting.

\vspace{3pt}

In this section, we precisely define blenders (Definition \ref{def:blender}) and associated structures. We start with the notion of a blender box, which is a chart where the blender structure will reside.

\begin{definition}[Blender Box] \label{def:blender_box} A \emph{blender box} $B$ of type $(k,l)$ in an $n$-manifold $X$ is a $C^1$-embedded $n$-manifold with corners with coordinates
\[
(s,t,u):B \xrightarrow{\sim} D^k_s \times D^1_t \times D^l_u \subset \R^n
\]  
Here $D^m_x$ denotes a closed metric $m$-ball of some (unspecified) radius in $\R^m_x$, with coordinates $x_i$. The boundary of $B$ has distinguished subsets that we denote as follows.
\[
\partial^sB = \partial D^k \times D^1 \times D^l \qquad \partial^cB = D^k \times \partial D^1 \times D^l \qquad \partial^u B = D^k \times D^1 \times \partial D^l
\]
These are the \emph{stable}, \emph{central} and \emph{unstable} boundaries, respectively. The boundary of $D^1$ is a union of two points, a negative point $\partial^-D^1$ and a positive point $\partial^+D^1$. We also fix the notation
\[
\partial^l B = D^k \times \partial^- D^1 \times D^l \qquad\text{and}\qquad \partial^r B = D^k \times \partial^+ D^1 \times D^l 
\]
These subsets are the \emph{left-side} and \emph{right-side} of $B$, respectively. \end{definition} 

Next, we introduce the notion of compatible cone-fields and vertical/horizontal disks. These structures can be viewed as local (and more flexible) versions of the invariant bundles and foliations associated to a partially hyperbolic map.

\begin{definition}[Compatible Cones] \label{def:compatible_cones} A triple of smooth cone-fields $K^s, K^{cu}$ and $K^u$ on $X$ are \emph{compatible} with a blender box $B$ if there are inclusions
\[TD^k \times D^1 \times D^l \subset K^s \qquad D^k \times TD^1 \times TD^l \subset K^{cu} \qquad D^k \times D^1 \times TD^l \subset K^u\] We refer to these cone-fields as \emph{stable, central-unstable} and \emph{unstable} cones for $B$, respectively. \end{definition}

\begin{definition}[Vertical/Horizontal Disks] An embedded $k$-disk $D \subset B$ is called \emph{vertical} with respect to an unstable cone-field $K^u$ if
\[
TD \subset K^u \qquad \text{and}\qquad \partial D \subset \partial^u B
\]
Similarly, an embedded $l$-disk $D \subset B$ is called \emph{horizontal} with respect to a stable cone-field $K^s$ if
\[TD \subset K^s \qquad\text{and}\qquad \partial D \subset \partial^sB\]
Given a horizontal disk $H \subset B$ that is disjoint from $\partial^u B$, there are precisely two homotopy classes of vertical disk disjoint from $H$, corresponding to the disks
\[
D_{\on{Right}} = 0_s \times +1 \times D^l \qquad \text{and} \qquad D_{\on{Left}} = 0_s \times -1 \times D^l
\]
A vertical disk $D$ is \emph{right} of $H$ if it is homotopic to $D_{\on{Right}}$ and \emph{left} of $H$ if it is homotopic to $D_{\on{Left}}$. \end{definition}

It will be useful to record some basic properties of vertical disks in the following lemmas. These are entirely elementary and left to the reader (also see \cite[\S 1]{bonattidiaz1995} or \cite{bonatti2004dynamics}).

\begin{lemma}[Vertical Manifolds] \label{lem:vertical_manifolds_disks} Let $\Sigma \subset B$ be any smooth sub-manifold in a blender box $B$ with
\[
T\Sigma \subset K^u \qquad\text{and}\qquad \partial \Sigma \subset \partial^u B
\]
Then $\Sigma$ is diffeomorphic to an $l$-disk, and therefore is a vertical disk. \end{lemma}

\begin{lemma}[Graphs] \label{lem:vertical_disks_graphs} Let $D$ be a vertical disk in a blender box $B$ with respect to an unstable cone-field $K^u$ of width $\epsilon$. Then $D$ is the graph
    \[
    D \subset D^k_s \times D^1_t \times D^l_u \qquad\text{of a $2\epsilon$-Lipschitz map }f:D^l_u \to D^k_s \times D^1_t
    \]
\end{lemma}

We are now prepared to introduce our operating definition of a blender (Definition \ref{def:blender}). We warn the reader that this definition is quite cumbersome, since it is entirely functional and tailored to the specific application. We will discuss more motivation in Remark \ref{rmk:hypotheses_of_blender} below.

\begin{definition}[Blender] \label{def:blender} Let $\Phi:X \to X$ be a $C^1$-diffeomorphism and let $B \subset X$ be a blender box (see Definition \ref{def:blender_box}). The pair
\[(B,\Phi)\]
is a \emph{simple stable blender} (or just a \emph{stable blender}) if it satisfies the following properties.
\begin{enumerate}[label=(\alph*)]
\item There is a connected component $A$ of $B \cap \Phi(B)$ that is disjoint from
\[\partial^s B \qquad \Phi(\partial^c B) \quad\text{and}\quad \Phi(\partial^u B)\]
\item There is an $m \in \N$ and a connected component $A'$ of $B \cap \Phi^m(B)$ that is disjoint from
\[\partial^r B \qquad \partial^s B \quad\text{and}\quad \Phi(\partial^uB)\]
\end{enumerate}
Moreover, there are constants $\mu > 1 > \epsilon$ and a compatible set of cone-fields $K^s,K^{cu}$ and $K^u$ on $Y$ of width less than $\epsilon$ (with respect to the standard metric on $B$) such that

\vspace{3pt}

\begin{enumerate}[label=(\alph*)] \setcounter{enumi}{2}

\item The cone-fields $K^s$ and $K^u$  are contracted and dilated (with constant $\mu$) by $\Phi^{-1}$ and $\Phi$.
\[\Phi^{-1}_*K^s \subset \on{int} K^s \qquad \Phi_*K^u \subset \on{int} K^u \qquad \Phi^{-1} \text{ dilates }K^s \qquad \Phi \text{ dilates }K^u 
\]

\item The cone-field $K^{cu}$ is contracted and dilated (with constant $\mu$) by $\Phi$ as follows.
\[\Phi_*K^{cu} \subset \on{int} K^{cu}  \quad\text{and}\quad \Phi \text{ dilates }K^{cu} \qquad\text{ over }\Phi^{-1}(A)\]
\[\Phi_*^mK^{cu} \subset \on{int} K^{cu}  \quad\text{and}\quad \Phi^m \text{ dilates }K^{cu} \qquad\text{ over }\Phi^{-m}(A')\]
\end{enumerate}
Finally, in any box $B$ equipped with such cone-fields $K^\bullet$, $\Phi$ has a unique hyperbolic fixed point $Q$ in the component $A$ of index $l$ (via \cite[Lemma 1.6]{bonattidiaz1995}), and the local stable manifold
\[
W := W^s_{\on{loc}}(Q,\Phi;B) \subset W^s(Q,\Phi) \cap B
\]
is a horizontal $k$-disk in the blender box $B$. We further assume that there are neighborhoods
\[
U_- \text{ of }\partial^l B \qquad U_+ \text{ of }\partial^r B \qquad U \text{ of }W
\]
satisfying the following assumptions.
\begin{enumerate}[label=(\alph*)] \setcounter{enumi}{4}
    \item Every vertical disk $D$ through $B$ to the right of $W$ is disjoint from $U_-$.
    \vspace{3pt}
    \item Every vertical disk $D$ through $B$ to the right of $W$ satisfies one of two possibilities.
    \vspace{3pt}
    \begin{enumerate}
        \item The intersection $\Phi(D) \cap A$ contains a vertical disk $D'$ through $B$ to the right of $W$ and disjoint from $U_+$.
        \vspace{3pt}
        \item The intersection $\Phi^m(D) \cap A'$ contains a vertical disk $D'$ through $B$ to the right of $W$ and disjoint from $U$.
    \end{enumerate}
\end{enumerate}
\vspace{3pt}
The pair $(B,\Phi)$ is an \emph{simple unstable blender} if the pair $(B,\Phi^{-1})$ is a simple stable blender. \end{definition}

\begin{remark}[Blender Property] \label{rmk:hypotheses_of_blender}  Definition \ref{def:blender} is best understood as being specifically tailored to enforce the following \emph{distinctive blender property} (c.f. \cite[Lem 6.6 and \S 6.2.2]{bonatti2004dynamics}): 
\begin{equation} \label{eq:distinctive_blender_property} \text{Every vertical disk to the right of $W$ must intersect the (global) stable manifold of $Q$}\end{equation}
This property has the following easy consequence (\cite[Lem 1.8]{bonattidiaz1995} or \cite[Lem 6.8]{bonatti2004dynamics}): if $P$ is a hyperbolic periodic point with unstable manifold $W^u(P,\Phi)$ intersecting the blender box along a vertical disk to the right of $W$, then
\[W^s(P,\Phi) \subset \on{close}(W^s(Q,\Phi))\]
If $W^s(P,\Phi)$ has larger dimension than $W^s(Q,\Phi)$, then this says that $W^s(Q,\Phi)$ is bigger than expected. Since the distinctive property is robust, this can lead to robust transitivity properties of the global dynamics of $\Phi$ \cite[\S 7.1.3]{bonatti2004dynamics}. 

\vspace{3pt}

More contemporary accounts of blenders (c.f. \cite{whatisblender} and \cite[\S 6-7]{bonatti2004dynamics}) take the distinctive property (or a related property) as the \emph{definition} of a blender. With this in mind, we highly encourage the reader to view Definition \ref{def:blender} as a concrete list of criteria leading to the distinctive property (\ref{eq:distinctive_blender_property}). \end{remark}

\begin{remark}[Blender Variants] \label{rmk:altdef_of_blender} Several alternative definitions have been introduced since \cite{bonattidiaz1995} (cf. \cite[\S 6.2.2]{bonatti2004dynamics}). Definition \ref{def:blender} modifies the original definition of Bonatti-Diaz \cite[p. 365]{bonattidiaz1995} by replacing the hypothesis \cite[p. 365 H3]{bonattidiaz1995} with the streamlined hypothesis (c-d). 

\end{remark}

\section{Contact Blenders And Robustly Mixing Contactomorphisms} \label{sec:robustly_mixing_contactomorphisms}

In this section, we give a precise statement of the existence theorem for contact blenders in Theorem \ref{thm:heteroclinic_contact_blender}. We then prove Theorem \ref{thm:main_anosov} on the existence of robustly mixing contactomorphisms using contact blenders. The proof of Theorem \ref{thm:heteroclinic_contact_blender} is deferred to Sections \ref{sec:blender_construction} and \ref{sec:proof_of_blender_axioms}.

\subsection{Statement Of Contact Blender Theorem} We start by stating our existence theorem for contact blenders. We require the following setup.

\begin{setup}[Blender Setup] \label{set:blender_setup} Let $(Y,\xi)$ be a contact $(2n+1)$-manifold with contact form $\alpha$ and Reeb vector-field $R$. Fix a strict contactomorphism
\[\Phi:Y \to Y\]
that satisfies the following properties.
\begin{enumerate}[label=(\alph*)]
    \item The contactomorphism $\Phi$ is partially hyperbolic with stable, central and unstable splitting
    \[E^s(\Phi) \oplus E^c(\Phi) \oplus E^u(\Phi) \qquad\text{with}\qquad \xi = E^s(\Phi) \oplus E^u(\Phi) \quad\text{and}\quad E^c(\Phi) = \on{span}(R)\]
    \item There is a closed Reeb orbit $\Gamma \subset Y$ that is a normally hyperbolic set of fixed points of $\Phi$.
    \vspace{3pt}
    \item There is a neighborhood $V$ of $\Gamma$ and a smooth, integrable sub-bundle
    \[E_{\on{sm}}^s(\Phi) \subset \xi \quad\text{with foliation}\quad F_{\on{sm}}^s(\Phi) \qquad\text{ over }V\]
    that is uniformly contracted by $\Phi$ and that agrees with $E^s(\Phi)$ on $W^s(\Gamma,\Phi)$.

    \vspace{3pt}
    
    \item There are two points $P$ and $Q$ in $\Gamma$ such that the stable and unstable manifolds
    \[W^s(Q,\Phi) \qquad\text{and}\qquad W^u(P,\Phi)\]
    intersect cleanly along a heteroclinic orbit $\chi$ of $\Phi$ orbit from $P$ to $Q$.
\end{enumerate}
\end{setup}

\begin{remark} A contactomorphism satisfying Setup \ref{set:blender_setup} can be constructed by taking the time $T$ map of a (slight perturbation of) an Anosov Reeb flow with $C^\infty$ stable and unstable foliations. This will be discussed in the next part,  Section \ref{subsec:construction_of_RMCont}.
\end{remark}

The main existence result for contact blenders that we will use can now be stated as follows.

\begin{thm}[Heteroclinic Contact Blender] \label{thm:heteroclinic_contact_blender} For any contactomorphism $\Phi:Y \to Y$ as in Setup \ref{set:blender_setup}, there is an integer $N > 0$ and a smooth family of contactomorphisms
\[
\Psi:[0,1]_r \times Y \to Y \qquad\text{with}\qquad \Psi_0 = \Phi
\]
satisfying the following properties for all sufficiently small $r > 0$.
\begin{enumerate}[label=(\alph*)]
\item The points $P$ and $Q$ are hyperbolic fixed points of $\Psi_r$ of index $n+1$ and $n$, respectively.
\vspace{3pt}
\item There is a neighborhood $B_r(Q)$ of $Q$ such that $(B_r(Q),\Psi_r^N)$ is a stable blender. 
\vspace{3pt}
\item The intersections $W^u(P,\Psi_r) \cap B_r(Q)$ contains a vertical disk $D_Q$ to the right of $W^s(Q,\Psi_r)$.

\end{enumerate}
\end{thm}

\noindent We reiterate that the proof of Theorem \ref{thm:heteroclinic_contact_blender} is presented in the last two sections (Sections \ref{sec:blender_construction}-\ref{sec:proof_of_blender_axioms}).

\subsection{Construction Of Robustly Mixing Contactomorphisms} \label{subsec:construction_of_RMCont} We now prove Theorem \ref{thm:main_anosov} on the existence of robustly mixing contactomorphisms. Given the existence of a blender in the contact setting, this is a standard argument that follows the smooth case (cf. \cite[Section 4.C] {bonattidiaz1995}). 

\begin{thm}[Theorem \ref{thm:main_anosov}] \label{thm:body_anosov} Let $(Y,\xi)$ be a closed contact manifold admitting an Anosov Reeb flow $\Phi$ with $C^\infty$ stable and unstable foliations.  Then the $C^1$-open set of robustly mixing contactomorphisms
\[
\on{Cont}_{\on{RM}}(Y,\xi) \subset \on{Cont}(Y,\xi)
\]
is non-empty. Moreover, if $T$ is the period of a closed Reeb orbit of $\Phi$, then $\Phi_T$ is in the closure of this set.
\end{thm}

\begin{proof} Since $\Phi$ is Reeb Anosov, it must be transitive by the Plante alternative \cite[Thm 8.1.3 and 8.1.4]{fh2019hyperbolic}. A transitive Anosov flow must also have a dense set of closed orbits \cite[Thm 6.2.10]{fh2019hyperbolic}. 

\vspace{3pt}

Fix a closed orbit $\Gamma$ of period $T$ and an open neighborhood $U$ of $\Gamma$. Also pick an auxilliary orbit $\Xi$ with a neighborhood $V$. Let $\alpha$ denote the contact form with Reeb vector-field generating the Reeb flow $\Phi$ and choose a $1$-parameter family of contact forms $\alpha_s$ on $Y$ such that
\[\alpha_0  = \alpha \qquad \alpha_s|_U = \alpha \quad\text{and}\quad \alpha_s|_V = (1+s) \cdot \alpha\]
Let $\Phi^s$ denote the Reeb flow of $\alpha_s$. Anosov flows are $C^1$-structurally stable (\cite{a1963} or \cite[Thm 5.4.22]{fh2019hyperbolic}), and thus $\Phi_s$ is a smooth Anosov Reeb flow for sufficiently small $s$. Moreover, $\Gamma$ and $\Xi$ are orbits of $\Psi^s$ for all $s$.

\vspace{3pt}

Thus, choose an $s$ such that $\Phi^s$ is Reeb Anosov and such that $\Xi$ has a period that is not a multiple of $T$ with respect to $\Phi^s$. We let $\Psi$ be the time $T$ map of $\Phi^s$ and note that
\begin{itemize}
    \item $\Psi$ is a strict, partially hyperbolic contactomorphism with stable and unstable bundle equal to those of $\Phi^s$ and central bundle given by the span of the Reeb vector-field of $\alpha_s$.
    \vspace{3pt}
    \item $\Gamma$ is a closed Reeb orbit that is a normally hyperbolic fixed set of $\Psi$.
    \vspace{3pt}
    \item The stable bundle $E^s(\Phi)$ of $\Phi_T$ is smooth, integrable and uniformly contracted by $\Psi$ on $U$ since $\Psi$ agrees with $\Phi_T$ on $U$. 
\end{itemize}
Note also the local stable manifolds of $\Psi$ and $\Phi$ for the set $\Gamma$ (or for any point $P \in \Gamma$) agree.
\[W^s_{\on{loc}}(\Gamma,\Psi;U) = W^s_{\on{loc}}(\Gamma,\Phi_T;U) \qquad\text{and}\qquad W^s_{\on{loc}}(P,\Psi;U) = W^s_{\on{loc}}(P,\Phi_T;U)  \]
Since the points $P$ of $\Gamma$ are fixed, the stable manifold of $P$ is equal to the stable leaf of $P$ with respect to both $\Psi$ and $\Phi_T$. It follows that
\[
F^s(\Psi,P) = F^s(\Phi,P) \qquad\text{ on }W_{\on{loc}}^s(\Gamma,\Psi;U)
\]
This verifies the criteria in Theorem \ref{thm:heteroclinic_contact_blender}(a-c). To check Theorem \ref{thm:heteroclinic_contact_blender}(d) we apply the following lemma of Bonatti-Diaz. 

\begin{lemma} \cite[Lemma 4.3]{bonattidiaz1995} Let $\Psi$ be a partially hyprbolic map and $\Gamma,\Xi$ be closed orbits of different period, as constructed above. Then there are a pair of points $P,Q \in \Gamma$ and a heteroclinic orbit from $P$ to $Q$.  
\end{lemma}

The argument now proceeds identically to \cite[p. 395]{bonattidiaz1995} and we recall it here. Apply Theorem \ref{thm:heteroclinic_contact_blender} to acquire a $1$-parameter family of contactomorphisms $\Psi_r$ and an $N > 0$ such that
\begin{itemize}
    \item $\Psi_r$ has hyperbolic fixed points $P$ and $Q$ on $\Gamma$ of index $n+1$ and $n$, respectively. 
    \vspace{3pt}
    \item There is a neighborhood $B_r(Q)$ of $Q$ such that $(B_r(Q),\Psi_r^N)$ is a stable blender.
    \vspace{3pt}
     \item The intersection $W^u(P,\Psi_r) \cap B_r(Q)$ contains a vertical disk $D_Q$ to the right of $W^s(Q,\Psi_r)$.
\end{itemize}
This implies that the tuple $((B_r(Q),\Psi_r^N),P)$ is a chain of blenders in the sense of Bonatti-Diaz (see \cite[p. 369]{bonattidiaz1995} or \cite[\S 7.1]{bonattidiaz1995}). Lemma 1.12 of \cite{bonattidiaz1995} now states that there is a $C^1$-neighborhood $\mathcal{U}$ of $\Psi_r$ such that, for any $\Psi' \in \mathcal{U}$, there are fixed points $P'$ and $Q'$ (the continuations of $P$ and $Q$, which are non-degenerate hyperbolic fixed points and thus persist in a $C^1$-nieghborhood) such that
\begin{equation} \label{eq:stable_closure}
W^s(P',\Psi') \subset \on{close}(W^s(Q',\Psi')) \end{equation}
Here $\on{close}(-)$ denotes the topological closure. 

\vspace{3pt}

Next, since $\Phi^s$ is Anosov and transitive, the time $T$ map $\Psi$ is partially hyperbolic with well-defined center-stable and center-unstable foliations
\[F^{cs}(\Psi) \text{ tangent to }\on{span}(R) \oplus E^s(\Psi) \qquad\text{and}\qquad F^{cu}(\Psi) \text{ tangent to } \on{span}(R) \oplus E^u(\Psi)\]
Moreover, the center-stable and center-unstable foliations have dense leaves \cite[Thm 6.2.10]{fh2019hyperbolic}. By Hirsch-Pugh-Robinson \cite{hps1977}, these properties are $C^1$-robust. Thus $\Psi'$ is partially hyperbolic with invariant foliations
\[F^{cs}(\Psi') \text{ tangent to }E^c(\Psi') \oplus E^s(\Psi') \qquad\text{and}\qquad F^{cu}(\Psi') \text{ tangent to }E^c(\Psi') \oplus E^u(\Psi')\] 
with dense leaves. Now note that we have the following identifications. 
\[F^{cs}(\Psi',P') = W^s(P',\Psi') \qquad\text{and}\qquad F^{cu}(\Psi',Q') = W^u(Q',\Psi') \]
Indeed, the uniqueness of local invariant manifolds near hyperbolic invariant sets \cite[Thm 4.1(b)]{hps1977} implies equality near $P'$ and $Q'$, and then global invariance implies global equality. In particular, $W^s(P',\Psi')$ and $W^u(Q',\Psi')$ are dense. Moreover, by (\ref{eq:stable_closure}) $W^s(Q',\Psi')$ is also dense.

\vspace{3pt}

\vspace{3pt}

Now \cite[Lem 7.3]{bonatti2004dynamics} states that $\Psi'$ is robustly transitive (and in fact mixing). We recall the argument. Let $U$ and $V$ be neighborhoods in $Y$. Then
\[U \cap W^s(Q',\Psi') \neq \emptyset \qquad\text{and}\qquad V \cap W^u(Q',\Psi') \neq \emptyset\]
since these invariant manifolds are dense. This implies that $(\Psi')^j(U) \cap (\Psi')^{-k}(V)$ is non-empty for all sufficiently large $j$ and $k$. In other words, $\Psi'$ is topologically transitive. \end{proof}

\section{Contact Blender Construction} \label{sec:blender_construction} In this section, we describe the construction of the various objects in Theorem \ref{thm:heteroclinic_contact_blender}. We will use the notation of Setup \ref{set:blender_setup} throughout this section. We also fix the following notation.

\begin{notation}[Reeb Intervals] \label{not:Reeb_intervals} Let $\gamma:[0,T] \to Y$ be an embedded segment of a Reeb trajectory with end-points $x = \gamma(0)$ and $y = \gamma(T)$. Then we use the shorthand
\[
[x,y] = \gamma[0,T] \subset Y
\]
\end{notation}

We begin the section by constructing certain contact chart near a segment of the closed loop $\Gamma$ in Setup \ref{set:blender_setup} (Section \ref{subsec:standard_chart}). We then use this chart to construct the family of perturbations $\Psi$ appearing in Theorem \ref{thm:heteroclinic_contact_blender} (Section \ref{subsec:family_of_contactomorphisms}) and prove various useful properties of the family (Section \ref{subsec:properties_of_the_family}). Finally, we describe the family of blender boxes (Definitions \ref{def:blender_box} and \ref{def:blender}) corresponding to the maps in the family $\Psi$ (Section \ref{subsec:family_of_boxes}).

\subsection{Standard Chart} \label{subsec:standard_chart} We start by constructing a particular standard rectangular neighborhood of the segment within the fixed circle $\Gamma \subset Y$ from the point $Q$ to the point $P$ (see Setup \ref{set:blender_setup}). 

\vspace{3pt}

\begin{lemma}[Standard Chart] \label{lem:standard_chart} Let $V \subset Y$ be an open neighborhood of $\Gamma$. Then there is a smoothly embedded cube $U \subset V$, constants $L,\delta,\epsilon > 0$ and local coordinates
\[
(s,t,u):U \simeq [-3,3]^n_s \times [-\delta,L+\delta]_t \times [-\epsilon,\epsilon]^n_u
\]
satisfying the following properties.

\begin{enumerate}[label=(\alph*)]
    \item The contact form is given by $\alpha|_U = dt - (s_1du_1 + \dots + s_ndu_n)$.
    \vspace{3pt}
    \item The Reeb vector-field $R$ of $\alpha$ is given by $R|_U = \partial_t$.
    \vspace{3pt}
    \item The points $P$ and $Q$ are given by $(0_s,L_t,0_u)$ and $(0_s,0_t,0_u)$ respectively. 
    \vspace{3pt}
    \item The local stable and unstable manifolds $W_{\on{loc}}^s(\Gamma,\Phi;U)$ and $W_{\on{loc}}^u(\Gamma,\Phi;U)$ of $\Gamma \cap U$ are given by
    \[[-3,3]^n_s \times [-\delta,L-\delta]_t \times 0_u \qquad \text{and}\qquad 0_s \times [-\delta,L-\delta]_t \times [-\epsilon,\epsilon]_u\]
    \item The stable and unstable bundles $E^s(\Phi)$ and $F^u(\Phi)$ satisfy
    \[TF^s(\Phi) = T\R^n_s \text{ on }W_{\on{loc}}^s(\Gamma,\Phi;U) \qquad\text{and}\qquad  TF^u(\Phi) = T\R^n_u \text{ on }W_{\on{loc}}^u(\Gamma,\Phi;U) \]
    \item The bundle $T\R^n_s = E^s_{\on{sm}}(\Phi)$ is invariant under $\Phi$ and uniformly contracted by $\Phi$ on $U$.
    \vspace{3pt}
    \item The heteroclinic orbit $\chi$ in Setup \ref{set:blender_setup} contains the point
    \[a = (1_s,0_t,0_u) \qquad\text{where }1_s = (1,0,\dots,0) \in [-3,3]^n_s\]
    \item There is an integer $k > 0$ such that
    \[ \Phi^{-k}(a) \not\in U\]
\end{enumerate}
 \end{lemma}

\begin{proof} We construct this chart in two steps: the construction of a nice transverse hypersurface and the construction of a good chart using that hypersurface. 

\vspace{3pt}

{\bf Step 1: Transverse Hypersurface.} In this step, we construct a nice embedded hypersurface. We shrink $V$ so that the smooth unstable bundle $E^s_{\on{sm}}(\Phi)$ is defined on $V$ (see Setup \ref{set:blender_setup}). Consider the foliations $F^u(\Phi)$ and $F^s_{\on{sm}}(\Phi)$ corresponding to the unstable bundle $F^u(\Phi)$ and the smooth stable bundle $E^s_{\on{sm}}(\Phi)$. These foliations are Legendrian (Lemma \ref{lem:phcontactomorphism_properties}) and $F^u(\Phi)$ is H\"{o}lder with smooth leaves (Theorem \ref{thm:invariant_foliations}).

\vspace{3pt}

Next, choose a ball $\Lambda_u$ in the leaf of $F^u(\Phi;Q)$ containing $Q$. Since $F^u_{\on{sm}}(\Phi)$ is smooth, we may choose an embedded codimension $1$ surface that is contained in the union of the leaves of $F^s_{\on{sm}}(\Phi)$ intersecting $\Lambda_u$. Denote this surface by
\[\Sigma \subset Y \qquad\text{with}\qquad \Lambda_u \subset \Sigma\]
This surface is transverse to the Reeb vector-field, and is thus symplectic with symplectic form
\[d\alpha|_\Sigma \qquad\text{with primitive $\alpha|_\Sigma $ vanishing on the leaves of $F^s(\Phi)$ and $F^u(\Phi)$}\]  
By flowing $\Sigma$ slightly by the Reeb flow, we acquire an embedding as follows (for small $\delta > 0$).
\begin{equation} \label{eq:lem:standard_chart:1}
\iota:(-\delta,\delta)_t \times \Sigma \to Y \qquad\text{with}\qquad \iota^*\alpha = dt + \alpha|_\Sigma
\end{equation}

Next, apply the Weinstein neighborhood theorem for Lagrangian foliations \cite[Thm 7.1]{w1971} to the smooth Legendrian foliation $F^s_{\on{sm}}(\Phi)$ near the transverse smooth Lagrangian $\Lambda_u \subset \Sigma$. After shrinking $\Sigma$, this yields a symplectic embedding
\begin{equation} \label{eq:lem:standard_chart:2}
\jmath:\Sigma \to T^*F^u(\Phi;Q) \qquad\text{with}\qquad F^s_{\on{sm}}(\Phi) \cap \Sigma = \jmath^*F_{\on{std}}
\end{equation}
Here $F_{\on{std}}$ is the standard Lagrangian foliation of $T^*F^u(\Phi;Q)$ by cotangent fibers. Finally, consider the pullback $\lambda_{\on{std}}$ of the standard Liouville form on $T^*\Lambda_u$ by (\ref{eq:lem:standard_chart:2}). Since $\lambda_{\on{std}}$ and $\alpha|_\Sigma$ vanish on $\Lambda_s$ and $\Lambda_u$, we can find a primitive $f$ such that
\[\lambda_{\on{std}} - \alpha|_\Sigma = df \]
The primitives $\alpha|_\Sigma$ and $\lambda_{\on{std}}$ vanish on $L_u$ and on the foliation $F^s(\Phi) \cap \Sigma$, so $f$ is constant these manifolds. It follows that $f$ is constant on $\Sigma$ and so
 \begin{equation}  \label{eq:lem:standard_chart:3}
 \lambda_{\on{std}} = \alpha|_\Sigma  
 \end{equation}

 {\bf Step 2: Cube Coordinates.} In this step, we use $\Sigma$ to construct a rectangular chart and verify the requirements. Since $\chi$ is a homoclinic orbit asymptotic to $Q$ in the forward direction, we may choose a point $a \in \chi$ such that
 \[a \in  \Sigma \cap F^s_{\on{sm}}(\Phi;Q) \subset F^s(\Phi;Q)\]
 Here we use the fact that $Q \in \Gamma$ so that $F^s_{\on{sm}}(\Phi;Q) = F^s(\Phi;Q) \subset W^s(\Gamma,\Phi)$ by Setup \ref{set:blender_setup}(c). Choose a smooth embedding of a cube of the form
 \begin{equation} \label{eq:lem:standard_chart:4}
 [-\epsilon,\epsilon]_u^n \to \Lambda_u \subset F^u(Q) \qquad\text{with}\qquad 0_u \mapsto Q 
 \end{equation}
 This extends naturally to a map of cotangent bundles, giving a Liouville embedding
 \[ \kappa:([-3,3]_s^n \times [-\epsilon,\epsilon]_u^n,\lambda_{\on{std}}) \to (\Sigma,\alpha|_\Sigma) \qquad\text{where}\qquad \lambda_{\on{std}} = -\sum_i s_i \cdot du_i\]
 By shrinking $\Sigma$ around $\Lambda_s$, we may assume that this map is a symplectomorphism. Using the Reeb flow as in (\ref{eq:lem:standard_chart:1}), we may extend $\kappa$ to a local contactomorphism
 \[
\kappa:\big([-3,3]_s^n \times [-\delta ,L+\delta]_t \times [-\epsilon ,\epsilon]_u\; , \; dt + \lambda_{\on{std}}\big) \to (Y\; , \; \alpha)
 \]
 The restriction of $\kappa$ to $0_s \times [-\delta,L+\delta]_t \times 0_u$ is a parametrization of a sub-arc of $\Gamma$ containing $Q$ and $P$ sending $0$ to $Q$ and $L$ to $P$. Thus by choosing a sufficiently small surface $\Sigma$ and choosing $\delta,\epsilon$ small, we may guarantee that $\kappa$ is an embedding and the image of $\kappa$ lies in $V$.  

 \vspace{3pt}

Now take $U$ to be the image of $\kappa$ and take $(s,t,u)$ as the coordinates induced by $\kappa$. By construction $U \subset V$. We now check that $U$ and $(s,t,u)$ satisfy (a-g). For (a), we note that
\[
\kappa^*\alpha = dt + \lambda_{\on{std}} = dt - (s_1 du_1 + \dots + s_n du_n)
\]
The requirements (b) follows immediately from (a). Requirements (c) and (f,g) follow trivially from the construction. To see property (d), note that the local stable manifold $W^s_{\on{loc}}(\Gamma,\Phi;U)$ of $\Gamma$ is precisely the $R$-orbit of the component of the local stable leaf $F^s_{\on{loc}}(\Phi,Q;U)$ containing $Q$. By construction, $\kappa$ maps $[-3,3]_s^n \times 0_t \times 0_u$ to $F^s_{\on{loc}}(\Phi,Q;U)$. It follows that $W^s(\Gamma,\Phi;U)$ is identified with the zero set
\[\{u = 0\} = [-3,3]_s^n \times [-\delta,L+\delta]_t \times 0_u\]
An analogous discussion applies to $W^u_{\on{loc}}(\Gamma,\Phi;U)$, verifying (d). Requirement (e) follows from the same discussion, and the fact that $F^s(\Phi)$ and $F^u(\Phi)$ are $R$-invariant. Finally, by shrinking the neighborhood $V$ in the construction, we may guarantee that the orbit $\chi$ contains a point that is not in $V$ or $U$. It follows that $\Phi^{-k}(a) \not\in U$ for some $k$. This verifies (h).\end{proof}

We will require an enhancement of Lemma \ref{lem:standard_chart} that incorporates a set of invariant cone-fields.

\begin{lemma}[Standard Chart With Cone-Fields] \label{lem:standard_chart_w_cones} For any $\mu > 1$ and $\epsilon \in (0,1)$, there exists
\begin{itemize}
    \item An integer $N \ge 1$.
    \vspace{3pt}
    \item A cube $U \subset Y$ with coordinates $(s,t,u)$ as in Lemma \ref{lem:standard_chart}.
    \vspace{3pt}
    \item A Riemannian metric $g$ on $Y$ that is compatible with the splitting $TY = E^s \oplus E^c \oplus E^u$
    \vspace{3pt}
    \item Continuous cone-fields $K^s$ and $K^u$ on $Y$.
\end{itemize}
that satisfy the following properties (after possibly rescaling the contact form $\alpha$).

\begin{enumerate}[label=(\alph*)]
    \item The cone-fields $K^s$ and $K^u$ are compatible with the blender box $U$ (see Definition \ref{def:compatible_cones}).
    \[T\R^n_s \subset K^s \quad\text{and}\quad  T\R^n_u \subset K^u\]
    \item The cone-fields $K^s$ is contracted by $\Phi^{-N}$, and $K^{cu}$ and $K^u$ are contracted by $\Phi^N$.
    \[
    \Phi_*^{-N}(K^s) \subset \on{int} K^s  \quad\text{and}\quad \Phi_*^N(K^u) \subset \on{int} K^u
    \]
    \item The cone-fields $K^s$ and $K^u$ are dilated by $\Phi^{-N}$ and $\Phi^N$ with contant of dilation $\mu$, respectively.
    \[\mu \cdot |T\Phi^N(v)| < |v| \text{ if }v \in K^u \qquad\text{and}\qquad \mu \cdot |T\Phi^{-N}(v)| < |v| \text{ if }v \in K^s\]
    \item The cone-fields $K^s$ and $K^u$ are width less than $\epsilon$ for both $g$ and the standard metric on $U$.
\end{enumerate}\end{lemma}

\begin{proof} Choose an auxilliary chart $U'$ and coordinates $(s',t',u')$ as in Lemma \ref{lem:standard_chart}. Also choose a continuous metric $g$ on $Y$ that is compatible with the splitting $E^s \oplus E^c \oplus E^u$. Finally, choose a $\delta < \epsilon$ sufficiently small such that the cone-fields
\[K^s = K_\delta(E^s(\Phi)) \qquad K^{cu} = K_\delta(E^c(\Phi) \oplus E^u(\Phi)) \qquad K^u = K_\delta(E^u(\Phi))\]
are width less than $\epsilon$ with respect to $g$. By construction, these cone-fields satisfy (d). By Theorem \ref{thm:invariant_cone_field} and Setup \ref{set:blender_setup}, we may choose an $N \ge 1$ such that $\Phi$ satisfies (b) and (c).

\vspace{3pt}

To achieve (a), note that $E^s(\Phi), E^c(\Phi) \oplus E^u(\Phi)$ and $E^u(\Phi)$ agree with the tangent spaces $T\R^n_s, T(\R^1_t \times \R^n_u)$ and $T\R^n_u$ along $\Gamma$, respectively. Thus in a small neighborhood $V$ of $\Gamma$, the chosen cone-fields satisfy (a). We may then rescale the coordinates $(s',t',u')$ by taking
\[
(s,t,u) = (c \cdot s', c \cdot t', c \cdot u') \qquad\text{for $c$ large}
\]
For $c$ sufficiently large, the cube with $(s,t,u)$-coordinates $[2,-2]_s \times [-\delta,L+\delta]_t \times [-\epsilon,\epsilon]_u^n$ will lie within $V$. Properties (a-b) are preserved after we scale the contact form by $c^{-1}$. Properties (c-f,h) in Lemma \ref{lem:standard_chart} are preserved by this coordinate change. Property (g) in in Lemma \ref{lem:standard_chart} can be preserved by replacing $a$ with $\Phi^j(a)$ for large $j$.\end{proof}

\subsection{Family of Contactomorphisms} \label{subsec:family_of_contactomorphisms} Our next task is to construct the family of contactomorphisms $\Psi$ in Theorem \ref{thm:heteroclinic_contact_blender}. For the rest of Sections \ref{sec:blender_construction} and \ref{sec:proof_of_blender_axioms}, fix
\[\text{constants $L$, $N$, $k$} \qquad U \subset Y\text{ with coordinates }(s,t,u) \qquad \text{metric $g$} \quad\text{and}\quad \text{cone-fields }K^\bullet\]
as in Lemmas \ref{lem:standard_chart} and \ref{lem:standard_chart_w_cones}. We can now begin the main construction of the family of maps. We will require two auxilliary contact Hamiltonians for the construction.

\begin{construction}[Hamiltonian $H$] \label{con:Hamiltonian_H} Let $H:U \to \R$ be a contact Hamiltonian such that
\[H(s,t,u) = h(t) \]
Here $h:\R \to \R$ is a smooth function of the $t$-variable satisfying the following constraints. 
\[h(t) = t \quad \text{if}\quad t \le L/3 \qquad\text{and}\qquad h(t) = L - t \quad \text{if}\quad t \ge 2L/3\]
\[h(t) > 0 \quad\text{and}\quad |h'(t)| < 1/L \qquad \text{if} \quad 0 < t < L 
\]
The contact Hamiltonian vector-field $V_H$ generating the flow of contactomorphisms $\Phi^H$ is
\begin{equation} \label{eqn:form_of_V_H}
V_H = h(t) \cdot \partial_t - h'(t) \cdot \sum_i s_i \cdot \partial_{s_i} \qquad\text{ in the chart }U
\end{equation}
In particular, we have the following formulas for special ranges of $t$.
\begin{equation} \label{eq:form_of_VH_in_ranges}
V_H = t\partial_t - \sum_i s_i \partial_{s_i}\;\;\text{if $t \le L/3$}\qquad\qquad V_H = (L-t)\partial_t + \sum_i s_i \partial_{s_i}\;\;\text{if $t \ge 2L/3$}\end{equation}
\noindent Using (\ref{eqn:form_of_V_H}), it is a simple calculation to show that $\Phi^H$ takes the following general form on $U$.
\begin{equation} \label{eq:form_of_Phi}
\Phi^H_r(s,t,u) = (f_r(t) \cdot s,\psi_r(t),u) \quad\text{where}\quad \partial_r\psi  = h \circ \psi \quad\text{and}\quad \partial_r f = -h' \cdot f
\end{equation}
In particular, we have the following formulas for $\Phi^H$ on the ranges of $t$ appearing in (\ref{eq:form_of_VH_in_ranges}). 
\begin{equation} \label{eq:form_of_Phi_in_range_13} \Phi^H_r(s,t,u) = (e^{-r}s,e^rt,u)\;\;\text{if $t \le 1/3$}\end{equation}
\begin{equation} \label{eq:form_of_Phi_in_range_23} \Phi^H_r(s,t,u) = (e^{r}s,L + e^{-r}(t-L),u) \;\;\text{if $t \ge 2L/3$}\end{equation}
\end{construction}

\begin{construction}[Hamiltonian $G$] \label{con:Hamiltonian_G} We define the (time-dependent) contact Hamiltonian
\[G:[0,1]_r \times U \to \R\]
as follows. Recall that the heteroclinic point $a$ in Lemma \ref{lem:standard_chart} is in the unstable manifold of $P$. Moreover, the local unstable manifold of $P$ in $U$ is the set $0_s \times L_t \times [-\delta,\delta]^n_u$. Thus we choose 
\begin{equation} \label{eqn:choice_of_m}
m \ge 2 \qquad\text{such that}\qquad \Phi^{-Nm}(a) = (0_s,L_t,x_u) \text{ for some }x_u \in [-\delta,\delta]^n_u 
\end{equation}
 Choose a neighborhood $W$ of $\Phi^{-Nm}(a)$ that is small enough so that the sets
\[
(\Phi^H_r \circ \Phi)^j(W) \qquad\text{for }j = 0 \dots m
\]
are all disjoint for sufficiently small $r$. Since $\Phi^{-Nm}(a)$ and $a$ are in $U$, we may also assume that
\[W \subset U \qquad\text{and}\qquad (\Phi^H_r \circ \Phi)^{Nm}(W) \subset U\]
for sufficiently small $r$. Moreover, since $\Phi^{-k}(a)$ is not in $\Phi^{-1}(U)$ for some $k$, we have that
\[
(\Phi^H_r \circ \Phi)^{Nm-k}(W) \cap \Phi^{-1}(U) = \emptyset
\]
for sufficiently small $r$ and neighborhood $W$. We may thus fix a pair of open sets $W' \subset W''$ that have the following properties for sufficiently small $r$.
\[
(\Phi^H_r \circ \Phi)^{Nm-k}(W) \subset W' \qquad W'' \cap \Phi^{-1}(U) = \emptyset\]
\[W'' \cap (\Phi^H_r \circ \Phi)^j(W) = \emptyset \quad \text{if}\quad 0 \le j \le Nm \quad \text{and}\quad j \neq Nm-k
\]
Now we define a smooth family of contactomorphism $\Phi^K_r:Y \to Y$ (generated by a parameter-dependent contact Hamiltonian $K$) implicitly by the following equation. 
\begin{equation} \label{eq:lem:construction_of_Psi_1}
\Phi^R_r \circ \Phi^{Nm} \circ \Phi^H_{Nmr} = (\Phi^H_r \circ \Phi)^k \circ G_r \circ (\Phi^H_r \circ \Phi)^{Nm-k} 
\end{equation}
Here $\Phi^R$ is the Reeb flow of $Y$. Note that $\Phi^K_0 = \on{Id}$. We now define $G$ so that it satisfies
\[
G_r = K_r \;\text{ on }\; W' \qquad\text{and}\qquad G_r = 0 \;\text{ on }\; Y \setminus W'' 
\] \end{construction}

\begin{construction}[Family Of Maps] \label{con:family_Psi} We define the family of contactomorphisms
\[\Psi:[0,1]_r \times Y \to Y \qquad\text{as the composition}\qquad \Psi_r = \Phi^G_r \circ \Phi^H_r \circ \Phi\]
Note that $\Psi_r$ satisfies the following elementary properties for $W$ and $m$ as in Construction \ref{con:Hamiltonian_G}. \begin{equation} \label{eq:basic_properties} \Psi_r = \Phi^H_r \circ \Phi \;\;\text{on} \;\;\Phi^{-1}(U) \qquad\text{and}\qquad \Psi_r^{Nm}(x) = \Phi^R_r \circ \Phi^{Nm} \circ \Phi^H_{Nmr} \;\;\text{on} \;\; W\end{equation}\end{construction}

\begin{remark}[Role Of Perturbations] The two perturbations $\Phi^H$ and $\Phi^G$ play two different roles in the above construction. The perturbation $\Phi^H$ breaks the family of fixed points $\Gamma$ into two fixed points $P$ and $Q$ of coindex one. The perturbation $\Phi^G$ is essentially designed to guarantee (\ref{eq:lem:construction_of_Psi_1}), which imposes the existence of the heteroclinics between $P$ and $Q$, among other properties.
\end{remark}

\begin{remark}[Error Terms] \label{rmk:error_terms} There is a subtle difficulty arising in the construction of $\Psi$ here that is not present in the smooth setting. Namely, the vector field $V_H$ in (\ref{eqn:form_of_V_H}) generating the perturbation $\Phi^H$has a both a term tangent to $\Gamma$, which is needed to break $\Gamma$ into fixed points $P$ and $Q$, and an extra error term normal to $\Gamma$. Since we must work with contact vector fields, this error term cannot be removed. Moreover, such an error term is not harmless in general, since it could cause the blender axiom in Definition \ref{def:blender}(d) to fail. Indeed, after large iterations of $\Psi_r$, even a small error term could grow large enough to ruin the cone invariance property in Definition \ref{def:blender}(d).

\vspace{3pt}

This difficulty is resolved by the existence of the smoothly integrable, invariant and uniformly contracted Legendrian sub-bundle $E^s_{\on{sm}}(\Phi)$, assumed in Setup \ref{set:blender_setup}. Since the foliation tangent to $E^s_{\on{sm}}(\Phi)$ is smooth, we are able to choose the chart in Section \ref{subsec:standard_chart} and the Hamiltonian $H$ in Construction \ref{con:Hamiltonian_H} carefully so that the unavoidable error term in $V_H$ lies entirely within the uniformly contracted sub-bundle $T\R^n_s$. This allows us to control the term when demonstrating the axioms Definition \ref{def:blender}(c-d) in Section \ref{subsec:axioms_C_and_D} (and specifically in Lemma \ref{lem:axiom_d_part_2}).
\end{remark}

\subsection{Properties Of The Family} \label{subsec:properties_of_the_family} In this section, we prove several properties of the family $\Psi$ in Construction \ref{con:family_Psi}. We will use these properties extensively in the proof of the blender properties.

\vspace{3pt}

We start by recording the effect of the maps $\Psi$ on the coordinates in the standard chart $U$.

\begin{lemma}[Coordinate Projections] \label{lem:coordinate_projections} Let $\pi_W$ and $\pi_\Gamma$ be the coordinate projections
\[\pi_W:U \to W^u_{\on{loc}}(\Gamma,\Phi;U) = U \cap \{s = 0\} \qquad\text{given 
by}\qquad \pi_W(s,t,u) = (0_s,t,u)\]
\[\pi_\Gamma:U \to \Gamma \cap U = U \cap \{s,u = 0\} \qquad\text{given 
by}\qquad \pi_\Gamma(s,t,u) = (0_s,t,0_u)\]
Then $\pi_W$ and $\pi_\Gamma$ commute with $\Psi_r$ over $U \cap \Phi^{-1}(U)$. In particular, if $\psi$ is the flow of $h \cdot \partial_t$ on  $\R_t$ then
\[
t(\Psi_r(x)) = t(\pi_\Gamma \circ \Psi_r(x)) = t(\Psi_r \circ \pi_\Gamma(x)) = \psi_r(t(x))
\]
\end{lemma}

\begin{proof} On the set $\Phi^{-1}(U)$, we have $\Psi_r = \Phi^H_r \circ \Phi$ by the first identity in (\ref{eq:basic_properties}). By Lemma \ref{lem:standard_chart}(d-f) and by examination of (\ref{eq:form_of_Phi}), both $\Phi$ and $\Phi^H_r$ preserves the foliation
\begin{equation} \label{eq:lem:coordinate_contraction:1}
E^s_{\on{sm}}(\Phi) = T\R^n_s \text{ on }U \qquad\text{and}\qquad T\R^n_u \text{ on }W^u_{\on{loc}}(\Gamma,\Phi;U)
\end{equation}
Moreover, both $\Phi$ and $\Phi^H_r$ (and therefore $\Psi_r$) preserve the sets
\[
W^u_{\on{loc}}(\Gamma,\Phi;U) = U \cap \{s = 0\} \qquad\text{and}\qquad \Gamma \cap U = U \cap \{s,t = 0\}
\]
This $\Psi_r$ sends the leaf of $T\R^n_s$ through $x \in U \cap \{s=0\}$ to the leaf through $\Psi_r(x)$. In other words, $\pi_W$ commutes with $\Psi_r$. Similarly, $\Psi_r$ preserves the leaves of $T\R^n_u$ on $U \cap \{s = 0\}$ and so
\[\pi_\Gamma \circ \Psi_r(x) = \Psi_r \circ \pi_\Gamma(x) \qquad\text{if }x \in U \cap \{s = 0\}\]
Since $\pi_\Gamma \circ \pi_W = \pi_\Gamma$, this implies that $\pi_\Gamma$ commutes with $\Psi_r$ in general. The final claim follows since $t \circ \Psi_r = \psi_r \circ t$ on $\Gamma \cap U$, since $\Phi$ fixes $\Gamma$ pointwise and $t \circ \Phi^H_r = \psi_r \circ t$ on $\Gamma \cap U$.  \end{proof}

\begin{lemma}[Coordinate Contraction] \label{lem:coordinate_contraction} Let $U' = U \cap \dots \cap \Psi^{-N}(U)$. Then for small $r$, we have
\[\mu \cdot |s(\Psi_r^N(x))| < |s(x)| \qquad\text{and}\qquad  \mu \cdot |u(x)| < |u(\Psi_r^N(x))|\]
\end{lemma}

\begin{proof} Note that $T\R^n_s$ is uniformly contracted by $\Phi^N$ by a factor of $\mu$ on $U$ and $T\R^n_u$ is uniformly expanded by $\Phi^N$ by a factor of $\mu$ on $W^u_{\on{loc}}(\Gamma,\Phi;U)$ due to Lemma \ref{lem:standard_chart_w_cones}.  

\vspace{3pt}

These properties are robust in the $C^1$-topology, so the map $\Psi_r$ also satisfies these properties for small $r$. Since $x$ and $\pi_\Gamma(x)$ lie in the same fiber of $T\R^n_s$, we can apply Lemma \ref{lem:coordinate_projections} to see that
\[|s(\Psi^N_r(x))| = |s(\Psi^N_r(x)) - s(\Psi^N_r(\pi_W(x))| \le \mu^{-1} \cdot |s(x) - s(\pi_W(x))| = \mu^{-1} \cdot |s(x)|\]
Also, the maps $u,\pi_\Gamma$ and $\pi_W$ satisfy $u \circ \pi_W = u$ and $u \circ \pi_\Gamma = 0$. Thus by Lemma \ref{lem:coordinate_projections}, we have
\[
|u(\Psi^N_r(x))| = |u \circ \pi_W(\Psi^N_r(x)) - u \circ \pi_\Gamma(\Psi^N_r(x))| \ge \mu \cdot |u(\pi_W(x)) - u(\pi_\Gamma(x))| \ge \mu \cdot |u(x)| \qedhere
\]
 \end{proof}

 \vspace{3pt}

Next, note that the maps in the family $\Psi$ are partially hyperbolic for sufficiently small parameter. This follows from Theorems \ref{thm:persistence_of_splitting}-\ref{thm:invariant_foliations} and the partial hyperbolicity of $\Phi$.

\begin{lemma}[Partial Hyperbolicity] \label{lem:partial_hyperbolicity} The maps $\Psi_r$ (for sufficiently small $r$) has a dominated splitting
    \[TY = E^s(\Psi_r) \oplus E^c(\Psi_r) \oplus E^u(\Psi_r) \quad\text{with stable/unstable foliations}\quad F^s(\Psi_r) \text{ and }F^u(\Psi_r)\]
\end{lemma}

\noindent Also note that the points $P$ and $Q$ become non-degenerate and hyperbolic fixed points of $\Psi_r$.

\begin{lemma}[Fixed Points] \label{lem:fixed_points} The points $Q$ and $P$ are hyperbolic fixed points of $\Psi_r$ of index
\[
\on{ind}(\Psi_r;Q) = n \qquad\text{and}\qquad \on{ind}(\Psi_r;P) = n+1 
\]
The local stable and unstable manifolds of $P$ and $Q$ in $U$ with respect to  $\Psi_r$ are given by
\begin{equation*}
W^s_{\on{loc}}(P,\Psi_r;U) = [-3,3]^n_s \times (0,L+\epsilon]_t \times 0_u \qquad\quad W^u_{\on{loc}}(P,\Psi_r;U) = 0_s \times L_t \times  [-\delta,\delta]_u \end{equation*}
\begin{equation*} \hspace{28pt} W^s_{\on{loc}}(Q,\Psi_r;U) = [-3,3]^n_s \times 0_t \times 0_u \hspace{71pt}
W^u_{\on{loc}}(Q,\Psi_r;U) = 0_s \times [-\epsilon,L)_t \times [-\delta,\delta]_u
\end{equation*}\end{lemma}

\begin{proof} The points $P$ and $Q$ are fixed by $\Phi$ due to Setup \ref{set:blender_setup}(d) and fixed by $\Phi^H_r$ due to (\ref{eq:form_of_Phi_in_range_13}-\ref{eq:form_of_Phi_in_range_23}). Thus (\ref{eq:basic_properties}) implies that $P$ and $Q$ are fixed by $\Psi_r$. To compute the index of $P$, note that $T\Phi$ and $T\Phi^H_r$ at $P$ decomposes via the splitting $TY = E^s(\Phi) \oplus E^c(\Phi) \oplus E^u(\Phi)$ as follows.
\[
T_P\Phi = T_P\Phi^s \oplus \on{Id}_c \oplus T_P\Phi^u \qquad\text{and}\qquad T_P\Phi^H_r = e^{r}\on{Id}_s \oplus \; e^{-r}\on{Id}_c \oplus \on{Id}_u
\]
The linear maps $T_P\Phi^s$ and $T_P\Phi^u$ are the restrictions to $E^s(\Phi)$ and $E^u(\Phi)$ respectively. Since $P$ is in $U$, the first formula in (\ref{eq:basic_properties}) implies that the differential of $\Psi_r$ at $P$ is given by
\[
T_P\Psi_r = e^{r}T_P\Phi^s_r \;\oplus \; e^{-r}\on{Id}_c \; \oplus \; T_P\Phi^u_r
\]
Since $\Gamma$ is a normally hyperbolic fixed set of $\Phi$ by Setup \ref{set:blender_setup}, the maps $e^rT_P\Phi^s_r$ (for small $r$) and $T_P\Phi^u_r$ have real eigenvalues of norm bounded above by $1$ and below by $1$, respectively. It follows that $T_P\Psi_r$ will be hyperbolic of index $n+1$. The same analysis applies to the index of $Q$.

\vspace{3pt}

To find the stable and unstable manifolds, fix $x \in U$. By Lemma \ref{lem:coordinate_projections}, we know that $t(\Psi_r^k(x)) = \psi^k_r(t(x))$ where $\psi_r$ of the vector-field $h \cdot \partial_t$ and $H$ is as in Construction \ref{con:Hamiltonian_H}. Moreover $h > 0$ on $(0,L)$ and $h < 0$ on $[-\epsilon,0)$ and $(L,L+\epsilon]$. This implies that $t(\Psi_r^k(x))$ converges to
\[
L \text{ if }t(x) \in (0,L+\epsilon] \qquad 0 \text{ if }t(x) = 0
\]
and that $\Psi_r^k(x)$ leaves $U$ if $t(x) \in [-\epsilon,0)$. Now Lemma \ref{lem:coordinate_contraction} implies that $|u(\Psi^k_r(x))|$ diverges if $u(x) \neq 0$, and so $\Psi^k_r(x)$ leaves $U$. On the other hand, if $u(x) = 0$ then $|s(\Psi^k_r(x))| \to 0$ as $k \to \infty$ by Lemma \ref{lem:coordinate_contraction}. This shows that
\[W^s_{\on{loc}}(P,\Psi_r;U) = [-3,3]^n_s \times (0,L+\epsilon]_t \times 0_u \qquad W^s_{\on{loc}}(Q,\Psi_r;U) = [-3,3]^n_s \times 0_t \times 0_u\]
An identical analysis works for the unstable manifolds. \end{proof}

Next, we have the following description of the local stable and unstable foliations, and the action of $\Psi_r$ on the leaves in the chart $U$.

\begin{lemma}[Stable/Unstable Foliations] \label{lem:stab_unstab_foliations} The local stable and unstable foliations of $\Psi_r$ satisfy
\[F^s(\Psi_r) = T\R^n_s \;\; \text{on}\;\;W_{\on{loc}}^s(P,\Psi_r;U) \qquad\text{and}\qquad F^u(\Psi_r) = T\R^n_u \;\; \text{on}\;W_{\on{loc}}^u(Q,\Psi_r;U)\]
Moreover, if $\Lambda^s$ and $\Lambda^u$ are leaves of $F^s(\Psi_r)$ and $F^u(\Psi_r)$ intersecting $\Gamma \cap U$, respectively, then
\[
\Psi_r(\Lambda^s) \cap U = \Phi^H_r(\Lambda^s) \cap U \qquad\text{and}\qquad \Psi_r(\Lambda^u) \cap U = \Phi^H_r(\Lambda^u) \cap U
\]
\end{lemma}

\begin{proof} For the first claim, note that $T\R^n_s$ and $T\R^n_u$ are tangent to $W^s_{\on{loc}}(P,\Psi_r;U)$ and $W^u_{\on{loc}}(Q,\Psi_r;U)$ by Lemma \ref{lem:fixed_points}. Moreover, $T\R^n_s$ and $T\R^n_u$ are uniformly contracted and expanded, respectively, on those sets via (\ref{eq:basic_properties}). The lemma then follows from the uniqueness of the strong stable and unstable foliations on the stable and unstable manifolds of a hyperbolic invariant set \cite[\S 4]{hps1977}.  The second claim follows from the first claim, the formula (\ref{eq:basic_properties}) and the fact that $\Phi$ preserves the stable and unstable foliations in $U$.
\end{proof}

\noindent Finally, we have the following computation of a certain important family of heteroclinics.

\begin{lemma}[Heteroclinics] \label{lem:homoclinic_points} Consider the Reeb segment $[b_{-r},b_r]$ containing the point $a$, where
\[
b_\tau = (1_s,\tau,0_u) \qquad\text{and we recall that}\qquad a = (1_s,0,0_u) 
\]
Then for $r > 0$ sufficiently small, we have
\begin{enumerate}[label=(\alph*)]
    \item The leaf of $F^u(\Psi_r)$ through $(0_s,L-\delta,0_u)$ meets $[b_{-r},b_r]$ at the point $(1_s,r-\delta,0_u)$ for $\delta \in [0,2r]$.\vspace{3pt}
    \item The interval $(a,b_r)$ is a connected component of the intersection of $W^s(P,\Psi_r)$ and $W^u(Q,\Psi_r)$.
    \vspace{3pt}
    \item The points $a$ and $b_r$ are transverse homoclinic points of $Q$ and $P$, respectively.
    
\end{enumerate}
\end{lemma}

\begin{proof} Let $m$ and $W$ be the integer and fixed ($r$-independent) neighborhood in Construction \ref{con:family_Psi}. Note that $W$ is a fixed neighborhood of the point
\[
\Phi^{-Nm}(a) = (0,L_t,1_u)
\]We may thus choose $r$ small enough so that we have the following inclusion for all $\delta \in [0,2r]$.
\[\Phi_{-Nmr}^H(0_s,L-\delta,1_u) \in W\]
We first compute the image of $(0_s,L-\delta,1_u)$ under $\Psi_r^{Nm} \circ \Phi^H_{-Nmr}$. By (\ref{eq:basic_properties}), we see that
\[
\Psi_r^{Nm} \circ \Phi_{-Nmr}^H(0_s,L-\delta,1_u) = \Phi^R_r \circ \Phi^{Nm} \circ \Phi^H_{Nmr} \circ \Phi^H_{-Nmr}(0_s,L-\delta,1_u) = \Phi^R_r \circ \Phi^{Nm}(0_s,L-\delta,1_u)
\]
Since the Reeb vector-field commutes with $\Phi$, the latter expression becomes
\[
\Phi^R_r \circ \Phi^{Nm}(0_s,L-\delta,1_u) = \Phi^R_r \circ \Phi^{Nm} \circ \Phi^R_{-\delta}(0_s,L,1_u) = \Phi^R_{r - \delta} \circ \Phi^m(0_s,L,1_u) 
\]
Since $(1_s,0,0_u) = \Phi^{Nm}(0_s,L,1_u)$ by (\ref{eqn:choice_of_m}) in Construction \ref{con:family_Psi}, we then acquire the equality
\[
\Phi^R_{r - \delta} \circ \Phi^{Nm}(0_s,L,1_u) = \Phi^R_{r - \delta}(1_s,0,0_u) = (1_s,r-\delta,0_u) 
\]
We thus acquire the following formula that we will shortly use to prove (a-c).
\begin{equation} \label{eq:lem:homoclinic_points:1}
\Psi^{Nm}_r \circ \Phi^H_{-Nmr}(0_s,L-\delta,1_u) = (1_s,r-\delta,0_u)
\end{equation}

Now we prove the claims above. For (a), note that by Lemma \ref{lem:stab_unstab_foliations}, $\Psi^{Nm}_r$ maps the leaf $\Lambda$ of $F^u(\Psi_r)$ containing $\Phi^H_{-Nmr}(0_s,L-\delta,1_u)$ to the leaf $\Psi^{Nm}_r(\Lambda) = \Phi^H_{Nmr}(\Lambda)$ containing the points
\[
\Psi^{Nm}_r \circ \Phi^H_{-Nmr}(0_s,L-\delta,1_u) = (1_s,r-\delta,0_u)\]
\[\Phi_{Nmr}^H \circ \Phi^H_{-Nmr}(0_s,L-\delta,1_u) = (0_s,L-\delta,1_u)
\]
By Lemma \ref{lem:stab_unstab_foliations}, the leaf that contains $(0_s,L-\delta,1_u)$ also contains $(0_s,L-\delta,0_u)$, proving (a). For (b), we note that by (a), the point $(0_s,r-\delta,0_u)$ is contained in the stable manifold $W^s(P,\Psi_r)$ by Lemma \ref{lem:fixed_points}. On the other hand, it is also contained in the leaf of in $F^u(\Psi_r)$ passing through $(1_s,r-\delta,0_u)$ by (a), and this leaf is contained in the unstable manifold $W^u(Q,\Psi_r)$ by Lemma \ref{lem:fixed_points}. This proves (b). For (c), we argue similarly to (b). By (a), the point $a = (1_s,0_t,0_u)$ is contained in the leaf of $F^u(\Psi_r)$ containing $(0_s,L-r,0_u)$, which is contained in the unstable manifold $W^u(Q,\Psi_r)$ by Lemmas \ref{lem:fixed_points}. On the other hand, $a$ is contained in $W^s(Q,\Psi_r)$ (also by Lemma \ref{lem:fixed_points}). Thus it is a homoclinic point for $Q$. A similar discussion holds for $b_r$ and $P$.\end{proof}

\subsection{Family Of Boxes} \label{subsec:family_of_boxes} We conclude this section by describing the family of blender boxes
\[B_r(Q)\]
appearing in the blenders in Theorem \ref{thm:heteroclinic_contact_blender}. We start by introducing notation for several important points and quantities appearing in the description of the boxes.

\begin{notation} \label{not:m_r} For $r \in [0,1]$, let $P_r$ and $Q_r$ denote the points in $U$ given by
\[
P_r = (0_s,L-r,0_u) \qquad \text{and}\qquad Q_r = (0_s,r \cdot (1 - e^{-8Nr}), 0_u)\]
Then we define the constant $m_r$ to be the unique integer satisfying  
\begin{equation} \label{eq:def_of_mr} \Psi_r^{Nm_r}(Q_r) \in [\Psi^{5N}_r(P_r),\Psi^{6N}_r(P_r)]\end{equation}
\end{notation}

\begin{lemma} \label{lem:growth_of_mr} The integers $m_r$ diverge to $\infty$ as $r \to 0$. Precisely, there are constants $C,r_0 > 0$ such that
\[m_r > -Cr^{-1} \cdot \log(r) \qquad\text{for all }r \in (0,r_0)\]
\end{lemma}

\begin{proof} Since $\Psi_r$ restricts to the flow of $\psi_r$ on $\Gamma$ (see Construction \ref{con:Hamiltonian_H} or Lemma \ref{lem:coordinate_projections}), we may equivalently characterize $m_r$ by
\[\psi_{rNm_r}(r(1 - e^{-8Nr})) \in [L-re^{-5Nr},L-re^{-6Nr}]\]
Here $\psi_r:\R_t \to \R_t$ is the flow of the vector-field $h \cdot \partial_t$ and $h$ is as in Construction \ref{con:Hamiltonian_H}. Briefly switch notation by letting $\psi_r(t) = \psi(r,t)$, and let the quantities $T_{\on{st}}(r), T_{\on{mid}}$ and $T_{\on{end}}(r)$ be the unique values such that
\[\psi(T_{\on{st}}(r),r(1 - e^{-8Nr})) = L/3 \qquad \psi(T_{\on{mid}},L/3) = 2L/3 \qquad \psi(T_{\on{end}}(r),2L/3) = L-re^{-5Nr} \]
Note that $T_{\on{mid}}$ is independent of $r$. Moreover, using the formulas (\ref{eq:form_of_Phi_in_range_13}) and (\ref{eq:form_of_Phi_in_range_23}) on the intervals $[0,L/3]_t$ and $[2L/3,L]_t$ (see Construction \ref{con:Hamiltonian_H}), we can compute that
\[
T_{\on{st}}(r) = \log(L/3) - \log(r) - \log(1 - e^{-8Nr})
\]
\[
T_{\on{end}}(r) = -\log(3/2) - \log(r) + 5Nr
\]
Now we simply note that in the $r \to 0$ limit, we have
\[
rNm_r \sim (T_{\on{st}}(r) + T_{\on{mid}} + T_{\on{end}}(r)) \ge -3\log(r)
\]
Here $\sim$ denotes that the limit of the ratio is $1$. This proves the result.\end{proof}

We can now introduce the definition of the blender box $B_r(Q)$. 

\begin{construction}[Blender Box For $\Psi$] \label{con:smooth_boxes} The blender box $B_r(Q) \subset U$ is defined as the set
\[B_r(Q) = D^n_s(2) \times D^1_t(t(Q_r)) \times D^n_u(l \cdot \mu^{-m_r})\]
The constants in the formula are defined as follows.
\begin{itemize}
\item The quantity $t(Q_r)$ is the $t$-coordinate $r \cdot (1 - e^{-8Nr})$ of the point $Q_r$.
\vspace{3pt}
\item The constant $l$ is a positive number such that $l/3$ is larger than the minimum radius of a ball in the unstable leaf $F^u(\Psi_r;P)$ containing $a$.
    \vspace{3pt}
    \item The constant $\mu$ is an $r$-independent constant of dilation for $\Psi_r$ via Lemma \ref{lem:partial_hyperbolicity}.
    \vspace{3pt}
    \item The constant $m_r$ is the integer defined by the formula (\ref{eq:def_of_mr}).
\end{itemize}
The stable, central, unstable, left and right boundaries are defined as in Definition \ref{def:blender_box}. We will also need an auxilliary H\"{o}lder continuous box constructed using unstable leaves. Let
\begin{equation} \label{eq:def_of_SrQ} S_r(Q) = D^n_s(2 + \mu^{-m_r/2}) \times D^1_t(t(\Psi^N_r(Q_r))) \times 0_u\end{equation}
The H\"{o}lder box $\tilde{B}_r(Q)$ is defined as the union of disks in the leaves of the unstable foliation $F^u(\Psi_r)$ that have boundary on the set $\{|u| = l \cdot \mu^{-m_r}\}$ and that intersect $S_r(Q)$ . That is
\begin{equation} \label{eq:def_of_BrQ}
\tilde{B}_r(Q) = \big\{x \; : \; F^u(\Psi_r,x) \cap S_r(Q) \neq \emptyset \quad \text{and}
\quad |u(x)| \le l \cdot \mu^{-m_r}\big\}
\end{equation}
There is a map $\pi:\tilde{B}_r(Q) \to S_r(Q)$ mapping a point $x$ to the intersection point $\pi(x) \in F^u(\Psi_r,x) \cap S_r(Q)$. The following lemma allows us to replace $\tilde{B}_r(Q)$ with $B_r(Q)$ in some arguments.
\end{construction} 

\begin{lemma} The H\"{o}lder box $\tilde{B}_r(Q)$ contains the blender box $B_r(Q)$ for sufficiently small $r$. 
\end{lemma}

\begin{proof} Let $x \in B_r(Q)$. By Construction \ref{con:smooth_boxes} we know that $|u(x)| \le l \cdot \mu^{-m_r}$. Let $D \subset F^u(\Psi_r,x)$ be the disk given by $F^u(\Psi_r,x) \cap \{|u(x)| \le l \cdot \mu^{-m_r}\}$. This disk is tangent to the vertical cone-field $K^u$ in $U$ (see Lemma \ref{lem:standard_chart_w_cones}), and thus is the graph of a Lipschitz map
\[
f:D_u^n(l \cdot \mu^{-m_r}) \to \R^n_s \times \R^1_t
\]
with a uniform Lipschitz constant $C$ independent of $r$ (see Lemma \ref{lem:vertical_disks_graphs}) and so the image of $f$ in $\R_s \times \R_t$ has diameter bounded by $C \cdot l \cdot \mu^{-m_r}$. Thus if $y = f(0_u)$ then
\[|s(y) - s(x)| < 5C \cdot l \cdot \mu^{-m_r} \qquad \text{and}\qquad |t(y) - t(x)| < 5C \cdot l \cdot \mu^{-m_r}\]
In particular, this implies that for sufficiently small $r$, we have
\[|s(y)| \le |s(x)| + 5C \cdot l \cdot \mu^{-m_r} \le 2 + \mu^{-m_r/2}\]
\[|t(y)| \le |t(x)| + 5C \cdot l \cdot \mu^{-m_r} \le t(Q_r) + 5C \cdot l \cdot \mu^{-m_r} \le t(\Psi_r(Q_r))\]
The last inequality follows from the fact that, for $r$ small, we have
\[t(\Psi_r(Q_r)) - t(Q_r) = r(e^{9Nr} - e^{8Nr}) \ge \frac{1}{2}r^2\] 
Thus $D$ intersects $S_r(Q)$ at $y \times 0_u$ and $F^u(\Psi_r,x) \cap S_r(Q)$ is non-empty. In particular, $x$ is in $S_r(Q)$ and $B_r(Q) \subset \tilde{B}_r(Q)$. \end{proof}

\section{Proof Of Blender Axioms And Theorem \ref{thm:heteroclinic_contact_blender}} \label{sec:proof_of_blender_axioms} In the previous section, we constructed the family of maps $\Psi$ and the accompanying family of blender boxes appearing in Theorem \ref{thm:heteroclinic_contact_blender}. Our objective in this section is to prove that the pair
\[(\Psi_r^N,B_r(Q)) \qquad\text{for sufficiently small }r\]
satisfies the axioms of a stable blender given in Definition \ref{def:blender}. We will also prove Theorem \ref{thm:heteroclinic_contact_blender}. 

\vspace{3pt}

We begin by proving the axioms of a blender stated in Definition \ref{def:blender}(a-f) in order (Sections \ref{subsec:axiom_A}-\ref{subsec:axiom_E_F}) and conclude with the proof of Theorem \ref{thm:heteroclinic_contact_blender} (Section \ref{subsec:conclusion_of_thm}).

\begin{remark} We warn the reader that the proofs in the coming section are the lengthy and notation heavy.  
\end{remark}

\begin{construction}[Useful Holonomy] \label{con:useful_holonomy} The following holonomy map will be useful in the proofs below. Recall that $a$ and $P$ lie on a leaf of the unstable foliation $F^u(\Phi)$ of $\Phi$, with transversals
\begin{equation} \label{eq:holonomy_transversals_for_blender}
S = T = [-2,2]^n_s \times [-\epsilon,L+\epsilon]_t \times 0_u
\end{equation}
These transversals will play the role of the source and target transversals of the holonomy map (Definition \ref{def:holonomy}). By Lemma \ref{lem:holder_holonomy}, we can choose a neighborhood $\on{Nbhd}(P)$ and a H\"{o}lder constant $\kappa$ so that the corresponding holonomy maps $\on{Hol}_{\Psi_r}$ from $\on{Nbhd}(P) \cap S$ to $T$ are H\"{o}lder with H\"{o}lder constant $\kappa$ for sufficiently small $r$. For small $r$, this restricts to a holonomy map
\begin{equation} \label{eq:useful_holonomy}
\on{Hol}_{\Psi_r}:D^n_s(2 \cdot \mu^{-m_r}) \times [L-2r,L]_t \times 0_u \to [-2,2]^n_s \times [-\delta,L+\delta]_t \times 0_u
\end{equation}
Lemma \ref{lem:homoclinic_points}(a) can be restated as the following formula.
\begin{equation} \label{eq:holonomy_of_Gamma} \on{Hol}_{\Psi_r}(0_s,L-\delta,0_u) = (1_s,r-\delta,0_u) \qquad\text{for any }\delta \in [0,2r]\end{equation}
    
\end{construction} 

\subsection{Axiom A} \label{subsec:axiom_A} We start by giving the definition of the subset $A$ in Definition \ref{def:blender} and explaining the proof of the axiom in  Definition \ref{def:blender}(a). Both are fairly short and straightward.

\begin{definition}[Blender Set $A$] We let $A_r(Q) \subset B_r(Q)$ be the connected component of the intersection $B_r(Q) \cap \Psi^N(B_r(Q))$ that contains the point $Q = (0_s,0_t,0_u)$. 
\end{definition}

\begin{lemma}[Blender Axiom A] \label{lem:blender_axiom_A} The intersection $B_r(Q) \cap \Psi_r^N(B_r(Q))$ is disjoint from \[\partial^s B_r(Q) \qquad \Psi^N_r(\partial^c B_r(Q))\quad \text{and}\quad \Psi_r^N(\partial^u B_r(Q)) \qquad\text{for sufficiently small }r\] In particular, the connected component $A_r(Q)$ of $B_r(Q) \cap \Psi_r^N(B_r(Q))$ satisfies these properties. \end{lemma}

\begin{proof} We start by noting that for sufficiently small $r$, we have
\begin{equation} \label{eq:blender_axiom_A:1} B_r(Q) \subset U \cap \Psi_r^{-1}(U) \cap \dots \cap \Psi_r^{-N}(U)\end{equation}
Now fix points $x \in \partial^s B_r(Q)$ and $y \in \Psi^N_r(\partial^u B_r(Q))$. By Lemma \ref{lem:coordinate_contraction} and Construction \ref{con:smooth_boxes}, we know that the coordinates of $x$ and $y$ satisfy
\[|s(x)| = 2 \qquad\text{and}\qquad |u(y)| > \mu \cdot l \cdot \mu^{-m_r}\]
Also by Lemma \ref{lem:coordinate_contraction} and Construction \ref{con:smooth_boxes}) we know that any $z \in B_r(Q) \cap \Psi^N_r(B_r(Q))$ satisfies 
\[
|s(z)| < 2 \cdot \mu^{-1} \qquad\text{and}\qquad |u(z)| = l \cdot \mu^{-m_r}
\]
It follows that $B_r(Q) \cap \Psi^N_r(B_r(Q))$ is disjoint from $\partial^s B_r(Q)$ and $\Psi^N_r(\partial^u B_r(Q))$. For the central boundary, note that $B_r(Q)$ consists of points where $t \le L/3$ for small $r$. Thus by Lemma \ref{lem:coordinate_contraction}
\[
|t(\Psi^N_r(w))| = |\psi_{Nr}(t(w))| = e^{Nr} \cdot r \cdot (1 - e^{-8Nr})
 \qquad\text{ for any }w \in \partial^c B_r(Q) \]
On the other hand, any $z \in B_r(Q)$ has $|t(z)| \le r \cdot (1 - e^{-8Nr})$ by Construction \ref{con:smooth_boxes}. Thus $B_r(Q)$ and $\Psi^N_r(\partial^c B_r(Q))$ are disjoint, and the proof is finished.
\end{proof}

\subsection{Axiom B} \label{subsec:axiom_B} Next, we define the subset $A'$ in Definition \ref{def:blender} and prove the axiom in Definition \ref{def:blender}(b). This is more involved than the previous axiom. We require the following lemma.

\begin{lemma} \label{lem:heteroclinic_is_in_set} The heteroclinic point $a \in \chi$ in the heteroclinic orbit $\chi$ in Setup \ref{set:blender_setup} is contained in
\[B_r(Q) \cap \Psi^{N m_r}(B_r(Q))\]
\end{lemma}

\begin{proof} By the construction of the integer $m_r$ (see Notation \ref{not:m_r}), we know that
\[\Psi^{Nm_r}_r(Q_r) \in \big[\Psi_r^5(P_r),\Psi_r^6(P_r)\big] \qquad\text{(using Notation \ref{not:Reeb_intervals})}\]
Since $\Psi_r$ and $\Phi^H_r$ agree on the Reeb segment $\Gamma \cap U$, and $(\Phi^H_r)^{N m_r}$ maps the interval $\Gamma \cap B_r(Q)$ to an interval in $\Gamma$ containing $[0,\Psi^5_r(P_r)]$, we thus deduce that  
\[
P_r \in \Psi^{Nm_r}([-Q_r,Q_r]) \qquad\text{where}\qquad \pm Q_r = 0_s \times \pm r (1 - e^{-8Nr}) \times 0_u
\]
Let $x \in [-Q_r,Q_r]$ be the point with $\Psi^{N m_r}(x) = P_r$ and let $D \subset B_r(Q)$ be the disk
\[D = 0_s \times x \times D^n_u(l \cdot \mu^{-m_r}) \qquad\text{contained in}\qquad  F^u(\Psi_r,x) \cap B_r(Q)\]
This disk is a disk of radius $l \cdot \mu^{-m_r}$ in $F^u(\Psi_r,x)$. Since $\Psi^N_r$ uniformly expands distances in leaves of $F^u$ (see Lemma \ref{lem:standard_chart_w_cones}), the disk $\Psi^{Nm_r}(D)$ has radius larger than $l$ in $F^u(\Psi_r,P_r)$. It follows from the definition of $l$ that
\[a \in \Psi^{Nm_r}(D) \subset \Psi^{N m_r}(B_r(Q))\qedhere\]\end{proof}

\begin{definition}[Blender Set $A'$] \label{def:blender_set_Aprime} The sets $A'_r(Q)$ and $\tilde{A}_r(Q)$ are the components of the intersections
\[B_r(Q) \cap \Psi^{Nm_r}(B_r(Q)) \qquad\text{and}\qquad \tilde{B}_r(Q) \cap \Psi^{Nm_r}(\tilde{B}_r(Q))\]that contain the point $a = (1_s,0_t,0_u)$, respectively. Note that $A_r(Q) \subset \tilde{A}_r(Q)$. 
\end{definition}

We next begin working towards the proof of the corresponding axiom, Definition \ref{def:blender}(b). We need some technical lemmas about the holonomy map (see Construction \ref{con:useful_holonomy}). Consider the set
\[C_r = D^n_s(5 \cdot \mu^{-m_r}) \times [L-2r,L]_t \times 0_u\]
The holonomy from Construction \ref{con:useful_holonomy} is well-defined on $C_r$ for small $r$, yielding a smooth map
\[\on{Hol}_{\Psi_r}:C_r \to [-3,3]^n_s \times [-\epsilon,L + \epsilon]_t \times 0_u \qquad\text{for small }r\]
Our first goal is to analyze the image of $C_r$ under the holonomy map.

\begin{lemma}[Holonomy Estimates] \label{lem:holonomy_estimates} There is a $\kappa > 0$ independent of $r$ with the following property. Fix
\[x \in C_r \qquad\text{and}\qquad y = \on{Hol}_{\Psi_r}(x)\]
Then for sufficiently small $r$, the coordinates of $y$ satisfy
\[|s(y)| \le \mu^{-\kappa m_r} \qquad \text{and}\qquad |t(y) - (r + t(x) - L)| \le \mu^{-\kappa m_r}\]\end{lemma}

\begin{proof} If $x$ has $s(x) = 0$, then (\ref{eq:holonomy_of_Gamma}) says that $s(y) = 0$ and $t(y) = r + t(x) - L$. In general, let $x' = \pi_\Gamma(x)$ be the projection of $x$ to the $t$-axis and let $y' = \on{Hol}_{\Psi_r}(x')$.  Then by Construction \ref{con:useful_holonomy}
\[\on{dist}(y',y) \le \on{dist}(x',x)^\kappa = |s(x)|^\kappa \le (5l\mu^{-m_r})^\kappa\]
Here $\kappa$ is the H\"{o}lder constant in Construction \ref{con:useful_holonomy}. The implies that
\[
|s(y)| = |s(y) - s(y')| \le (5 l\mu^{-m_r})^\kappa \qquad \text{and}\qquad |s(y) - (r + t(x) - L)| \le (5 l  \mu^{-m_r})^\kappa
\]
Since $m_r$ grows faster than $1/r$, we can eliminate the factor of $5l$ by shrinking $\kappa$. \end{proof}

\noindent Next let $\Sigma_r$ be the union of the disks $D$ in the unstable foliation satisfying the following properties.
\[\partial D \subset \{|u| \le l \cdot \mu^{-m_r}\} \qquad\text{and}\qquad D \cap \{u = 0\} \in \on{Hol}_{\Psi_r}(\partial C_r)\]

\begin{lemma} \label{lem:holonomy_wall} The set $\Sigma_r$ is disjoint from $\tilde{B}_r(Q) \cap \Psi^{Nm_r}(\tilde{B}_r(Q))$. \end{lemma}
\begin{proof} Fix $z \in \Sigma_r$ and let $D$ be the corresponding unstable disk of radius $l \cdot \mu^{-m_r}$, centered at $y = \on{Hol}_{\Psi_r}(x)$ where $x \in \partial C$. Recall that $S_r(Q)$ is the intersection of $\tilde{B}_r(Q)$ with $\{u = 0\}$, given by
\[
S_r(Q) = D^n_s(2 + \mu^{-m_r/2}) \times D^1_t(t(\Psi_r^N(Q_r))) \times 0_u
\]
Note that $P$ is connected to $\on{Hol}_\Phi(a)$ by a path of length less than $l/3$ in the unstable leaf $F^u(\Phi,P)$. Moreover, $\on{Hol}_{\Psi_r}$ and $C_r$ converge to $\on{Hol}_\Phi$ and $P$ as $r \to 0$. Therefore $\on{Hol}_{\Psi_r}(x)$ and $x$ are connected by a path of length less than $l/2$ in $F^u(\Psi_r,x)$ for small $r$, and $z$ is connected to $x$ by a path in $\Lambda$ of length less than $l$. Since $\Psi^{Nm_r}(\tilde{B}_r(Q))$ contains the unstable disks of radius $l$ around $\Psi^{Nm_r}(S_r(Q))$, it follows that
\[z \in \Psi^{Nm_r}(\tilde{B}_r(Q)) \qquad\text{if and only if}\qquad x \in \Psi^{Nm_r}(S_r(Q))\]

We now claim that the intersection $\Psi^{Nm_r}(S_r(Q)) \cap \partial C_r$ is contained in the set
\begin{equation} \label{eq:holonomy_wall:1}
D_s^n(5 \cdot l \cdot \mu^{-m_r}) \times \{L - 2r\}_t \times 0_u  
\end{equation}
Indeed, any point in $S_r(Q)$ satisfies $|s(x)| \le 2 + \mu^{-m-r/2}$ and $t(x) \le t(\Psi_r^N(Q_r))$. Therefore it follows from Lemmas \ref{lem:coordinate_projections}-\ref{lem:coordinate_contraction} and the construction of $m_r$ that
\[
|s(\Psi^{Nm_r}(x))| \le (2+\mu^{-m_r/2}) \cdot \mu^{-m_r} < 5\mu^{-m_r}\]
\[t(\Psi^{Nm_r}(x)) \le t(\Psi^{N(m_r+1)}(Q_r)) \le t(\Psi^{7N}(P_r)) = L - re^{7Nr} < L
\]
By the definition of $C_r$, this implies that $x$ is in $\Psi^{Nm_r}(S_r(Q)) \cap \partial C_r$ only if $x$ is in the set (\ref{eq:holonomy_wall:1}).

\vspace{3pt}

Finally, we claim that if $x$ is in (\ref{eq:holonomy_wall:1}), then $z$ is disjoint from $\tilde{B}_r(Q)$ for small $r$. Indeed, by Lemma \ref{lem:holonomy_estimates}, we know that
\[
t(y) \le r + t(x) - L + \mu^{-\kappa m_r} = L - r + \mu^{-\kappa m_r} < -r(1 - r^{-9Nr}) = -t(\Psi_r^N(Q_r))
\]
On the other hand, if $z$ is in $\tilde{B}_r(Q)$, then $y$ must be contained in $S_r(Q)$ which only consists of points $y$ with $t(y) \ge -t(\Psi^N_r(Q_r))$. This shows that $z$ must be disjoint from $\tilde{B}_r(Q) \cap \Psi^{Nm_r}(\tilde{B}_r(Q))$, concluding the proof. \end{proof}

We now apply Lemmas \ref{lem:holonomy_estimates} and \ref{lem:holonomy_wall} to acquire some estimates on $A'_r(Q)$ and $\tilde{A}_r(Q)$. 

\begin{lemma}[Bounds On $\tilde{A}$] \label{lem:bounds_on_Atilde} There is a $\kappa > 0$ so that for any point $z$ in $\tilde{A}_r(Q)$ and small $r$, we have
\begin{equation} \label{eq:lem:bounds_on_Atilde:1}
|s(z) - 1_s| \le \mu^{-\kappa m_r} \qquad -t(Q_r) \le t(z) \le r(1 - e^{-6Nr}) + \mu^{-\kappa m_r} \qquad |u(z)| \le l \cdot \mu^{-m_r}
\end{equation}
\end{lemma}

\begin{proof} The bound on $u(z)$ and the lower bound on $t$ follow since $\tilde{A}_r(Q) \subset \tilde{B}_r(Q)$.

\vspace{3pt}

For the remaining bounds, let $V_r \subset U$ denote the tube $[-2,2]^n_s \times [-\delta,L+\delta] \times D^n_u(l \cdot \mu^{-m_r})$. Note that $\Sigma_r \subset V_r$ and $V_r \setminus \Sigma_r$ consists of two components $V^{\on{in}}_r$ and $V^{\on{out}}_r$ where
\[
V^{\on{in}}_r = \big\{z \in D \; : \;D \text{ is unstable disk with $D \cap S_r(Q) \in \on{Hol}_{\Psi_r}(\on{int}(C_r))$ and $\partial D \subset \{|u| = l \cdot \mu^{-m_r}\}$}\big\} 
\]
We will not need a description of $V^{\on{out}}_r$. Since $\tilde{B}_r(Q) \subset V_r$ by construction, $\tilde{A}_r(Q) \subset V_r$ as well. By Lemma \ref{lem:holonomy_wall}, $\tilde{A}_r(Q)$ must be in one of the components $V^{\on{in}}_r$ and $V^{\on{out}}_r$. A direct computation gives $a \in V_r^{\on{in}}$, so it follows from Definition \ref{def:blender_set_Aprime} that $\tilde{A}_r(Q) \subset V^{\on{in}}_r$. 

\vspace{3pt}

It therefore suffices to prove the remaining bounds for a point $z$ in $V^{\on{in}}_r$. Let $D$ be the unstable disk through $z$ and let $y = \on{Hol}_{\Psi_r}(x)$ be the intersection of $D$ with $D_r(Q)$. By Lemma \ref{lem:holonomy_estimates}
\[|s(y)| \le \mu^{-\kappa m_r} \qquad\text{and}\qquad t(y) \le L - r + t(x) \le L - r + t(\Psi^7_r(P_r)) = r(1 - e^{7Nr})\]
The point $z$ lies on the vertical disk $D$, which is a graph of a $2\epsilon$-Lipschitz graph from $D^n_u(l \cdot \mu^{-m_r})$. It follows that $\on{dist}(z,y) \le C \mu^{-m_r}$ for some $C$ independent of $r$, and so
\[
|s(x)| \le \mu^{-\kappa m_r} + C\mu^{-m_r} \qquad t(y) \le r(1 = e^{7Nr}) +  C\mu^{-m_r}
\]
The remaining estimates follow by taking $r$ small and possibly shrinking $\kappa$. \end{proof}

The axiom in Definition \ref{def:blender}(b) is now an easy consequence of Lemma \ref{lem:bounds_on_Atilde}.

\begin{lemma}[Blender Axiom B] \label{lem:blender_axiom_B} The regions $\tilde{A}_r(Q)$ and $A'_r(Q)$ is disjoint from
\[\partial^r B_r(Q) \qquad \partial^sB_r(Q) \quad\text{and}\quad \Psi_r^N(\partial^u B_r(Q))\]
\end{lemma}

\begin{proof} Disjointness from $\Psi^N(\partial^u B_r(Q))$ follows from the same argument as in Lemma \ref{lem:blender_axiom_A}. For the other two boundary regions, note that by Lemma \ref{lem:bounds_on_Atilde} we have
\[
|s(x) - 1_s| \le \mu^{-\kappa m_r} \quad \text{and}\quad t(x) \le r(1 - e^{-7Nr}) + \mu^{-\kappa m-r} \quad\text{for all }x \in \tilde{A}_r(Q)
\]
This implies that for sufficiently small $r$, the set $\tilde{A}_r(Q)$ is disjoint from the sets
\[
\partial^sB_r(Q) = B_r(Q) \cap \{|s| = 2\} \qquad\text{and}\qquad \partial^r B_r(Q) = B_r(Q) \cap \{t = r \cdot (1 - e^{-8Nr})\} \qedhere
\]
\end{proof}

\subsection{Axioms C And D} \label{subsec:axioms_C_and_D} Next, we prove the blender axioms related to cone-fields (see Definition \ref{def:blender}(c-d)). The first of these axioms is relatively straightforward.

\begin{lemma}[Blender Axiom C] \label{lem:blender_axiom_C} There are compatible cone-fields $K^s$ and $K^u$ for $B_r(Q)$ of width less than $\epsilon$ (with respect to the standard metric) that are contracted and dilated (with constant $\mu$) as follows.
\[(\Psi_r^{-N})_*K^s \subset \on{int} K^s \qquad (\Psi_r^N)_*K^u \subset \on{int} K^u  \qquad \Psi_r^{-N} \text{ dilates }K^s \qquad \Psi_r^N \text{ dilates }K^u
\]
\end{lemma}

\begin{proof} Let $K^s$ and $K^u$ be the cone-fields in Lemma \ref{lem:standard_chart_w_cones}. Note that $\Psi_r \to \Phi$ in the $C^\infty$-topology. Moreover, contraction and dilation (with constant $\mu$) of a given cone-field are $C^1$-robust properties. Thus, Lemma \ref{lem:standard_chart_w_cones} implies this axiom for our choice of constants $N,\mu,\epsilon$ and small $r$. \end{proof}

\noindent The second cone-field related axiom (Definition \ref{def:blender}(d)) is much more complicated and delicate. This is essentially due to the presence of the additional error term discussed in Remark \ref{rmk:error_terms}, which emerge inevitably in the contact setting. We will prove Definition \ref{def:blender}(d) in Lemma \ref{lem:blender_axiom_D} after several preliminary lemmas. To start, choose a Riemannian metric
\[
g \text{ on }TY \qquad\text{such that the splitting }E^u(\Psi_r) \oplus T\R^1_t \oplus T\R^n_s\text{ is orthogonal on $U$}
\]
We consider the following cone-fields of width $\epsilon$ and $\delta$ with respect to $g$.
\[K^{us}_\epsilon = K_\epsilon E^u(\Psi_r) \cap (E^u(\Psi_r) \oplus T\R^n_s) = \{u + v \; : \; u \in E^u(\Psi_r) \text{ and }v \in T\R^n_s \text{ with }|u| \ge \epsilon \cdot v\}\]
\[K^{cs}_\delta = K_\delta(T\R^1_t) \cap (T\R^1_t \oplus T\R^n_s) = \{u + v \; : \; u \in T\R^1_t \text{ and }v \in T\R^n_s \text{ with }|u| \ge \delta \cdot v\}\]
These cone-fields are well-defined over $U$ (i.e. wherever $(s,t,u)$-coordinates are well-defined). Finally, we define the fiberwise sum of cones
\[
K^{cu}_{\delta,\epsilon} = K^{us}_\epsilon + K^{cs}_\delta
\]
The following key lemma describes choices of the parameters of the cone $K^{cu}$ that guarantee contraction and dilation.

\begin{lemma}[Stretching $K^{cu}$] \label{lem:stretching_Kcu} Fix a positive integer $k$, positive constants $\mu,\epsilon,\nu,\eta$ , and a subset $V \subset U$ with the following properties.
\begin{itemize} 
    \item The constants satisfy $\mu^2 > 1 + \epsilon^2$ and $\nu > 1 > \eta$.
    \vspace{3pt}
    \item $T\Psi^{k}_r$ dilates $E^u(\Psi_r)$ with constant $\mu$ and $T\Psi^{k}_r$ dilates $T\R^n_s$ with constant $\mu^{-1}$ over $V$. 
    \vspace{3pt}
    \item $T\Psi^{k}_r(R) = \nu \cdot R + w$ where $w \in E^s(\Phi) = T\R^n_s$ and $|w| \le \eta \cdot |R|$ over the subset $V$, where $R = \partial_t$.
\end{itemize}
Then $K^{cu}_{\epsilon,\delta}$ is contracted and uniformly dilated by $\Psi_r^k$ over the subset $V$ for
\[\frac{\eta}{1 - \mu^{-1}} < \delta < \sqrt{\nu^2 - 1}\]\end{lemma}

\begin{proof} To prove that $K^{cu}_{\epsilon,\delta}$ is uniformly dilated with some positive constant of dilation, we write an arbitrary vector $v$ in $K^{cu}_{\epsilon,\delta}$ as follows.
\[
v = (v^u + v^s) + (a  R + w^s) \qquad\text{where}\qquad |v^u| \ge \epsilon |v^s| \text{ and }|aR| = a \ge \delta  |w^s|
\]
We then compute the norm of the image of $v$ under $\Psi^{Nk}$.
\[|\Psi^k_r(v)|^2 = |\Psi^k_r(v^u)|^2 + |\Psi^k_r(a R)|^2 + |\Psi^k_r(v^s + w^s)|^2\]
\[\ge \mu^2 \cdot |v^u|^2 + \nu^2 \cdot |aR|^2 + |aw|^2 + \mu^{-2} \cdot |v^s + w^s|^2 \ge  \mu^2 \cdot |v^u|^2 + \nu^2 \cdot |aR|^2\]
Now since $|v^u| \ge \epsilon |v^s|$ and $|aR| \ge \delta |w^s|$, we see that
\[
(1 + \epsilon^2) \cdot |v^u|^2 \ge |v^u + v^s|^2\qquad\text{and}\qquad (1 + \delta^2)\cdot |aR|^2 \ge |aR + w^s|^2 
\]
Finally, we calculate that
\[
|\Psi^{Nk}_r(v)|^2 \ge \frac{\mu}{1 + \epsilon^2} \cdot |v^u + v^s|^2 + \frac{\nu^2}{1 + \delta^2} \cdot |aR|^2 \ge \on{min}(\frac{\mu}{1 + \epsilon^2},\frac{\nu^2}{1 + \delta^2}) \cdot |v|^2
\]
Since $\mu^2 > 1 + \epsilon^2$ and $\nu^2 > 1 + \delta^2$ by assumption, we thus find that $v$ is uniformly expanded.

To prove that $K^{cu}_{\delta,\epsilon}$ is contracted, we argue as follows. First, note that for any $v = v^u + v^s$ in $K^{us}_\epsilon$ with $|v^u| \ge \epsilon |v^s|$, we have
\[
\Psi^k_r(v) = \Psi^k_r(v^u) + \Psi^k_r(v^s) \qquad\text{where}\qquad \Psi^k_r(v^u) \in E^u(\Psi_r) \text{ and }\Psi^k_r(v^s) \in T\R^n_s
\]
By our hypothesis on the dilation of $T\Psi_r^k$, we know that
\[|T\Psi^k_r(v^u)| \ge \mu^2 \cdot \epsilon \cdot |T\Psi^k_r(v^s)| > \epsilon \cdot |T\Psi^k_r(v^s)|\]
Therefore $T\Psi_r^k(v)$ is strictly contracted by $T\Psi_r^k$. Likewise, take any vector $v = aR + v^s$ in $K^{us}_\delta$ with $|aR| \ge \delta \cdot |v^s|$. Then
\[
T\Psi^k_r(v) = a\nu R + aw + \Psi^k_r(w^s)
\]
Now we note that we have the following estimate.
\[
|aw + \Psi^k_r(w^s)| \le |aw| + |\Psi^k(w^s)| \le \eta |aR| + \epsilon \cdot \mu^{-1} \cdot |aR| \le (\eta + \delta \cdot \mu^{-1}) \cdot |aR|
\]
By assumption we have $\eta + \delta \cdot \mu^{-1} < \delta$ and thus $T\Psi^k_r$ contracts the cone $K^{us}_\delta$. We have thus proven that
\[
T\Psi^k_r(K^{cu}_{\delta,\epsilon}) \subset \on{int} K^{cu}_{\delta,\epsilon} \qedhere
\]\end{proof}

We next verify the third criterion in Lemma \ref{lem:stretching_Kcu} in the cases relevant to our axiom. Recall that we have fixed constants $N,m$ (see the beginning of Section \ref{subsec:family_of_contactomorphisms}).

\begin{lemma}[Axiom D, Part 1] \label{lem:axiom_d_part_1} For sufficiently small $r$, we have
\[T\Psi^N_r(R)  = e^{Nr} \cdot R \qquad\text{over the subset }A_r(Q)\]
\end{lemma}

\begin{proof} For sufficiently small $r$, we have
\[\Psi^j_r(B_r(Q)) \subset [-2,2]_s \times [-\delta,L/3]_t \times [-\epsilon,\epsilon] \text{ for all }j = 0,\dots,N\]
Here $\delta,\epsilon,L$ are the parameters of the chart in Lemma \ref{lem:standard_chart}. In this region, we know that $\Psi_r = \Phi^H_r \circ \Psi$. Therefore by Construction \ref{con:Hamiltonian_H} (and more specifically (\ref{eq:form_of_Phi_in_range_13})) we find that
\[T\Phi^H_r(R) = e^r R \quad\text{and}\quad T\Psi(R) = R \qquad\text{and therefore}\quad T\Psi_r^N(R) = e^{Nr} \cdot R \qedhere\]\end{proof}

For the other part of Axiom D, we require the following lemma tracking the behavior of the set $\tilde{A}_r(Q)$ under $\Psi^{-1}_r$.

\begin{lemma}[$\tilde{A}$ Stays In $U$] \label{lem:Atilde_stays_in_U} Let $m$ and $W$ be as in Construction \ref{con:Hamiltonian_G} and (\ref{eqn:choice_of_m}). Then
\[\tilde{A}_r(Q) \subset \Psi^{Nm}_r(W) \qquad\text{and}\qquad \Psi^{-j}(\tilde{A}_r(Q)) \subset U \text{ for all }j = Nm,\dots,Nm_r \qquad\text{for small $r$}\]
\end{lemma}

\begin{proof} For the first claim, note that $N$ and $m$ are independent of $r$ and $\Psi_r \to \Phi$ in $C^\infty$ as $r \to 0$. Thus we may choose an $r$-independent open neighborhood $V \subset \Psi^{Nm}_r(W)$ with $a \in V$. By Lemma \ref{lem:bounds_on_Atilde}, we know that $\tilde{A}_r(Q)$ is contained in the region
\[
B_r(Q) \cap \{|s - 1_s| \le \beta \cdot l \cdot \mu^{-\kappa m_r} \} \qquad\text{for $\beta,\kappa > 0$}
\]
It follows from Construction \ref{con:smooth_boxes} that this region is contained in a ball of radius bounded by $r$ around $a$. Thus for small $r$, this region is contained in $V$ and the first claim is proven. 

\vspace{3pt}

For the second claim, we require some preliminary observations. Consider the subsets $\Xi \subset \Gamma$ and  $\Sigma \subset W_{\on{loc}}^s(\Phi;U)$ given by
\[\Xi = 0_s \times [0,t(Q_r)]_t \times 0_u = [Q,Q_r] \qquad\text{and}\qquad \Sigma = D^n_s(5/2) \times [0,t(Q_r)]_t \times 0_u\]
Note that $\tilde{B}_r(Q) \cap \{u = 0,t \ge 0\}$ is contained in $\Sigma$ by Construction \ref{con:smooth_boxes} and (\ref{eq:def_of_SrQ}). Let $\psi_r$ be the flow of $h \cdot \partial_t$ (see Lemma \ref{lem:coordinate_contraction}). Then $\psi_r^j(0) = 0$ for all $j$ and by construction of $m_r$, we know that
\[
\psi_r^j(t(Q_r)) \in [-\delta,L+\delta]_t \quad\text{and}\quad \Psi_r(Q_r) \in \Gamma \cap U\qquad\text{for all }j = 0,\dots,Nm_r
\]
Moreover, since $\Psi_r(\Xi) = 0_s \times [0,\psi_r(t(Q_r))]_t \times 0_u$ and $\Psi_r$ contracts the $s$-coordinate (see Lemma \ref{lem:coordinate_contraction}), this implies that
\[\Psi_r^j(\Xi) \subset \Gamma \cap U\quad\text{and}\quad\Psi_r^j(\Sigma) \subset \{u = 0\} \cap U \qquad \text{for all }j = 0,\dots,Nm_r\]
Now we prove the second claim. By the first claim and the definition of $\tilde{A}_r(Q)$, we know that
\begin{equation} \label{eq:lem:Atilde_stays_in_U:1} \Psi^{-Nm}(\tilde{A}_r(Q)) \subset W \subset U \qquad\text{and}\qquad \Psi^{-Nm}_r(\tilde{A}_r(Q)) \subset \Psi^{Nm_r - Nm}_r(\tilde{B}_r(Q))\end{equation}
From these two inclusions, it follows that $\Psi^{-Nm}_r(\tilde{A}_r(Q))$ is included in the set
\[
V' = \{x \in F^u(\Psi_r;y) \; : \; y \in \Psi^{Nm_r - Nm}_r(\Sigma) \quad\text{and}\quad |u(x)| \le 2\}
\]
Here we choose $W$ in Construction \ref{con:Hamiltonian_H} so that $W$ is contained in $\{|u| \le 2\}$. Finally, we note that
\[
\Psi_r^{-j}(V') \subset U \qquad\text{for all }j = 0,\dots,Nm_r - Nm
\]
Indeed, the intersection of $\Psi_r^{-j}(V')$ with $\{u = 0\}$ is $\Psi^{Nm_r - Nm - j}(\Sigma)$ and the union $V''$ of unstable disks intersecting $\Psi^{Nm_r - Nm - j}(\Sigma)$ with boundary on $\{u = 3/2\}$ is contained in $U$. On the other hand, since $\Psi^{-1}_r$ contracts distances in $U$, $\Psi_r^{-j}(V') \subset V''$. This proves the second claim.  \end{proof}

The following proof is the most difficult step in this section. Here we make crucial use of the specific construction of our contact Hamiltonians to control certain error terms (see Remark \ref{rmk:error_terms}).

\begin{lemma}[Axiom D, Part 2] \label{lem:axiom_d_part_2} For sufficiently small $r$, we have 
\[T\Psi^{Nm_r}(R) = \lambda_r \cdot R + v\qquad\text{over}\quad \Psi^{-Nm_r}(A'_r(Q))\]
Here $\lambda_r \ge r^{-1/2}$ and $|v| \le r^2 \cdot \mu^{-\log(r)/r}$ for small $r$.
\end{lemma}

\begin{proof} Fix an arbitrary point $x$ in $\Psi^{-Nm_r}(\tilde{A}_r(Q))$. We let $n^{\on{st}}_r$ and $n^{\on{mid}}_r$ denote the quantities
\[n_r^{\on{st}} = \on{min}\big\{j \: : \; t(\Psi_r^{Nj}(x)) \ge L/3\big\} \qquad\text{and}\qquad n_r^{\on{mid}} = \on{min}\big\{j - n^{\on{st}}_r \: : \; t(\Psi_r^{Nj}(x)) \ge 2L/3\big\}\]
Finally, let $m$ and $W$ be as in Construction \ref{con:Hamiltonian_G} and (\ref{eqn:choice_of_m}), and let
\[n^{\on{end}}_r = m_r - Nm - n^{\on{st}}_r - n^{\on{mid}}_r\]
By essentially identical analysis to Lemma \ref{lem:growth_of_mr}, we have the following asymptotic behavior for these quantities.
\begin{equation} \label{eq:n_r_formula} n^{\on{st}}_r \sim -2r^{-1} \cdot \log(r) \qquad n^{\on{mid}}_r \sim r^{-1}  \quad\text{and}\quad n^{\on{end}}_r \sim -r^{-1} \cdot \log(r)\end{equation}
Here $\sim$ means the limit of the ratio is $1$ as $r \to 0$. We will analyze the sequence of terms $T\Psi^{Nj}_{r,x}(R)$ starting at the point $x$ in four regimes, corresponding to the following intervals for the index $j$.
\[[0,n^{\on{st}}_r] \qquad [n^{\on{st}}_r,n^{\on{st}}_r + n^{\on{mid}}_r] \qquad [n^{\on{st}}_r + n^{\on{mid}}_r, m_r-Nm] \qquad [m_r - Nm,m-r]\]
We will call these the starting regime, middle regime, ending regime and extra regime. Crucially, by Lemma \ref{lem:Atilde_stays_in_U}, we know that $\Psi^{Nj}_r(x)$ stays in $U$ for the first three periods.

\vspace{3pt}

\noindent {\bf Step 1: Starting Regime.} Start by assuming that $j$ is in the interval $[0,n^{\on{st}}_r]$. In this regime, $t(\Psi_{Nj}(r)) \le L/3$. It follows from (\ref{eq:basic_properties}) and Construction \ref{con:Hamiltonian_H} (see (\ref{eq:form_of_Phi_in_range_13})) that
\begin{equation} \label{eq:lem:axiom_d_part_2:starting} T\Psi^{Nj}_r(R) = (T\Phi^H_r \circ T\Phi)^{Nj}(R) = (T\Phi^H_r)^{Nj}(R) = e^{rNn^{\on{st}}_r} \cdot R\end{equation}

\noindent {\bf Step 2: Middle Regime.} Next, assume that $j$ is in the middle interval $[n^{\on{st}}_r, n^{\on{st}}_r + n^{\on{mid}}_r]$. In this case, it follows from the general form for $\Phi^H_r$ in (\ref{eq:form_of_Phi}) that for any $y$ in $U$ with $t(y) \in [L/3,2L/3]$, we have
\begin{equation} \label{eq:TPhiH_of_R} T\Phi^H_{r,y}(R) = \partial_t\psi_r (t(y)) \cdot R + \partial_tf_r(t(y)) (s(y) \cdot \partial_s)\end{equation}
under the splitting $TU = T\R^n_s \oplus T\R^1_t \oplus T\R^n_u$. Here $\psi$ is given in (\ref{eq:form_of_Phi}) and satisfies
\[
\psi_0(t) = t \qquad \partial_r\psi = h \circ \psi
\]
By the assumption that $|h'| \le 1/L$ on the interval $[0,L]$, we know that $|h| \le 1$ and so
\begin{equation} \label{eq:lem:axiom_d_part_2:1}
|\partial_t \psi_r(t(y))| \ge e^{-r} \qquad \text{for small }r
\end{equation}
Next note that $|\partial_t f_r(t(y))|$ is bounded by some constant $C > 0$ for small $r$. Finally, $\Psi_r^N$ and $\Phi^N$ both preserve $T\R^n_s$ and contract $T\R^n_s$ by a factor of $\mu^{-1}$ for small $r$ (see Lemmas \ref{lem:standard_chart_w_cones} and \ref{lem:partial_hyperbolicity}).

\vspace{3pt}

Now let $x_{\on{st}}$ and $R_{\on{st}}$  denote the point and vector that is left after we exit the early regime.
\[x_{\on{st}} 
= \Psi^{Nn^{\on{st}}_r}_r(x)\qquad\text{and}\qquad R_{\on{st}}= T\Psi_{r,x}^{Nn^{\on{st}}_r}(R) \in \on{span}(R)\]
It follows from the discussion above and (\ref{eq:TPhiH_of_R}) that we can write the following expansion.
\begin{equation} \label{eq:ugly_formula}
T\Psi^{N n^{\on{mid}}_r}_{r,x_{\on{st}}}(R_{\on{st}}) = T(\Phi^H)^{N n^{\on{mid}}_r}_{r,x_{\on{st}}}(R_{\on{st}}) + \sum_{k = 0}^{Nn^{\on{mid}}_r} w_k
\end{equation}
Here the vectors $w_k$ can be written as follows.
\[
w_k = T\Psi_{r}^{Nn^{\on{mid}}_r - k}( c_k \cdot (s(\Psi^k_r(x_{\on{st}})) \cdot \partial_s)) \in T\R^n_s \qquad\text{where}\qquad c_k = \partial_r f_r(t(\Psi^k_r(x_{\on{st}})))
\]
Now we estimate the terms appearing in (\ref{eq:ugly_formula}). First, by (\ref{eq:lem:axiom_d_part_2:1}) we know that
\[
|T(\Phi^H)^{N n^{\on{mid}}_r}_{r,x_{\on{st}}}(R_{\on{st}})| \ge e^{-Nn^{\on{mid}}_r} \cdot |R_{\on{st}}| \qquad\text{and}\qquad 
T(\Phi^H)^{N n^{\on{mid}}_r}_{r,x_{\on{st}}}(R_{\on{st}}) \in \on{span}(R)
\]
Next we estimate the norm of $w_k$. As noted previously, $|c_k| \le C$ for some $C$ independent of $k$ and $r$ small. Since $\Psi^k_r(x_{\on{st}})$ is in the image of $U$ under the $(Nn_r^{\on{st}} + k)$-th power of $\Psi_r$ and $\Psi_r$ uniformly contracts $T\R^n_s$ by a factor of $\mu^{-1}$, we know that the $s$-vector $s(\Psi^k(x_{\on{st}})) \cdot \partial_s$ is bounded as follows.
\[
|s(\Psi^k(x_{\on{st}})) \cdot \partial_s| \le 2 \cdot \mu^{-n^{\on{st}}_r-\lfloor k/N \rfloor}
\]
Finally, $T\Psi^N_r$ uniformly contracts vectors in $T\R^n_s$ by a factor of $\mu^{-1}$. Combining these estimates, we find that for some constant $C' > 0$ independent of $x$ and small $r$, we have
\[
|w_k| \le C' \cdot \mu^{-(n^{\on{st}}_r + n^{\on{mid}}_r)} \qquad\text{and}\qquad \Big|\sum_{k = 0}^{Nn^{\on{mid}}_r} w_k \Big| \le C' \cdot Nn^{\on{mid}}_r \mu^{-(n^{\on{st}}_r + n^{\on{mid}}_r)} \le C'' \cdot \frac{1}{r} \cdot \mu^{-2\log(r)/r}
\]
The outcome of this analysis of the middle regime is the following formula.
\begin{equation} \label{eq:lem:axiom_d_part_2:middle} R_{\on{mid}} = T\Psi_{r,x}^{N(n^{\on{st}}_r + n^{\on{mid}}_r)}(R) = e^{rNn^{\on{st}}_r} \cdot T\Psi_{r,x}^{Nn^{\on{mid}}_r}(R) = A_{\on{mid}} \cdot e^{rNn^{\on{st}}_r} \cdot R + w\end{equation}
Here $A_{\on{mid}}$ is some constant bounded by $e^{-Nn^{\on{mid}}_r}$ and $w$ is a vector in $T\R^n_s$ with norm bounded by the quantity  $C'' \cdot r^{-1} \cdot \mu^{-2\log(r)/r}$. 

\vspace{3pt}

\noindent {\bf Step 3: Ending Regime.} We next examine the ending regime, where $j$ is in the interval $[n^{\on{st}}_r + n^{\on{mid}}_r, m_r-Nm]$. In this regime, $t(\Psi_{Nj}(r)) \ge 2L/3$. By using (\ref{eq:basic_properties}) and Construction \ref{con:Hamiltonian_H} (and specifically (\ref{eq:form_of_Phi_in_range_23})), we see that we have
\[
T\Psi^{Nn^{\on{end}}_r}_r(R_{\on{mid}}) =  A_{\on{mid}} \cdot e^{rNn^{\on{st}}_r} \cdot T\Psi^{Nn^{\on{end}}_r}_r(R) + T\Psi^{Nn^{\on{end}}_r}_r(w) = A_{\on{mid}} \cdot e^{rN(n^{\on{st}}_r - n^{\on{end}}_r)}  R + T\Psi^{Nn^{\on{end}}_r}_r(w)
\]
Focusing on the first term, the lower bound $A_{\on{mid}} \ge e^{-rNn^{\on{mid}}_r}$ and the asymptotic formula (\ref{eq:n_r_formula}) imply that
\[
A_{\on{mid}} \cdot e^{rN(n^{\on{st}}_r - n^{\on{end}}_r)} \ge e^{-r \cdot \log(r)/r} \ge r^{-1/2} \qquad\text{for small }r
\]
Moreover, $\Psi^N$ uniformly contracts $T\R^n_s$ by a factor of $\mu^{-1}$, and we thus see that
\[
w' = T\Psi^{Nn^{\on{end}}_r}_r(w) \qquad\text{satisfies}\qquad |w'| \le N \cdot n^{\on{mid}}_r \cdot \mu^{-(n^{\on{st}}_r + n^{\on{mid}}_r + n^{\on{end}}_r)} \le N \cdot m_r \cdot \mu^{-m_r+Nm}
\]
Combining all of the analysis up to now, we have proven that
\begin{equation} \label{eq:axiom_D_step_2:final}
T\Psi^{N(m_r-m)}_r(R) = \lambda_r \cdot R + v \qquad\text{over }\Psi^{-Nm_r}(A'_r(Q))
\end{equation}
where $\lambda_r \ge r^{-1/2}$ and $v \in T\R^n_s$ satisfies $|v| \le r^2 \cdot \mu^{-\log(r)/r}$ for small $r$.

\vspace{3pt}

\noindent {\bf Step 4: Extra Regime.} Finally, we consider the last regime. By Lemma \ref{lem:Atilde_stays_in_U}, we know that
\[
\Psi^{N(m_r - m)}_r(x) \subset W
\]
By the construction of $\Psi_r$, or more precisely (\ref{eq:basic_properties}), we know that
\[
T\Psi^{Nm}_r = T\Phi^R_r \circ T\Phi^{Nm} \circ T \Phi^H_{Nmr}
 \qquad\text{on }W\]
Note that: $T\Phi^R_r$ preserves $R$ and $T\R^n_s$; $T\Phi^{Nm}$ preserves $R$ and shrinks $T\R^n_s$ by a factor of $\mu^{-1}$; and $T\Phi^H_{Nmr}$ preserves $R$ and expands $T\R^n_s$ by at most a factor of $e^{Nmr}$. It follows that
\[T\Psi^{Nm_r}_r(R) = \lambda_r \cdot R + v\]
where $\lambda_r$ and $v$ satisfy the same estimates as in (\ref{eq:axiom_D_step_2:final}). This concludes the proof.\end{proof}

We are (finally) ready to prove the second cone axiom. Thanks to the onerous work of Lemmas \ref{lem:axiom_d_part_1} and \ref{lem:axiom_d_part_2}, this will be a simple application of Lemma \ref{lem:stretching_Kcu}.

\begin{lemma}[Blender Axiom D] \label{lem:blender_axiom_D} There is a compatible cone-field $K^{cu}$ for $B_r(Q)$ of width less than $\epsilon$ (with respect to the standard metric) that is contracted and dilated uniformly as follows.
\[(\Psi_r^N)_*K^{cu} \subset \on{int} K^{cu}  \quad\text{and}\quad \Psi^N_r \text{ dilates }K^{cu} \qquad\text{ over }\Psi^{-N}(A_r(Q))\]
\[(\Psi^{Nm_r}_r)_*K^{cu} \subset \on{int} K^{cu}  \quad\text{and}\quad \Psi^{Nm_r}_r \text{ dilates }K^{cu} \qquad\text{ over }\Psi_r^{-Nm_r}(A'_r(Q))\]
\end{lemma}

\begin{proof} Fix constants $N,\mu,\epsilon$ as in Lemma \ref{lem:standard_chart_w_cones} such that $\mu > 1 + \epsilon^2$. By Lemma \ref{lem:Atilde_stays_in_U}, we may take $r$ small enough so that $T\Psi^N_r$ dilates $E^u(\Psi_r)$ by the constant $\mu$ and dilates $T\R^n_s$ by $\mu^{-1}$ over $U$.  We take
\[K^{cu} = K^{cu}_{\epsilon,\delta}\]
for a judiciously chosen $\delta$. To choose $\delta$ appropriately, note that by Lemma \ref{lem:axiom_d_part_1}, we have
\[
T\Psi^N_r(R) = e^{Nr} \cdot R \qquad\text{ over the subset }\Psi^{-N}_r(A_r(Q)) \subset U
\]
It follows by Lemma \ref{lem:stretching_Kcu} that $K^{cu}$ is contracted and uniformly dilated by $\Psi^N_r$ as long as
\begin{equation} \label{eq:delta_constraint_1}0 \le \delta \le \sqrt{e^{Nr} - 1} \end{equation}
Likewise, by Lemma \ref{lem:axiom_d_part_2}, we know that for large $r$
\[T\Psi^{Nm_r}(R) = \lambda_r \cdot R + v  \qquad\text{ over the subset }\Psi^{-Nm_r}_r(A_r(Q)) \subset U
\]
where $|\lambda_r| \ge r^{-1/2}$ and $|v| \le r^2 \cdot \mu^{-\log(r)/r}$. It follows by Lemma \ref{lem:stretching_Kcu} that $K^{cu}$ is contracted and uniformly dilated by $ \Psi^{Nm_r}_r$ as long as
\begin{equation} \label{eq:delta_constraint_2}\frac{\mu^{-\log(r)/r}}{1 - \mu^{-1}} \le \delta \le \sqrt{r^{-1} - 1} \le \sqrt{\lambda_r^2 - 1} \end{equation}
On the other hand, for every small $r$ we know that
\[
\frac{\mu^{-\log(r)/r}}{1 - \mu^{-Nm_r}} \le r^2 \le \sqrt{e^{Nr} - 1}
\]
Thus we may choose $\delta$ to satisfy both (\ref{eq:delta_constraint_1}) and (\ref{eq:delta_constraint_2}). The result now follows by Lemma \ref{lem:stretching_Kcu}.\end{proof}

\subsection{Axioms E And F} \label{subsec:axiom_E_F} Finally, we prove the blender axioms regarding vertical disks, Definition \ref{def:blender}(e-f). The first such axiom is relatively straightforward.

\begin{lemma}[Blender Axiom E] \label{lem:blender_axiom_E}  Let $D$ be a vertical disk through $B_r(Q)$ to the right of the local unstable manifold of $Q$ with respect to $\Psi_r$. Then
\[\on{dist}(D,\partial^l B_r(Q)) > r^3 \qquad\text{for small }r\]
In particular, there is a neighborhood $U_-$ of $\partial^l B_r(Q)$ that is disjoint from all such disks $D$. \end{lemma}

\begin{proof} By Lemma \ref{lem:fixed_points}, the local unstable manifold $W$ of $Q$ with respect to $\Psi_r$ in $B_r(Q)$ is given by $D^n_s(2) \times 0_t \times 0_u$. It follows that a vertical disk $D$ to the right of $W$ must intersect the manifold
\[W' = D^n_s(2) \times D^1_t(t(Q_r)) \times 0_u\]
at a point $x$ with $t(x) > 0$ and $u(x) = 0$. By Lemma \ref{lem:vertical_disks_graphs}, $D$ is the graph of a $2\epsilon$-Lipschitz map
\[f:D^n_u(l \cdot \mu^{-\mu_r}) \to D^n_s(2) \times D^1_t(t(Q_r))\]
The image of $f$, or equivalently the projection of $D$ to $D^n_s(2) \times D^1_t(t(Q_r))$, is contained in a ball of radius $2 \epsilon  \cdot l \cdot \mu^{-\mu_r}$ and the $t$-coordinate of $D$ is lower bounded by
\[
t_0-2\epsilon \cdot l \cdot \mu^{-m_r} > -2\epsilon \cdot l \cdot \mu^{-m_r}
\]
By Construction \ref{con:smooth_boxes}, the left boundary $\partial^l B_r(Q)$ consists of points in $B_r(Q)$ with $t$-coordinate $-t(Q_r)$. Finally, note that $m_r$ diverges faster than $1/r$ as $r \to 0$ by Lemma \ref{lem:growth_of_mr}. Therefore
\[\on{dist}(D,\partial^l B_r(Q)) > r \cdot (1 - e^{-8Nr}) -2\epsilon \cdot l \cdot \mu^{-m_r} > r^3 \qquad\text{for small $r$ }\qedhere\]\end{proof}

The final blender axiom is more difficult. We prove the following version.

\begin{lemma}[Blender Axiom F]  \label{lem:blender_axiom_F} Let $D$ be a vertical disk through $B_r(Q)$ to the right of the local unstable manifold $W$ of $Q$ with respect to $\Psi_r$. Consider the intersection point
\[x = (s,t,0_u) = D \cap W' \qquad\text{where}\qquad W' := D^n_s \times D^1_t \times 0_u \]
Then the following two alternatives hold for $D$.
\begin{enumerate}[label=(\alph*)]
    \item If $0 < t(x) \le t(\Psi^{-3N}(Q_r))$ then the intersection $\Psi_r^N(D) \cap A_r(Q)$ contains a disk $\tilde{D}$ through $B_r(Q)$ to the right of $W$, such that
    \[\on{dist}(\tilde{D},\partial^r B_r(Q)) > r^3\]
    \item Otherwise, if $t(\Psi^{-3N}(Q_r)) \le t(x)$ then the intersection $\Psi_r^{Nm_r}(D) \cap A'_r(Q)$ contains a disk $\tilde{D}$ through $B_r(Q)$ to the right of $W$ such that
    \[
    \on{dist}(\tilde{D},W) > r^3
    \]
\end{enumerate}
Thus there are neighborhoods $U_+$ of $\partial^r B_r(Q)$ and $U$ of $W$ such that $\tilde{D}$ is disjoint from either $U$ or $U_+$.\end{lemma}

\begin{proof} We proceed in several steps. First, we address (a) which is easy. Second, we discuss a special case of (b). Third, we discuss the general case of (b).

\vspace{3pt}

\noindent {\bf Step 1.} We prove (a) first. Note that the $t$-coordinate of $\Psi_r^N(x)$ satisfies
\[
0 < t(\Psi^N_r(x)) = e^{Nr} \cdot t(x) \le e^{-2Nr} \cdot t(Q_r)\]
Take the disk $\tilde{D}$ to be the connected component of $\Psi_r^N(D)$ containing $\Psi_r^N(x)$. Then $\tilde{D}$ is a vertical disk to the right of $W$ since it intersects $W'$ at a point with positive $t$-coordinate. By Lemma \ref{lem:vertical_disks_graphs}, it is the graph of a $2\epsilon$-Lipschitz map
\[f:D^n_u(l \cdot \mu^{-m_r}) \to D^n_s(2) \times D^1_t(t(Q_r))\]
As in Lemma \ref{lem:blender_axiom_E}, this implies that the $t$-coordinate is bounded by
\[e^{-2Nr} \cdot t(Q_r) + 2\epsilon \cdot l \cdot \mu^{-m_r} < e^{-Nr} \cdot t(Q_r) \qquad\text{for small }r\]
The right side $\partial^r B_r(Q)$ is the set of points with $t$-coordinate $t(Q_r)$, by Definition \ref{con:smooth_boxes}. Therefore
\[\on{dist}(\partial^r B_r(Q),\tilde{D}) > t(Q_r) - e^{-Nr} \cdot t(Q_r) =  e^{-Nr} \cdot (1 - e^{-Nr}) \cdot (1 - e^{-8Nr}) > r^3\qquad\text{for small }r\]

\vspace{3pt}

\noindent {\bf Step 2.} In this step, we consider a useful special case of (b). Let $C'_r \subset W'$ denote the set
\[
C'_r = D^n_s(2+\mu^{-m_r/2}) \times [t(\Psi^{-4N}_r(Q_r)),t(\Psi^{N}_r(Q_r))]_t \times 0_u
\]
Given a point $w \in C'_r$, we let $D_w$ denote the following vertical disk
\[
D_w = \big\{ (s,t,u) \in F^u(\Psi_r,w) \; : \; |u| \le l \cdot \mu^{-m_r}\big\} 
\]
Note that by Lemmas \ref{lem:coordinate_contraction} and \ref{lem:coordinate_projections}, we have
\[
\Psi^{Nm_r}_r(C'_r) \subset D^n_s(2 \cdot \mu^{-m_r}) \times [t(\Psi^{N(m_r-4)}_r(Q_r)),t(\Psi^{N(m_r+1)}_r(Q_r))]\]
By the construction of $m_r$ (see Notation \ref{not:m_r}) we know that
\[t(\Psi_r^{N(m_r-4)}(Q_r)) \ge t(\Psi_r^{N}(P_r)) \qquad\text{and}\qquad t(\Psi_r^{N(m_r+1)}(Q_r)) \ge t(\Psi_r^{7N}(P_r))\]
Therefore we have the following inclusion
\[
\Psi^{Nm_r}_r(C'_r) \subset D^n_s(2 \cdot \mu^{-m_r}) \times [t(\Psi^{N}_r(P_r)),t(\Psi^{7}_r(P_r))]\]
For sufficiently small $r$, Construction \ref{con:useful_holonomy} yields a well-defined holonomy map 
\[
\on{Hol}_{\Psi_r}:D^n_s(2 \cdot \mu^{-m_r}) \times [t(\Psi^{N}_r(P_r)),t(\Psi^{7N}_r(P_r))] \to D^n_s(3) \times [-\delta,L+\delta]_t \times 0_u 
\]
Thus the point $z = \Psi^{Nm_r}_r(w)$ has well-defined holononomy. By Lemma \ref{lem:holonomy_estimates}, we have
\[|\on{Hol}_{\Psi_r}(z) - (r + t(z) - L)| \le \mu^{-\kappa m_r}\]
In particular, this implies that for small $r$, we have the following inequality.
\begin{equation} \label{eq:axiom_F_proof:1} t(\on{Hol}_{\Psi_r}(z)) \ge r + t(\Psi^N_r(P_r)) - L - \mu^{-\kappa m_r} = r(1 - e^{-Nr}) - \mu^{-\kappa m_r} > \frac{1}{2}r^2\end{equation}
Moreover, $\on{Hol}_{\Psi_r}(z)$ lies on the unstable disk $\Psi^{Nm_r}(D_w)$ through $\Psi_r^{Nm_r}(w)$, since it contains all points in the unstable leaf of distance less than $l$ from $\Psi_r^{Nm_r}(w)$. We thus acquire a point
\[
x_w = \Psi_r^{-Nm_r}(\on{Hol}_{\Psi_r}(z)) \quad\text{with}\quad \Psi^{Nm_r}_r(x_w) \in \tilde{A}_r(Q) \cap \{u = 0\} \text{ and }t(\Psi^{Nm_r}_r(x_w)) > \frac{1}{2}r^2
\]
Note that $x_w$ varies continuously with $w$.

\vspace{3pt}

\noindent {\bf Step 3.} In this step we discuss the general case of (b). Fix a vertical disk $D \subset B_r(Q)$ as in (b), with $t(\Psi^{-3N}(Q_r)) < t(x)$ where $x = D \cap W'$. Let
\[\pi:U \to [-3,3]^n_s \times [-\epsilon,L+\epsilon]_t \times 0_u\]
be projection to the $(s,t)$-plane and let $\Sigma \subset W'$ be the union of all projections $\pi(D_w)$ where $D_w$ intersects $\pi(D)$. The disks $D$ and $D_w$ are $2\epsilon$-Lipschitz graphs over $D^n_u(l \cdot \mu^{-m_r})$ (Lemma \ref{lem:vertical_disks_graphs}). Therefore these projections are all contained in balls of radius $2 \epsilon \cdot l \cdot \mu^{-m_r}$ in $W'$ (see also Step 2). It follows that there is a $C > 0$ independent of $r$ such that
\[
\on{dist}(x,y) \le C \cdot \mu^{-m_r} \qquad\text{for all }y \in \Sigma
\]
This implies that $\Sigma \subset C'_r$ for sufficiently small $r$, and therefore that
\[
D \subset V \qquad\text{where}\qquad V = \{z \in D_w \;: \; w \in C'_r\}
\]
since $\Sigma'$ contains the $(s,t)$-ball of radius $C \cdot \mu^{-m_r}$ around $x$ for small $r$. We let $\pi':V \to C'_r$ be the obvious projection mapping $z \in D_w$ to $w$. 

\vspace{3pt}

By Step 2, the projection $\pi':V \to C'_r$ has a natural continuous section
\[
\sigma:C'_r \to V \qquad\text{given by}\qquad \sigma(w) = x_w
\]
Since $D$ is vertical, it must necessarily intersect one point $x \in D$ in the image of $\sigma$. By Step 2
\[
t(\Psi^{Nm_r}(x)) > \frac{1}{2}r^2 \qquad\text{and}\qquad \Psi^{Nm_r}(x) \in \tilde{A}_r(Q) \cap W'
\]
Note that $\tilde{A}_r(Q) \subset B_r(Q)$ for small $r$. We let $\tilde{D}$ be the component of $\Psi_r^{Nm_r}(D) \cap B_r(Q)$ containing $x$. This is a vertical disk containing $x$, so by the usual considerations we have
\[
\on{dist}(x,z) < C \cdot \mu^{-m_r} \qquad\text{for all }z \in \tilde{D}
\]
It follows that for sufficiently small $r$, we have the lower bound
\[
\on{dist}(\tilde{D},W) \ge \on{dist}(x,W) - C\mu^{-m_r} = t(x) - C\mu^{-m_r} > \frac{1}{2}r^2  - C\mu^{-m_r} > r^3
\]
This constructs the required disk and concludes the proof.\end{proof}

\subsection{Proof Of Theorem \ref{thm:heteroclinic_contact_blender}} \label{subsec:conclusion_of_thm} In this final part, we provide the proof of Theorem \ref{thm:heteroclinic_contact_blender}. We need a final lemma, demonstrating Theorem \ref{thm:heteroclinic_contact_blender}(c).

\begin{lemma} \label{lem:heteroclinic_blender_c} The intersection $W^u(P,\Psi_r) \cap B_r(Q)$ contains a vertical disk $D_Q$ to the right of $W^s(Q,\Psi_r)$.
\end{lemma}

\begin{proof} We demonstrate this for $B_r(Q)$. By Lemma \ref{lem:homoclinic_points}, $W^u(P,\Psi_r)$ contains a disk $D_r \subset U$ in $F^u(\Psi_r)$ centered at the homoclinic point
\[b_r = (1_s,r,0_u) \quad\text{of}\quad P\]
Let $c_r = (0_s,r,0_u)$ be the intersection point $F^s(\Psi_r,b_r) \cap \Gamma$ and let $D'_r$ be a disk in $F^u(\Psi_r) \cap U$ centered at $c_r$. Note that we may take the radius of $D'_r$ and $D_r$ to be bounded below by $A > 0$ independent of $r$. Let $l_r$ denote the unique integer such that
\[
\Psi_r^{l_r}(c_r) \in (\Psi_r(P_r),\Psi_r^2(P_r)] \quad\text{or equivalently}\quad t(\Psi_r^{l_r}(c_r)) \in (L - re^{-r}, L - re^{-2r}]  
\]
As in Lemma \ref{lem:growth_of_mr}, we know that $l_r > 1/r$ if $r$ is small. Now note that $\Psi_r^N$ uniformly expands $F^u(\Psi_r)$ with constant of dilation (greater than) $\mu$, for small $r$. Therefore $\Psi_r^{l_r}(D'_r)$ contains the disk in $F^u(\Psi_r)$ of radius greater than $2 \cdot l$ around $\Psi^{l_r}_r(c_r)$ for small $r$. In particular, by Lemma \ref{lem:homoclinic_points} and the definition of $l$ (Construction \ref{con:smooth_boxes}), there is a sub-disk
\[D''_r \subset \Psi^{l_r}(D'_r) \qquad\text{with a point $x = (1_s,t(x),0_u)$ with $t(x) \in [r(1 - e^{-r}), r(1 - e^{-2r})]$}\]
Thus the disk is to the right of $W^s_{\on{loc}}(Q,\Psi_r) \cap B_r(Q)$. The disk must also be contained in a unstable disk fiber of $\tilde{B}_r(Q)$ (see Construction \ref{con:smooth_boxes}). Therefore every point in $D''_r$ is within distance $2\epsilon \cdot l \cdot \mu^{-m_r}$ of $x$ and so
\[
t(y) \le t(x) + 2\epsilon \cdot l \cdot \mu^{-m_r} \le r(1 - e^{-3r}) \le 4r^2 \qquad\text{for every $y \in D''_r$ and small $r$}
\]
Similarly, $t(y) \ge r^2$. Finally, since $b_r$ and $c_r$ are on the same stable disk in $U$, we have
\[
\on{dist}(\Psi^{m_r}_r(c_r),\Psi^{m_r}_r(b_r)) \le \mu^{-m_r} 
\le \mu^{-1/r}
\]
By taking a path $\gamma$ from $\Psi^{m_r}_r(c_r)$ to $x$ in $F^u(\Psi_r)$ and using the holonomy map of $F^u$ (see Construction \ref{con:useful_holonomy}) we get a small unstable disk $D_Q$ in $F^u(\Psi_r)$ centered at a point $x_Q = \on{Hol}_{\Psi_r}(x)$ with
\[\on{dist}(x_Q,x) \le \mu^{-\kappa/r}\]
Here $\beta$ is the uniform H\"{o}lder constants in Construction \ref{con:useful_holonomy}. It follows, as with $D''$, that $D_Q$ is a vertical disk to the right of $W^s_{\on{loc}}(Q,\Psi_r) \cap B_r(Q)$.\end{proof}

\begin{proof} (Theorem \ref{thm:heteroclinic_contact_blender}) The existence theorem for the contact blender is now an immediate consequence of Lemma \ref{lem:fixed_points} for Theorem \ref{thm:heteroclinic_contact_blender}(a), the Lemmas \ref{lem:blender_axiom_A}, \ref{lem:blender_axiom_B}, \ref{lem:blender_axiom_C}, \ref{lem:blender_axiom_D}, \ref{lem:blender_axiom_E} and \ref{lem:blender_axiom_F} for Theorem \ref{thm:heteroclinic_contact_blender}(b) and the Lemma \ref{lem:heteroclinic_blender_c} for Theorem \ref{thm:heteroclinic_contact_blender}(c). \end{proof}

\bibliographystyle{hplain}
\bibliography{standard_bib}

\end{document}